\documentclass[a4paper,reqno]{amsart}

\usepackage{amssymb}
\usepackage{amstext}
\usepackage{amsmath}
\usepackage{amscd}
\usepackage{amsthm}
\usepackage{amsfonts}
\usepackage{enumerate}
\usepackage{graphicx}
\usepackage{latexsym}
\usepackage{mathrsfs}
\usepackage[all,color,dvips]{xy}
\usepackage{mathtools}
\xyoption{all}

\usepackage{pstricks}
\usepackage{lscape}
\usepackage{comment}

\newtheorem{theorem}{Theorem}[section]

\newtheorem{corollary}[theorem]{Corollary}
\newtheorem{lemma}[theorem]{Lemma}
\newtheorem{proposition}[theorem]{Proposition}
\newtheorem{definition-proposition}[theorem]{Definition-Proposition}
\newtheorem{question}[theorem]{Question}
\newtheorem{conjecture}[theorem]{Conjecture}

\theoremstyle{definition}
\newtheorem{definition}[theorem]{Definition}
\newtheorem{remark}[theorem]{Remark}
\newtheorem{example}[theorem]{Example}
\newtheorem{observation}[theorem]{Observation}
\newtheorem{construction}[theorem]{Construction}

\newcommand{\mm}{{\mathfrak{m}}}
\newcommand{\nn}{{\mathfrak{n}}}
\newcommand{\pp}{{\mathfrak{p}}}
\newcommand{\qq}{{\mathfrak{q}}}
\renewcommand{\AA}{\mathscr{A}}
\newcommand{\CC}{\mathscr{C}}
\newcommand{\CCC}{\mathsf{C}}
\newcommand{\DDD}{\mathsf{D}}
\newcommand{\KKK}{\mathsf{K}}
\newcommand{\OO}{{\mathcal O}}

\newcommand{\PP}{\mathscr{P}}
\newcommand{\II}{\mathscr{I}}

\newcommand{\RR}{\mathscr{R}}
\newcommand{\TT}{\mathscr{T}}
\newcommand{\UU}{\mathscr{U}}
\newcommand{\XX}{\mathscr{X}}

\newcommand{\Z}{\mathbb{Z}}
\newcommand{\R}{\mathbb{R}}

\newcommand{\bo}{\operatorname{b}\nolimits}
\newcommand{\depth}{\operatorname{depth}\nolimits}
\newcommand{\pd}{\operatorname{proj.dim}\nolimits}

\newcommand{\Ext}{\operatorname{Ext}\nolimits}
\newcommand{\Tor}{\operatorname{Tor}\nolimits}
\newcommand{\Hom}{\operatorname{Hom}\nolimits}
\newcommand{\End}{\operatorname{End}\nolimits}

\newcommand{\gl}{\operatorname{gl.\!dim}\nolimits}
\newcommand{\op}{\operatorname{op}\nolimits}
\newcommand{\RHom}{\mathbf{R}\strut\kern-.2em\operatorname{Hom}\nolimits}
\newcommand{\Lotimes}{\mathop{\stackrel{\mathbf{L}}{\otimes}}\nolimits}

\newcommand{\Kernel}{\operatorname{Ker}\nolimits}
\newcommand{\Cokernel}{\operatorname{Cok}\nolimits}
\newcommand{\Spec}{\operatorname{Spec}\nolimits}
\newcommand{\Supp}{\operatorname{Supp}\nolimits}
\newcommand{\MaxSpec}{\operatorname{MaxSpec}\nolimits}
\newcommand{\Ass}{\operatorname{Ass}\nolimits}
\newcommand{\ann}{\operatorname{ann}\nolimits}
\newcommand{\Proj}{\operatorname{Proj}\nolimits}
\newcommand{\MaxProj}{\operatorname{MaxProj}\nolimits}
\newcommand{\sign}{\operatorname{sgn}\nolimits}
\newcommand{\GL}{\operatorname{GL}\nolimits}
\newcommand{\SL}{\operatorname{SL}\nolimits}

\newcommand{\Aff}{\operatorname{Aff}\nolimits}

\newcommand{\Ho}{{\rm H}}

\DeclareMathOperator{\moduleCategory}{\mathsf{mod}} \renewcommand{\mod}{\moduleCategory}
\DeclareMathOperator{\Mod}{\mathsf{Mod}}
\DeclareMathOperator{\proj}{\mathsf{proj}}
\DeclareMathOperator{\inj}{\mathsf{inj}}
\DeclareMathOperator{\ind}{\mathsf{ind}}
\DeclareMathOperator{\Sub}{\mathsf{Sub}}
\DeclareMathOperator{\coh}{\mathsf{coh}}

\DeclareMathOperator{\Gr}{\mathsf{Gr}}
\DeclareMathOperator{\gr}{\mathsf{gr}}
\DeclareMathOperator{\qGr}{\mathsf{qGr}}
\DeclareMathOperator{\qgr}{\mathsf{qgr}}
\DeclareMathOperator{\add}{\mathsf{add}}

\DeclareMathOperator{\fd}{\mathsf{fd}}

\DeclareMathOperator{\repdim}{repdim}

\newcommand{\cut}{\ar@{-}@[|(5)]}

\newcommand{\iso}{\cong}
\newcommand{\equi}{\approx}

\numberwithin{equation}{section}

\begin{document}
\title{$n$-representation infinite algebras}
\author[Herschend]{Martin Herschend}
\address{M. Herschend: Graduate School of Mathematics, Nagoya University, Chikusa-ku, Nagoya, 464-8602 Japan}
\email{martinh@math.nagoya-u.ac.jp}
\thanks{The first author was partially supported by JSPS Grant-in-Aid for Young Scientists 24740010.}

\author[Iyama]{Osamu Iyama}
\address{O. Iyama: Graduate School of Mathematics, Nagoya University, Chikusa-ku, Nagoya, 464-8602 Japan}
\email{iyama@math.nagoya-u.ac.jp}
\urladdr{http://www.math.nagoya-u.ac.jp/~iyama/}
\thanks{The second author was partially supported by JSPS Grant-in-Aid for Scientific Research 24340004, 23540045, 20244001 and 22224001.}

\author[Oppermann]{Steffen Oppermann}
\address{S. Oppermann: Institutt for matematiske fag, NTNU, 7491 Trondheim, Norway}
\email{steffen.oppermann@math.ntnu.no}


\begin{abstract}
From the viewpoint of higher dimensional Auslander-Reiten theory, we introduce a new class of finite dimensional algebras of global dimension $n$, which we call $n$-representation infinite.
They are a certain analog of representation infinite hereditary algebras, and we study three important classes of modules: $n$-preprojective, $n$-preinjective and $n$-regular modules.
We observe that their homological behaviour is quite interesting. For instance they provide first examples of algebras having infinite $\Ext^1$-orthogonal families of modules. Moreover we give general constructions of $n$-representation infinite algebras.

Applying Minamoto's theory on Fano algebras in non-commutative algebraic geometry, we describe the category of $n$-regular modules in terms of the corresponding preprojective algebra.
Then we introduce $n$-representation tame algebras, and show that the category of $n$-regular modules decomposes into the categories of finite dimensional modules over localizations of the preprojective algebra.
This generalizes the classical description of regular modules over tame hereditary algebras.
As an application, we show that the representation dimension of an $n$-representation tame algebra is at least $n+2$.
\end{abstract}

\maketitle
\tableofcontents

\section{Introduction}

The notion of global dimension gives an important measure in representation theory of algebras: Algebras of global dimension zero are semisimple, and their representation theory is trivial in the sense that any module is a direct sum of simple modules.
Algebras of global dimension one are path algebras of quivers, and their representation theory has been one of the central subjects in modern representation theory.
Unfortunately it seems to be quite hard to develop general theory for algebras of higher global dimension though there are a number of important classes for which more is known.
This means that we need to restrict our consideration to some special classes to get a fruitful theory.

From the viewpoint of higher dimensional Auslander-Reiten theory, a distinguished class of finite dimensional algebras of global dimension $n$,
called \emph{$n$-representation finite algebras}, has been studied \cite{HI1,HI2,I1,I4,IO1,IO2}.
They have an $n$-cluster tilting subcategory with an additive generator, whose structure is controled by the higher Auslander-Reiten translations $\tau_n\colon\mod\Lambda\to\mod\Lambda$ and $\nu_n \colon \DDD^{\rm b}(\mod \Lambda) \to \DDD^{\rm b}(\mod \Lambda)$.
They are characterized as follows:
An algebra of global dimension $n$ is $n$-representation finite if and only if for any indecomposable projective module $P$, there exists $\ell_P \ge 0$ such that $\nu_n^{-\ell_P}(P)$ is indecomposable injective \cite[Theorem 3.1]{IO2}.
Thus for $n=1$ the notion of $1$-representation finite algebras coincides with the classical notion of representation finite hereditary algebras.

Since in classical theory one has representation infinite algebras as natural counterpart to representation finite ones, it is natural to ask ``What are $n$-representation infinite algebras?''

The aim of this paper is to introduce $n$-representation infinite algebras and study their properties.
Our definition is a simpler analogue of the above property of $n$-representation finite algebras and given in terms of a higher Auslander-Reiten translation:
A finite dimensional algebra of global dimension $n$ is called
\emph{$n$-representation infinite} if and only if $\nu_n^{-i}(\Lambda)$ is a module (i.e.\ concentrated in degree $0$) for any $i \ge 0$.
Thus for $n=1$ our notion of $1$-representation infinite algebras coincides with the classical notion of representation infinite hereditary algebras.
We show (in Theorem~\ref{dichotomy}) that $n$-representation finite algebras and $n$-representation infinite algebras give two disjoint classes of \emph{$n$-hereditary algebras}, which are certain homologically nice algebras.

As a first example of an $n$-representation infinite algebra consider the Beilinson algebra, given by the quiver
\[ \xymatrix{
1  \ar@/^1pc/[r]^{a^1_0}_{\scalebox{0.7}{\vdots}}\ar@/_1pc/[r]_{a^1_{n}} & 2 \ar@/^1pc/[r]^{a^2_0}_{\scalebox{0.7}{\vdots}}\ar@/_1pc/[r]_{a^2_{n}} & 3} 
\cdots\xymatrix{
n \ar@/^1pc/[r]^{a^{n}_0}_{\scalebox{0.7}{\vdots}}\ar@/_1pc/[r]_{a^{n}_{n}}& n+1}
\text{, subject to the relations }
a^k_ia^{k+1}_j = a^k_ja^{k+1}_i.
\]
By Beilinson \cite{Be} this is the endomorphism algebra of the tilting bundle $\bigoplus_{\ell=0}^n \mathcal{O}(\ell)$ on $\mathbb{P}^n$. Using knowledge on the category $\coh \mathbb{P}^n$, in particular Serre duality and sheaf cohomology, it is easy to see that the Beilinson algebra is $n$-representation infinite.
For details, see Example~\ref{ex.beilinson}. We see (in Section~\ref{sect.Atilde}) that the Beilinson algebra is actually only one member of a large class of $n$-representation infinite algebras which we call ``type $\widetilde{A}$''.

For $n$-representation infinite algebras, we introduce three distinguished classes of modules:
The first one is $\PP=\add \{ \nu_n^{-i}(\Lambda) \mid i \ge0\}$, which we call the \emph{$n$-preprojective modules}, and the second one is
$\II = \add \{ \nu_n^i(D\Lambda) \mid i \ge0\}$
which we call the \emph{$n$-preinjective modules}.
The third class which we call \emph{$n$-regular modules} is described as $\RR = \{X \in \mod \Lambda \mid \forall i \in\Z \colon \nu_n^i(X) \in \mod \Lambda \}$ (see Proposition~\ref{nakayama R}).
In the context of higher dimensional Auslander-Reiten theory, the category
\[\CC^{0}:=\{X\in\mod\Lambda\mid\forall i \in \{1, \ldots, n-1\}\colon\Ext^i_\Lambda(\PP\vee\II,X)=0\}\]
gives a higher analogue of module categories. We give the following result, showing that the properties of these subcategories are very similar to the classical hereditary situation.
\begin{theorem} [see Theorems~\ref{trichotomy} and \ref{n-almost split}]
Let $\Lambda$ be an $n$-representation infinite algebra.
\begin{itemize}
\item[(a)] We have $\CC^{0}=\PP \vee \RR \vee \II$.
\item[(b)] We have
\begin{eqnarray*}
&\Hom_{\Lambda}(\RR, \PP) = 0, \quad \Hom_{\Lambda}(\II, \PP) = 0 \text{, and } \Hom_{\Lambda}(\II, \RR) = 0,& \\
&\Ext_{\Lambda}^n(\PP, \RR) = 0, \quad \Ext_{\Lambda}^n(\PP, \II) = 0 \text{, and } \Ext_{\Lambda}^n(\RR, \II) = 0.&
\end{eqnarray*}
\item[(c)] For any indecomposable non-projective $X\in\PP\vee\II$ (respectively, non-injective $Y\in\PP\vee\II$), there exists an $n$-almost split sequence
\[0\to Y\xrightarrow{}C_{n-1}\xrightarrow{}C_{n-2}\xrightarrow{}\cdots\xrightarrow{}C_1\xrightarrow{}C_0\xrightarrow{}X\to 0\]
such that $Y\iso\tau_n(X)$ and $X\iso\tau_n^-(Y)$.
\end{itemize}
\end{theorem}

It is worth noting that $n$-representation infinite algebras form an answer to the following question in homological algebra, which was asked in \cite{I2} and at ICRA XI:

\begin{question} \label{ICRA XI}
Is there a finite dimensional algebra $\Lambda$ with an infinite set $\{X_i\}_{i\in I}$ of isomorphism classes of indecomposable $\Lambda$-modules
such that $\Ext^1_\Lambda(X_i,X_j)=0$ for any $i,j\in I$?
\end{question}

We see (in Proposition~\ref{P and I}), that for any $n$-representation infinite algebra with $n > 1$, both the categories $\mathscr{P}$ of $n$-preprojective and the category $\mathscr{I}$ of $n$-preinjective modules provide examples of the form asked for by Question~\ref{ICRA XI}.

As in the classical case $n=1$ and the case of $n$-representation finite algebras, one main method for studying an $n$-representation infinite algebra $\Lambda$ is to look at the associated \emph{preprojective algebra} $\Pi$ . 
The name is explained by the fact that $\Pi$ is the direct sum of all $n$-preprojective modules as a $\Lambda$-modules.
By a result of Keller \cite{K2} it is a bimodule $(n+1)$-Calabi-Yau algebra \cite{G}. Using results in \cite{AIR,MM}, we have a bijection between $n$-representation infinite algebras and bimodule $(n+1)$-Calabi-Yau algebras of Gorenstein parameter $1$ (Theorem~\ref{preprojective correspondence}).
Moreover we apply methods in non-commutative algebraic geometry to preprojective algebras to study the category of $n$-regular modules.
In particular, we use Minamoto's theory \cite{M} on Fano algebras to give the following description of $n$-regular modules:

\begin{theorem}[see Theorem~\ref{description of R}]
Let $\Lambda$ be $n$-representation infinite algebra and $\Pi$ the associated preprojective algebra. If $\Pi$ is left graded coherent, then we have
\[ \mathscr{R} = \qgr_0 \Pi\]
where $\qgr\Pi$ is the quotient category of graded $\Pi$-modules modulo finite dimensional modules and $\qgr_0 \Pi$ is the full subcategories consisting of graded $\Pi$-modules of dimension one.
\end{theorem}

One aspect in this paper that heavily relies on the interplay of $n$-representation infinite algebras and their preprojective algebras is our study of \emph{$n$-representation tame} algebras.
We call an $n$-representation infinite algebra $n$-representation tame if its preprojective algebra is a Noetherian algebra. For $n=1$ our notion of $1$-representation tame algebras coincides with the classical notion of representation tame hereditary algebras.
In this situation, we obtain a nice decomposition of the category of $n$-regular modules, which generalizes the classical result for tame hereditary algebras.

\begin{theorem}[see Theorem~\ref{theorem.R_decomp_tame}]
Let $\Lambda$ be $n$-representation tame and $\Pi$ the associated preprojective algebra. Then the category of $n$-regular modules decomposes as
\[ \mathscr{R} = \coprod_{\pp \in \MaxProj R} \fd \Pi_{(\pp)}\]
where $\Pi_{(\pp)}$ is given by a localization of $\Pi$ with respect to $\pp \in \MaxProj R$.
\end{theorem}

For instance in the case of the Beilinson algebra we obtain the decomposition $\mathscr{R} \equi \coprod \fd K[[t_1, \ldots, t_n]]$, where the coproduct runs over all closed points in $\mathbb{P}^n$ (Example~\ref{example of description of R}(b)).

\medskip
It is somewhat surprising that Auslander's representation dimension \cite{A2} may also be applied to our higher dimensional representation theory:
By a classical result, an algebra has representation dimension at most two if and only if it is representation finite.
In particular any representation infinite algebra has representation dimension at least three.
The first example of algebras with representation dimension at least four was given by Rouquier~\cite{Ro}.
Here we show that, at least if we restrict to the finite/tame situation, representation dimension also determines if an algebra is $n$-representation finite. More precisely, we show the following:
\begin{theorem}(Propositon~\ref{repdim of finite}, Theorem~\ref{main})
\[ \repdim \Lambda \begin{cases} \leq n+1 & \text{ if $\Lambda$ is $n$-representation finite,} \\ \geq n+2 & \text{ if $\Lambda$ is $n$-representation tame.} \end{cases} \]
\end{theorem}
We conjecture that any $n$-representation infinite algebra $\Lambda$ satisfies $\repdim\Lambda\ge n+2$ (Conjecture~\ref{conj.repdim}). We prove that this is true for $n=2$ (Theorem~\ref{at least 4}). One reason to believe this is that $n$-representation tame algebras are simplest among $n$-representation infinite algebras.

Results in this paper were presented in Oberwolfach (February 2011) \cite{I5}, Balestrand (June 2011) and Shanghai (October 2011).

\medskip\noindent
{\bf Acknowledgements.}
The authors thank Hiroyuki Minamoto for a valuable suggestion which improved our original Theorem~\ref{theorem.R_decomp_tame} by using Lemma~\ref{nice grading} and Proposition~\ref{new proof 1}.
We are grateful to him and Izuru Mori for stimulating discussion about $n$-representation infinite algebras and Fano algebras. The second author thanks Osamu Fujino, Susan Sierra and Hokuto Uehara for stimulating discussions.

\subsection{Notation}\label{notation}
For general background in representation theory, we refer to \cite{ARS,ASS}.

Throughout this paper $K$ denotes a field. We denote by $D$ the $K$-dual, that is $D(-) = \Hom_K(-,K)$.
By the composition $fg$ of morphisms means first $f$, then $g$.

Let $\Lambda$ be a $K$-algebra. All modules in this paper are left modules. 
We denote by $\Mod\Lambda$ the category of $\Lambda$-modules, by $\mod\Lambda$ the category of finitely presented $\Lambda$-modules, 
by $\proj\Lambda$ the category of finitely generated projective $\Lambda$-modules, by $\inj\Lambda$ the category of finitely generated injective $\Lambda$-modules, and by $\fd\Lambda$ the category of finite dimensional $\Lambda$-modules.
If $\Lambda$ is a graded $K$-algebra we write $\Gr \Lambda$ for the category of graded $\Lambda$-modules, $\gr \Lambda$ for the category of finitely presented graded $\Lambda$-modules, and $\gr \proj \Lambda$ for the category of finitely generated graded projective $\Lambda$-modules.

We denote by $\CCC(-)$, $\KKK(-)$, and $\DDD(-)$ the category of complexes, the homotopy category, and the derived category, respectively. By $\CCC^{\rm b}(-)$, $\KKK^{\rm b}(-)$ and $\DDD^{\rm b}(-)$ we mean the bounded version. 

For a class $\XX$ of objects in an additive category $\CC$, we denote by $\add_{\CC}\XX$ or $\add\XX$ the full subcategory of $\CC$ consisting of direct summands of finite direct sums of objects in $\XX$.
For additive categories $\CC_i$ ($i\in I$) we write $\coprod_{i \in I} \CC_i$ for the coproduct of the categories.
For a Krull-Schmidt category $\CC$ we denote by $\ind \CC$ the class of isomorphism classes of indecomposable objects in $\CC$.
For full subcategories $\CC_i$ ($i\in I$) of $\CC$, we denote by $\bigvee_{i \in I}\CC_i$ the full subcategory of $\CC$ satisfying $\bigvee_{i \in I}\CC_i=\add(\bigcup_{i\in I}\ind\CC_i)$.
Note that a decomposition $\CC=\bigvee_{i \in I} \CC_i$ is a coproduct if and only if $\Hom_{\CC}(\CC_i, \CC_j) = 0 \; \forall i \neq j$.

\section{$n$-representation infinite algebras}

Let $n$ be a positive integer and $\Lambda$ a ring-indecomposable finite dimensional $K$-algebra. Throughout this section we assume that $\Lambda$ has global dimension at most $n$.

Let us recall the Nakayama functor on $\DDD^{\rm b}(\mod \Lambda)$ following Happel \cite{Ha}.
Define the functors
\[
\nu:= D\Hom_\Lambda(-, \Lambda) \colon \mod\Lambda \to \mod\Lambda \ \mbox{ and }\ 
\nu^{-}:= \Hom_{\Lambda^{\op}}(D-, \Lambda) \colon \mod\Lambda \to \mod\Lambda.
\]
Since $\proj\Lambda = \add \Lambda$ and $\inj\Lambda = \add D \Lambda$, they induce quasi-inverse equivalences
\[
\xymatrix{\proj\Lambda \ar@<.5ex>[r]^{\nu}& \inj\Lambda,\ar@<.5ex>[l]^{\nu^{-}}}
\]
which in turn induce triangle equivalences between homotopy categories
\[
\xymatrix{\KKK^{\rm b}(\proj\Lambda) \ar@<.5ex>[r]^{\nu}& \KKK^{\rm b}(\inj\Lambda).\ar@<.5ex>[l]^{\nu^{-}}}
\]
Since $\Lambda$ has finite global dimension, the inclusions $\KKK^{\rm b}(\proj\Lambda) \to \DDD^{\rm b}(\mod \Lambda)$ and
$\KKK^{\rm b}(\inj\Lambda)\to \DDD^{\rm b}(\mod \Lambda)$ are triangle equivalences. Thus we obtain the Nakayama functors
\begin{eqnarray*}
\nu&:=&D\RHom_\Lambda(-,\Lambda) \colon \DDD^{\rm b}(\mod \Lambda) \to \DDD^{\rm b}(\mod \Lambda),\\
\nu^{-1}&:=&\RHom_{\Lambda^{\op}}(D-, \Lambda) \colon \DDD^{\rm b}(\mod \Lambda) \to \DDD^{\rm b}(\mod \Lambda).
\end{eqnarray*}
They are quasi-inverse each other. It is easy to check
$\nu \iso (D\Lambda)\Lotimes_\Lambda-$ and $\nu^{-1}\iso\RHom_\Lambda(D\Lambda,-)$.
Moreover $\nu$ gives a Serre functor \cite{BK} of $\DDD^{\rm b}(\mod \Lambda)$ in the sense that there exists a functorial isomorphism
\[\Hom_{\DDD^{\rm b}(\mod \Lambda)}(X,Y)\iso D\Hom_{\DDD^{\rm b}(\mod \Lambda)}(Y,\nu(X)).\]

In this paper an important role is played by the autoequivalence
\[\nu_n:=\nu\circ[-n] \colon \DDD^{\rm b}(\mod \Lambda)\to\DDD^{\rm b}(\mod \Lambda).\]
We collect some immediate properties of $\nu_n$.

\begin{observation}\label{easy properties of nu_n}
\begin{itemize}
\item[(a)] For any $i\in\Z$, we have the following functorial isomorphism:
\[
\Hom_{\DDD^{\rm b}(\mod \Lambda)}(X,Y[i])\iso D\Hom_{\DDD^{\rm b}(\mod \Lambda)}(Y,\nu_n(X)[n-i]).
\]
\item[(b)] We have the following commutative diagram:
\[ \xymatrix@C=5em{
\DDD^{\rm b}(\mod \Lambda)\ar^{D}[d]\ar^{\nu_n}[r]&\DDD^{\rm b}(\mod \Lambda) \ar^{D}[d]\\
\DDD^{\rm b}(\mod \Lambda^{\rm op})\ar^{\nu_n^{-1}}[r]&\DDD^{\rm b}(\mod \Lambda^{\rm op})
}\]
\item[(c)] For $i, j \in \Z$ we have
\begin{align*} \Ho^i(\nu_n^j(D \Lambda)) & = 
\Hom_{\DDD^{\rm b}(\mod \Lambda)}(\Lambda, \nu_n^j(D \Lambda)[i]) = \Hom_{\DDD^{\rm b}(\mod \Lambda)}(\nu_n^{-j}(\Lambda), D \Lambda [i]) \\
& = D \Ho^{-i}(\nu_n^{-j}(\Lambda)).
\end{align*}
\end{itemize}
\end{observation}

The following basic property of $\nu_n$ follows from the fact that we assume $\gl\Lambda \leq n$.

\begin{proposition}\label{t-structure}(e.g. \cite[Proposition 5.4]{I4})
The following inclusions hold.
\begin{itemize}
\item[(a)] $\nu_n(\DDD^{\ge0}(\mod \Lambda)) \subset \DDD^{\ge0}(\mod \Lambda)$.
\item[(b)] $\nu_n^{-1}(\DDD^{\le0}(\mod \Lambda)) \subset \DDD^{\le0}(\mod \Lambda)$.
\end{itemize}
\end{proposition}

As an easy consequences we get the following results that we will use often.

\begin{proposition}\label{modules under nakayama}
\begin{itemize}
\item[(a)] Let $M\in \mod\Lambda$. If $\nu_n^i(M) \in \mod\Lambda $ for some $i >0$, then $\nu_n^j(M) \in \mod\Lambda$ for all $0 \leq j \leq i$.
\item[(b)] $\Hom_{\DDD^{\bo}(\mod\Lambda)}(\nu_n^i(\Lambda),\nu_n^j(\Lambda))=0$ for any integers $i<j$.
\end{itemize}
\end{proposition}

\begin{proof}
(a) This is clear from Proposition~\ref{t-structure}.

(b) Without loss of generality, we can assume $i=0<j$.
Then $\nu_n^j(\Lambda)=\nu_n^{j-1}(D\Lambda[-n])\in\DDD^{\ge n}(\mod \Lambda)$ by Proposition~\ref{t-structure}.
Thus $\Hom_{\DDD^{\bo}(\mod\Lambda)}(\Lambda,\nu_n^{j}(\Lambda))=
\Ho^0(\nu_n^{j-1}(D\Lambda[-n]))=0$.
\end{proof}

We define the \emph{$n$-Auslander-Reiten translations} by
\begin{eqnarray*}
\tau_n:=D\Ext^n_\Lambda(-,\Lambda) \colon \mod\Lambda\to\mod\Lambda\ \mbox{ and }\ 
\tau_n^-:=\Ext^n_{\Lambda^{\rm op}}(D-,\Lambda) \colon \mod\Lambda\to\mod\Lambda.
\end{eqnarray*}
Then we have
\begin{equation}\label{description of taun}
\tau_n\iso\Tor^\Lambda_n(D\Lambda,-)\ \mbox{ and }\ \tau_n^-\iso\Ext^n_\Lambda(D\Lambda,-)\iso\Ext^n_\Lambda(D\Lambda,\Lambda)\otimes_\Lambda-
\end{equation}
They are closely related to the functors $\nu_n$ by the following formulas. 
\[\tau_n=\Ho^0(\nu_n-) \colon \mod\Lambda\to\mod\Lambda\ \mbox{ and }\ \tau_n^-=\Ho^0(\nu_n^{-1}-) \colon \mod\Lambda\to\mod\Lambda.\]
The functors $\tau_n$ and $\tau_n^-$ turn out to play an important role in higher dimensional Auslander-Reiten theory of $n$-representation finite algebras (see \cite{HI1,HI2,I1,I4,IO1,IO2}) defined as follows:

\begin{definition}
We say that a $\Lambda$-module $M$ is \emph{$n$-cluster tilting} if
\begin{eqnarray*}
\add M&=&\{X\in\mod\Lambda \mid \Ext^i_\Lambda(M,X)=0\ \mbox{ for any }\ 0<i<n\}\\
&=&\{X\in\mod\Lambda \mid \Ext^i_\Lambda(X,M)=0\ \mbox{ for any }\ 0<i<n\}.
\end{eqnarray*}
We say that $\Lambda$ is \emph{$n$-representation finite} if
there exists an $n$-cluster tilting $\Lambda$-module and $\gl\Lambda\le n$.
\end{definition}
Notice that $M\in\mod\Lambda$ is $n$-cluster tilting if and only if $DM\in\mod\Lambda^{\op}$ is $n$-cluster tilting. Thus $\Lambda$ is $n$-representation finite if and only if $\Lambda^{\op}$ is $n$-representation finite.

\begin{example}
\begin{itemize}
\item[(a)] Let $n=1$. Then $1$-cluster tilting $\Lambda$-modules are nothing but additive generators of $\mod\Lambda$.
Thus $1$-representation finite algebras are precisely representation finite hereditary algebras.
If $K$ is algebraically closed, then they are path algebras $KQ$ of Dynkin quivers $Q$ up to Morita equivalences.
\item[(b)] It is shown in \cite{HI2} that $2$-representation finite algebras over algebraically closed fields
are precisely truncated Jacobian algebras of selfinjective quivers with potentials.
\end{itemize}
See \cite{HI1,HI2,IO1} for more examples of $n$-representation finite algebras.
\end{example}

Now recall the following description of $n$-representation finite algebras.

\begin{proposition}\label{finite description}
The following conditions are equivalent.
\begin{itemize}
\item[(a)] $\Lambda$ is $n$-representation finite.
\item[(b)] Any indecomposable projective $\Lambda$-module $P$ satisfies that $\nu_n^{-i}(P)$ is an indecomposable injective $\Lambda$-module for some $i\ge0$.
\end{itemize}
In this case $\Lambda$ has an $n$-cluster tilting module $\bigoplus_{i\ge0}\tau_n^{-i}(\Lambda)$.
\end{proposition}

\begin{proof}
(a)$\Leftrightarrow$(b) is shown in \cite[Theorem 3.1(1)$\Leftrightarrow$(2)]{IO2}.

The statement for $n$-cluster tilting module follows from \cite[Theorem 1.6]{I4}.
\end{proof}

Motivated by the characterization of $n$-representation finite algebras in Proposition~\ref{finite description}(b), we define $n$-representation infinite algebras as follows.

\begin{definition}
Let $n$ be a positive integer. We say that a finite dimensional algebra $\Lambda$ is
\emph{$n$-representation infinite} if any indecomposable $\Lambda$-module $P$ satisfies that $\nu_n^{-i}(P)\in\mod\Lambda$ for any $i\ge0$, and $\gl\Lambda\le n$.

In this case, $\Lambda$ has global dimension precisely $n$ since $\Ext^n_\Lambda(D\Lambda,\Lambda)=\nu_n^{-1}(\Lambda)\neq0$.
\end{definition}

The name is explained by the following observation.

\begin{example}\label{path algebra}
Let $n=1$. Then $1$-representation infinite algebras are precisely representation infinite hereditary algebras:
For $\Lambda$ hereditary we have $\nu_1^{-1}(X) = \tau^-(X)$ for any non-injective indecomposable $\Lambda$-module $X$. The claim now follows from the descriptions of AR-quivers of hereditary algebras \cite{ARS}.
\end{example}

Let us prove left-right symmetry of $n$-representation infinite algebras.

\begin{proposition} \label{prop.infinite_leftrightsymmetric}
The following conditions are equivalent.
\begin{itemize}
\item[(a)] $\Lambda$ is $n$-representation infinite.
\item[(b)] $\nu_n^{-i}(\Lambda)\in\mod\Lambda$ for any $i\ge0$.
\item[(c)] $\Lambda^{\rm op}$ is $n$-representation infinite.
\item[(d)] $\nu_n^i(D\Lambda)\in\mod\Lambda$ for any $i\ge0$.
\end{itemize}
\end{proposition}

\begin{proof}
Clearly (a) is equivalent to (b), and (c) is equivalent to (d) by Observation~\ref{easy properties of nu_n}(b).

Now we prove that (b) is equivalent to (d).
By Observation~\ref{easy properties of nu_n}(c), we have that $\nu_n^j(D \Lambda)\in \mod \Lambda$ if and only if $\nu_n^{-j}(\Lambda)\in \mod \Lambda$.
\end{proof}

\subsection{First examples}

In the rest of this section, we give a few methods to construct $n$-representation infinite algebras. 

The first one is to use tensor products.
In what follows $\otimes$ denotes the tensor product over $K$. Let $\Lambda_1$ and $\Lambda_2$ be two finite dimensional $K$-algebras with Jacobson radicals $J_1$ and $J_2$ respectively. If $\Lambda_1/J_1 \otimes \Lambda_2 /J_2$ is semisimple then 
\[
\gl (\Lambda_1 \otimes \Lambda_2) = \gl \Lambda_1 + \gl \Lambda_2
\]
by \cite{A1}. Notice that $\Lambda_1/J_1 \otimes \Lambda_2 /J_2$ is always semisimple if $K$ is perfect or if $\Lambda_1$ and $\Lambda_2$ are path algebras of quivers factored by admissible ideals. Because of this compatibility of the tensor product with global dimension it is natural to consider tensor products of $n$-representation finite and $n$-representation infinite algebras.

Recall that an $n$-representation finite algebra is called \emph{$\ell$-homogeneous} if for any indecomposable $\Lambda$-module $P$, we have that $\tau_n^{-(\ell-1)}(P)$ is an indecomposable injective $\Lambda$-module. It is shown in \cite[Corollary 1.5]{HI1} that if $\Lambda_i$ is an $\ell$-homogeneous $n_i$-representation finite algebra for $i \in \{1,2\}$ such that $\Lambda_1/J_1 \otimes \Lambda_2 /J_2$ is semisimple then $\Lambda_1\otimes\Lambda_2$ is an $\ell$-homogeneous $(n_1+n_2)$-representation finite algebra.

For $n$-representation infinite algebras, we have the following simpler result, where we do not have to care $\ell$-homogeneity.

\begin{theorem}\label{thm.tensorInf}
Let $\Lambda_i$ be an $n_i$-representation infinite algebra for $i \in \{1,2\}$.
If $\Lambda_1/J_1 \otimes \Lambda_2 /J_2$ is semisimple then
$\Lambda_1\otimes\Lambda_2$ is an $(n_1+n_2)$-representation infinite algebra.
\end{theorem}

For the proof we need to investigate the relationship between the Nakayama functors and tensor products.

\begin{lemma}\label{lem.tensorNak}
Let $\Lambda_i$ be finite dimensional algebras with $\gl \Lambda_i \leq n_i$ for $i \in \{1,2\}$. For any $X \in \DDD^{\rm b}(\mod \Lambda_1)$ and $Y \in \DDD^{\rm b}(\mod \Lambda_2)$ the following statements hold.
\begin{itemize}
\item[(a)] $\nu(X \otimes Y) \iso \nu(X) \otimes \nu(Y)$,
\item[(b)] $\nu_{n_1+n_2}(X \otimes Y) \iso \nu_{n_1} (X) \otimes \nu_{n_2} (Y)$.
\end{itemize}
\end{lemma}
\begin{proof}
First we prove part (a).
\[
\nu(X \otimes Y) 
\iso D(\Lambda_1\otimes \Lambda_2) \Lotimes_{\Lambda_1\otimes \Lambda_2} (X \otimes Y)
\iso (D \Lambda_1 \Lotimes_{\Lambda_1} X) \otimes  (D\Lambda_2  \Lotimes_{\Lambda_2} Y)
\iso \nu (X) \otimes \nu (Y).
\]
Now (b) follows from:
\[
\nu_{n_1+n_2}(X \otimes Y) = \nu (X \otimes Y)[-n_1-n_2] \iso \nu (X)[-n_1] \otimes \nu (Y)[-n_2] = \nu_{n_1} (X) \otimes \nu_{n_2} (Y).
\]
\end{proof}

\begin{proof}[Proof of Theorem~\ref{thm.tensorInf}]
Since $\Lambda_1/J_1 \otimes \Lambda_2 /J_2$ is semisimple,
$\gl (\Lambda_1\otimes\Lambda_2) = n_1+n_2$. By applying Lemma~\ref{lem.tensorNak} repeatedly we get
\[
\nu_{n_1+n_2}^{-i}(\Lambda_1\otimes\Lambda_2)
\iso \nu_{n_1}^{-i}(\Lambda_1)\otimes \nu_{n_2}^{-i}(\Lambda_2)
\in\mod (\Lambda_1\otimes\Lambda_2)
\]
for any $i \ge 0$.
\end{proof}

\medskip
Combining Example~\ref{path algebra} and Theorem~\ref{thm.tensorInf} we get the following example (see \cite[Section 3.1]{HI1} for similar examples of $n$-representation finite algebras).

\begin{example}\label{tensor of Q}
Let $Q_i$ be a connected non-Dynkin quiver for $i=1,\ldots,m$.
Then the tensor product $KQ_1\otimes\cdots\otimes KQ_m$ is an $m$-representation infinite algebra.
\end{example}

\medskip
Next we give another construction using `$n$-APR tilting modules'.
Let $\Lambda$ be a basic finite dimensional algebra with global dimension at most $n$ and $P$ be a simple projective non-injective $\Lambda$-module.
When $\Ext^i_\Lambda(D\Lambda,P)=0$ holds for any $i$ with $0<i<n$, we define a $\Lambda$-module by
\[T:=\tau_n^-(P)\oplus (\Lambda/P).\]
Then $T$ is a tilting $\Lambda$-module with projective dimension $n$ by \cite[Theorem 3.2]{IO1}. We call $T$ the \emph{$n$-APR tilting $\Lambda$-module} with respect to $P$.
It is shown in \cite[Theorem 4.7]{IO1} that if $\Lambda$ is $n$-representation finite then $\End_\Lambda(T)$ is again $n$-representation finite.

By the following result, $n$-APR tilting modules also preserve $n$-representation infiniteness.

\begin{theorem}\label{n-APR preserves n-RI}
Let $\Lambda$ be an $n$-representation infinite algebra and $T$ an $n$-APR tilting $\Lambda$-module.
Then $\End_\Lambda(T)$ is again $n$-representation infinite.
\end{theorem}

\begin{proof}
Let $\Gamma:=\End_\Lambda(T)$. By uniqueness of the Serre functor, we have a commutative diagram
\[\xymatrix{
\DDD^{\rm b}(\mod\Lambda)\ar[d]^{\nu_n^{-1}}\ar[rr]^{\RHom_\Lambda(T,-)}&&
\DDD^{\rm b}(\mod\Gamma)\ar[d]^{\nu_n^{-1}}\\
\DDD^{\rm b}(\mod\Lambda)\ar[rr]^{\RHom_\Lambda(T,-)}&&\DDD^{\rm b}(\mod\Gamma).
}\]
We only have to show that $\Hom_{\DDD^{\rm b}(\mod\Gamma)}(\Gamma,\nu_n^{-i}(\Gamma)[j])$ is zero for any $i >0$ and $j\neq0$. Since $\tau_n^-(P)=\nu_n^{-1}(P)$, we have
\begin{align*}
\Hom_{\DDD^{\rm b}(\mod\Gamma)}(\Gamma,\nu_n^{-i}(\Gamma)[j])&\iso\Hom_{\DDD^{\rm b}(\mod\Lambda)}(T,\nu_n^{-i}(T)[j])\\
&=\Hom_{\DDD^{\rm b}(\mod\Lambda)}(\Lambda/P,\nu_n^{-i}(T)[j])\oplus\Hom_{\DDD^{\rm b}(\mod\Lambda)}(\nu_n^{-1}(P),\nu_n^{-i}(T)[j]).
\end{align*}
Since $\Lambda$ is $n$-representation infinite, we have for any $i>0$ and $j\neq0$
\begin{align*}
\Hom_{\DDD^{\rm b}(\mod\Lambda)}(\Lambda/P,\nu_n^{-i}(T)[j])&=0,\\
\Hom_{\DDD^{\rm b}(\mod\Lambda)}(\nu_n^{-1}(P),\nu_n^{-i}(T)[j])
&=\Hom_{\DDD^{\rm b}(\mod\Lambda)}(P,\nu_n^{1-i}(\Lambda/P)[j])\oplus\Hom_{\DDD^{\rm b}(\mod\Lambda)}(P,\nu_n^{-i}(P)[j])\\
&=0.
\end{align*}
Thus the assertion follows.
\end{proof}

The change of the quiver with relations via 2-APR tilting was described in \cite[Theorem~3.11]{IO1}. Here we give an example.

\begin{example}
Let $Q$ be an extended Dynkin quiver of type $\widetilde{A}_1$.
Then $KQ\otimes KQ$ is $2$-representation infinite by Example~\ref{tensor of Q}, and given by the following quiver with relations.
\[ \xymatrix@R=.7em{
1\ar@2{->}[rr]^{x_1}_{y_1}\ar@2{->}[dd]^{x_3}_{y_3}&&2\ar@2{->}[dd]^{x_2}_{y_2}\\
&&&\ \ \ x_1x_2=x_3x_4,\ x_1y_2=y_3x_4,\ y_1x_2=x_3y_4,\ y_1y_2=y_3y_4.\\
3\ar@2{->}[rr]^{x_4}_{y_4}&&4
}
\]
Let $T$ be an $2$-APR tilting module corresponding to the vertex $1$.
Then $\End_\Lambda(T)$ is $2$-representation infinite by Theorem~\ref{n-APR preserves n-RI}, and given by the following quiver with relations.
\[\xymatrix@R=.7em{
1&&2\ar@2{->}[dd]^(.3){x_2}_(.3){y_2}\\
&&&\ \ \ x_2r_1+y_2r_2=0,\ x_2r_3+y_2r_4=0,\ x_4r_1+y_4r_3=0,\ x_4r_2+y_4r_4=0.\\
3\ar@2{->}[rr]_(.3){x_4}^(.3){y_4}&&4\ar@<1.5mm>[lluu]^(.6){r_1}\ar@<.5mm>[lluu]|(.7){r_2}\ar@<-.5mm>[lluu]|(.5){r_3}\ar@<-1.5mm>[lluu]_(.6){r_4}
}
\]
\end{example}

We end this section by the following example, which motivates the non-commutative algebraic geometry approach given in Section~\ref{section: NCAG}.

\begin{example} \label{ex.beilinson}
Let $\Lambda$ be the Beilinson algebra, that is the algebra given by the quiver
\[ \xymatrix{
1  \ar@/^1pc/[r]^{a^1_0}_{\scalebox{0.7}{\vdots}}\ar@/_1pc/[r]_{a^1_{n}} & 2 \ar@/^1pc/[r]^{a^2_0}_{\scalebox{0.7}{\vdots}}\ar@/_1pc/[r]_{a^2_{n}} & 3} 
\cdots\xymatrix{
n \ar@/^1pc/[r]^{a^{n}_0}_{\scalebox{0.7}{\vdots}}\ar@/_1pc/[r]_{a^{n}_{n}}& n+1}
\text{, subject to the relations }
a^k_ia^{k+1}_j = a^k_ja^{k+1}_i.
\]
By \cite{Be} this is the endomorphism algebra of the tilting bundle $T = \bigoplus_{\ell = 0}^{n} \mathcal{O}(\ell)$ over $\mathbb{P}^n$.
Thus we have a triangle equivalence
\[ F:=\RHom_{\mathbb{P}^n}(T, -) \colon \DDD^{\rm b}(\coh \mathbb{P}^n) \to \DDD^{\rm b}(\mod \Lambda) \]
Since the Serre functor on $\DDD^{\rm b}(\coh \mathbb{P}^n)$ is $(-n-1)[n]$, we have the following commutative diagram by uniqueness of the Serre functor:
\[ \xymatrix@C=5em{
\DDD^{\rm b}(\coh \mathbb{P}^n)\ar_{(n+1)}[d]\ar^{F}[r]&\DDD^{\rm b}(\mod \Lambda) \ar^{\nu_n^{-1}}[d]\\
\DDD^{\rm b}(\coh \mathbb{P}^n)\ar^{F}[r]&\DDD^{\rm b}(\mod \Lambda)
} \]
We obtain the following equivalence of statements
\begin{align*}
\Lambda \text{ is $n$-representation infinite}\quad \Longleftrightarrow \quad & \Ho^i( \nu_n^{-j}(\Lambda)) = 0 \quad \forall i \neq0,\ \forall j \ge 0 \\
\Longleftrightarrow \quad & \Hom_{\DDD^{\rm b}(\coh \mathbb{P}^n)}(T, T(j(n+1))[i]) = 0 \quad \forall i \neq0,\ \forall j \ge 0 \\
\Longleftrightarrow \quad & \Ext_{\coh \mathbb{P}^n}^i(\mathcal{O}, \mathcal{O}(j)) = 0 \quad \forall i \neq0, \forall j \ge 0 \\
\Longleftrightarrow \quad & \Ho^i(\mathcal{O}(j)) = 0 \quad \forall i \neq 0,\ \forall j \ge 0
\end{align*}
where the last equivalence is just the definition of sheaf cohomology. 
This last equivalent statement holds since the sheaves $\mathcal{O}(j)$ with $j \ge 0$ are generated by their global sections.
Consequently $\Lambda$ is $n$-representation infinite.
\end{example}

A class of $n$-representation infinite algebras which generalizes Beilinson algebras was given in \cite{AIR}.
We refer to Section~\ref{sect.Atilde} for a general construction including this class.

\section{$n$-hereditary algebras and their dichotomy}

In this section we introduce another class of algebras of global dimension at most $n$ which we call $n$-hereditary algebras. We prove that the class of $n$-hereditary algebras is the disjoint union of
the class of $n$-representation finite algebras and the class of $n$-representation infinite algebras.

Our definition of $n$-hereditary algebras is strongly motivated by the following property of the derived categories of hereditary algebras: By \cite{Ha}, for a hereditary algebra $\Lambda$ we have $\DDD^{\bo}(\mod \Lambda)=\bigvee_{\ell\in\Z}(\mod\Lambda)[\ell]$.
To generalize this, we introduce a category
\[\DDD^{n\Z}(\mod \Lambda):=\{X\in\DDD^{\rm b}(\mod \Lambda)\mid\Ho^i(X)=0\mbox{ for any }i\in\Z\backslash n\Z\}.\]
We have the following property for of global dimension at most $n$.

\begin{proposition}[{\cite[Lemma 5.2(b)]{I4}}] \label{nZ}
Let $\Lambda$ be a finite dimensional algebra with $\gl\Lambda\le n$.
Then we have
\[\DDD^{n\Z}(\mod \Lambda)=\bigvee_{\ell\in\Z}(\mod\Lambda)[\ell n].\]
In particular any object $X\in\DDD^{n\Z}(\mod \Lambda)$ is isomorphic to its homology $\bigoplus_{\ell\in\Z}\Ho^{\ell n}(X)[-\ell n]$.
\end{proposition}

For the case $n=1$, this is what discussed above. Now we introduce $n$-hereditary algebras.

\begin{definition}
We say that a finite dimensional algebra $\Lambda$ is \emph{$n$-hereditary}
if $\nu_n^i(\Lambda)\in\DDD^{n\Z}(\mod \Lambda)$ for any $i\in\Z$ and $\gl\Lambda\le n$.
\end{definition}

Clearly any hereditary algebra is $1$-hereditary since $\DDD^{1\Z}(\mod \Lambda) = \DDD^{\rm b}(\mod \Lambda)$.

We begin our discussion of $n$-hereditary algebras by giving the following characterizations.

\begin{proposition}
The following conditions are equivalent.
\begin{itemize}
\item[(a)] $\Lambda$ is $n$-hereditary.
\item[(b)] $\nu_n^{-i}(\Lambda)\in\DDD^{n\Z}(\mod \Lambda)$ for any $i\ge0$.
\item[(c)] $\Lambda^{\rm op}$ is $n$-hereditary.
\item[(d)] $\nu_n^i(D\Lambda)\in\DDD^{n\Z}(\mod \Lambda)$ for any $i\ge0$.
\end{itemize}
\end{proposition}

\begin{proof}
By Observation~\ref{easy properties of nu_n}(c), we have that (b) is equivalent to (d).
Since $\nu_n(\Lambda)=D\Lambda[-n]$, we have that (a) is equivalent to (b) and (d).
Consequently, all conditions (a), (b) and (d) are equivalent.

By Observation~\ref{easy properties of nu_n}(b), we have that (a) is equivalent to (c).
\end{proof}

The main aim of this subsection is to prove the following dichotomy result of $n$-hereditary algebras.

\begin{theorem}\label{dichotomy}
Let $\Lambda$ be a ring-indecomposable finite dimensional algebra. Then $\Lambda$ is $n$-hereditary if and only if it is either $n$-representation finite or $n$-representation infinite.
\end{theorem}

To prove this, we need two preliminary observations.
The first one is an elementary but useful generalization from the classical theory for hereditary algebras.

\begin{lemma}[{\cite[Lemma 2.3(c)]{I4}}] \label{about DnZ}
Let $\Lambda$ be a finite dimensional algebra with $\gl\Lambda\le n$.
Assume that an indecomposable $X\in\mod \Lambda$ satisfies $\nu_n^{-1}(X)\in\DDD^{n\Z}(\mod \Lambda)$.
Then either $\nu_n^{-1}(X)\in\mod\Lambda$ or $X\in\inj\Lambda$.
\end{lemma}

\begin{proof}
For the convenience of the reader, we give a quick proof in our context.
Since $Y:=\nu_n^{-1}(X)=\RHom_\Lambda(D\Lambda,X)[n]$ satisfies $\Ho^i(Y)=0$ for any $i\not\in \{0, -n\}$,
we have $Y\iso \Ho^0(Y)\oplus \Ho^{-n}(Y)[n]$ by Proposition~\ref{nZ}.
Since $Y$ is indecomposable, we have either $Y\iso\Ho^0(Y)$ or $Y\iso \Ho^{-n}(Y)[n]$.
In the former case we have $Y\in\mod\Lambda$.
In the latter case we have $\Ext^i_\Lambda(D\Lambda,X)=\Ho^{i-n}(Y) =0$ for any $i\neq 0$, so $X$ is injective.
\end{proof}

The second one gives a characterization of $n$-hereditary algebras.

\begin{lemma}\label{2 cases}
Let $\Lambda$ be a finite dimensional algebra with $\gl\Lambda\le n$.
The following conditions are equivalent.
\begin{itemize}
\item[(a)] $\Lambda$ is $n$-hereditary.
\item[(b)] For any indecomposable $P\in\proj\Lambda$, one of the following conditions holds.
\begin{itemize}
\item[(i)] $\nu_n^{-i}(P)\in\inj\Lambda$ for some $i\ge0$.
\item[(ii)] $\nu_n^{-i}(P)\in\mod\Lambda$ for any $i\ge0$.
\end{itemize}
\end{itemize}
\end{lemma}

\begin{proof}
Assume that $\Lambda$ is $n$-hereditary and let $P\in\proj\Lambda$ indecomposable.
If $P$ does not satisfy the condition (b)(ii), then there is a maximal integer $i\ge0$ such that $X:=\nu_n^{-i}(P)$ is a $\Lambda$-module.
Applying Lemma~\ref{about DnZ} to $X$, we have that $X\in\inj\Lambda$.

Now assume that (b) holds, but $\Lambda$ is not $n$-hereditary. Let
$k$ be the minimal natural number such that there is $P\in\ind(\proj\Lambda)$
satisfying $\nu_n^{-k} (P) \not \in\DDD^{n\Z}(\mod \Lambda)$.
Then $P$ does not satisfy the condition (b)(ii). Hence there is $i\ge0$ such
that $\nu_n^{-i} (P)$ is an injective $\Lambda$-module. By
Proposition~\ref{modules under nakayama}(a) we have that $\nu_n^{-j} (P) \in \mod \Lambda$ for all $0 \le
j \le i$. Moreover, $P' := \nu_n^{-i-1} (P)[-n]$ is an indecomposable
projective $\Lambda$-module. In particular $i+1 < k$. Now $\nu_n^{-(k-i-1)}
(P') = \nu_n^{-k}(P)[-n] \not \in \DDD^{n\Z}(\mod \Lambda)$. This contradicts the minimality of $k$.
\end{proof}

Now we are ready to prove Theorem~\ref{dichotomy}.

\begin{proof}[Proof of Theorem~\ref{dichotomy}]
The `if' part follows immediately from Lemma~\ref{2 cases} (b)$\Rightarrow$(a) since the condition (i) is satisfied for $n$-representation finite algebras by Proposition~\ref{finite description} (a)$\Rightarrow$(b) and the condition (ii) is satisfied for $n$-representation infinite algebras.

In the rest we show the `only if' part.
We assume that $\Lambda$ is $n$-hereditary.
Let $P\in\ind(\proj\Lambda)$.
For any integer $j\in\Z$, there exists an integer $d_P(j)$ satisfying $\nu_n^{-j}(P)\in(\mod \Lambda)[d_P(j)n]$ by Proposition~\ref{nZ}.
We have $d_P(0)=0$, $d_P(-1)=-1$, and $d_P(j+1)-d_P(j)$ is either $0$ or $1$ for any $j\in\Z$.
Thus the image of $d_P$ is an interval of integers
\[d_P(\Z)=[a_P,b_P]\]
where $a_P\in\Z_{<0}\cup\{-\infty\}$ and $b_P\in\Z_{\ge0}\cup\{\infty\}$.

Note that $b_P = 0$ is equivalent to Condition~(ii) in Lemma~\ref{2 cases}(b), and, using Lemma~\ref{about DnZ}, $b_P > 0$ is equivalent to Condition~(i) in Lemma~\ref{2 cases}(b).
If $b_P>0$ for all $P\in\ind(\proj\Lambda)$, then $\Lambda$ is $n$-representation finite by Proposition~\ref{finite description} (b)$\Rightarrow$(a).
On the other hand, if $b_P=0$ for all $P\in\ind(\proj\Lambda)$, then $\Lambda$ is $n$-representation infinite by definition.
Thus we only have to show that either $b_P>0$ for all $P\in\ind(\proj\Lambda)$ or $b_P=0$ for all $P\in\ind(\proj\Lambda)$.

\textsc{Step I:} Assume $b_P>0$ for some $P\in\ind(\proj\Lambda)$.
Then for some $i$ we have $I = \nu_n^{-i}(P) \in \inj \Lambda$, which clearly satisfies
\begin{equation}\label{I and P}
a_{\nu^{-1}(I)}=a_P-1\ \mbox{ and }\ b_{\nu^{-1}(I)}=b_P-1
\end{equation}
where $\pm\infty-1:=\pm\infty$.

\textsc{Step II:} Let $P,Q\in \proj\Lambda$ indecomposable, such that $\Hom_{\DDD^{\rm b}(\mod \Lambda)}(Q,\nu_n^{-i}(P))\neq0$ for some $i\in\Z$. Then we have
\[ \Hom_{\DDD^{\rm b}(\mod \Lambda)}(\nu_n^{-j}(Q), \nu_n^{-i-j}(P)) = \Hom_{\DDD^{\rm b}(\mod \Lambda)}(Q, \nu_n^{-i}(P))\neq0 \quad \forall j \in \Z.\]
Thus we have $d_P(i+j)-d_Q(j)$ is either $0$ or $1$ for any $j\in\Z$.
Looking at $j \ll 0$ and $j \gg 0$ respectively, we have
\begin{equation}\label{Q and P}
a_Q\in \{a_P, a_P-1\} \ \mbox{ and }\ b_Q\in \{b_P, b_P-1\}.
\end{equation}

\textsc{Step III:} We show that if $P$ satisfies $b_P=0$, then $a_P=-1$.
Among the indecomposable projectives $P'$ with $b_{P'} = 0$, choose $P$ such that $a_P$ is minimal.
For any $i\ge0$, there exists $Q_i\in\ind(\proj\Lambda)$ satisfying $\Hom_{\DDD^{\rm b}(\mod \Lambda)}(Q_i,\nu_n^{-i}(P))\neq0$.
In particular, since there are only finitely many indecomposable projectives, there is $Q$ such that $\Hom_{\DDD^{\rm b}(\mod \Lambda)}(Q,\nu_n^{-i}(P))\neq0$ for infinitely many $i$.
By \eqref{Q and P}, we have $b_{Q}=0$, and $a_{Q}=a_P$ by minimality of $a_P$.
For sufficiently large $i$ we have $\nu_n^i(Q) \in (\mod \Lambda)[a_Q n]$. Thus
\[ 0 \neq \Hom_{\DDD^{\rm b}(\mod \Lambda)}(Q, \nu_n^{-i}(P)) = \Hom_{\DDD^{\rm b}(\mod \Lambda)}(\overbrace{\nu_n^i(Q)}^{\mathclap{\in \mod \Lambda[a_{Q}n]}}, \underbrace{P}_{\mathclap{\in \mod \Lambda}}), \]
and hence $a_P =a_Q \geq -1$.

\textsc{Step IV:} We show that $b_P$ is either $0$ or $\infty$ for any $P\in\ind(\proj\Lambda)$.
Otherwise applying \eqref{I and P} repeatedly, we have $Q\in\ind(\proj\Lambda)$ such that
\[a_Q=a_P-b_P\ \mbox{ and }\ b_Q=0.\]
Then $a_Q<-1$, a contradiction to \textsc{Step III}.

\textsc{Step V:} If $b_P=\infty$ for some $P\in\ind(\proj\Lambda)$, then \eqref{Q and P} and the connectedness of $\Lambda$ imply that $b_Q=\infty$ for any $Q\in\ind(\proj\Lambda)$,
so $\Lambda$ is $n$-representation finite.
Otherwise, $b_P=0$ for all $P\in\ind(\proj\Lambda)$ by \textsc{Step IV}, so $\Lambda$ is $n$-representation infinite.
\end{proof}

\section{$n$-preprojective, $n$-preinjective, and $n$-regular modules}

The aim of this section is to introduce three classes of modules over $n$-representation infinite algebras from the viewpoint of higher dimensional Auslander-Reiten theory.
We define them by using the $n$-Auslander-Reiten translation functors $\tau_n$ and $\tau_n^-$ in a similar way as the classical case $n=1$ of hereditary algebras.

We start by discussing the more general case of $n$-hereditary algebras.
Let $\Lambda$ be an $n$-hereditary algebra. Then an important role is played by the full subcategory
\[ \UU := \add\{\nu_n^i(\Lambda) \mid i\in\Z\} \]
of $\DDD^{\rm b}(\mod \Lambda)$, which is contained in $\DDD^{n\Z}(\mod \Lambda)$.
The category $\UU$ already appeared in the study of $n$-representation finite algebras \cite{I4,IO1,IO2}:

\begin{remark}[{\cite[Theorem 1.23]{I4}}] \label{U for n-RF}
If $\Lambda$ is $n$-representation finite,
then $\UU$ is an $n$-cluster tilting subcategory of $\DDD^{\rm b}(\mod \Lambda)$.
\end{remark}

The category $\UU$ is also important for $n$-representation infinite algebras. The first remarkable property of $\UU$ is the following.

\begin{proposition}\label{n rigid}
Let $\Lambda$ be an $n$-hereditary algebra.
Then we have $\Hom_{\DDD^{\rm b}(\mod \Lambda)}(\UU,\UU[i])=0$ for any $i\in\Z\backslash n\Z$.
\end{proposition}

\begin{proof}
Since $\UU\subset\DDD^{n\Z}(\mod \Lambda)$, we have $\Hom_{\DDD^{\rm b}(\mod \Lambda)}(\Lambda,\UU[i])=\Ho^i(\UU)=0$ for any $i\in\Z\backslash n\Z$.
Since $\nu_n$ is an autoequivalence of $\DDD^{\rm b}(\mod \Lambda)$ satisfying $\nu_n(\UU)=\UU$, we have the assertion.
\end{proof}

Next we have the following description of indecomposable objects in $\UU$.

\begin{proposition}\label{ind U}
We have a bijection from $\ind(\proj\Lambda)\times\Z$ to $\ind\UU$ given by $(P,i)\mapsto\nu_n^i(P)$.
\end{proposition}

\begin{proof}
Since $\nu_n$ is an equivalence, we have that $\nu_n^i(P)$ is indecomposable for any $(P,i)\in\ind(\proj\Lambda)\times\Z$.
Thus the map $(P,i)\mapsto\nu_n^i(P)$ gives a surjection $\ind(\proj\Lambda)\times\Z\to\ind\UU$.
It remains to prove injectivity.
Assume $\nu_n^i(P)\iso\nu_n^j(Q)$ for $(P,i),(Q,j)\in\ind(\proj\Lambda)\times\Z$.
By Proposition~\ref{modules under nakayama}(b), we have $i=j$. Thus we have $P\iso Q$.
\end{proof}

Again let $\Lambda$ be an $n$-hereditary algebra.
We will study the full subcategory
\begin{eqnarray*}
\CC &:=& \{ X \in \DDD^{\rm b}(\mod \Lambda) \mid \forall i \in\Z\backslash n\Z \colon \Hom_{\DDD^{\rm b}(\mod \Lambda)}(\UU, X[i]) = 0\},\\
&=&\{ X \in \DDD^{\rm b}(\mod \Lambda) \mid \forall i \in\Z\backslash n\Z \colon \Hom_{\DDD^{\rm b}(\mod \Lambda)}(X, \UU [i]) = 0\}
\end{eqnarray*}
inside $\DDD^{\rm b}(\mod \Lambda)$, where the second equality is a conclusion of Observation~\ref{easy properties of nu_n}(a) and $\nu_n(\UU)=\UU$.
We also study its module category analog
\[ \CC^{0} := (\mod\Lambda)\cap\CC. \]
The following assertions are clear.

\begin{observation} \label{autoeq}
Since $\nu_n$ induces an autoequivalence of $\UU$, $\nu_n$ also induces an autoequivalence of $\CC$. Clearly $[n]$ gives an autoequivalence of $\CC$, so it follows that also $\nu$ induces an autoequivalence of $\CC$.
\end{observation}

Now we give basic properties of $\CC$.
The second equality shows that $\CC$ is very similar to the derived category of hereditary algebras.

\begin{proposition}\label{basic of C}
Let $\Lambda$ be an $n$-hereditary algebra. The following assertions hold.
\begin{itemize}
\item[(a)] $\UU\subset\CC \subset \DDD^{n\Z}(\mod \Lambda)$.
\item[(b)] $\CC = \bigvee_{\ell\in\Z} \CC^{0}[\ell n]$.
\end{itemize}
\end{proposition}

\begin{proof}
(a) By Proposition~\ref{n rigid}, we have $\UU\subset\CC$. Since $\Lambda\in\UU$, we have $\CC\subset\DDD^{n\Z}(\mod \Lambda)$ by Proposition \ref{n rigid}.

(b) Since $\CC[\ell n]=\CC$ holds by Observation~\ref{autoeq}, we have $(\mod\Lambda)[\ell n]\cap\CC=\CC^{0}[\ell n]$ for any $\ell\in\Z$.
In particular, $\CC$ contains the right hand side.
Moreover $\CC$ is contained in the right hand side by (a) and Proposition~\ref{nZ}.
\end{proof}

\begin{remark}
If $\Lambda$ is $n$-representation finite, then $\CC=\UU$ by Remark~\ref{U for n-RF}.
Moreover $\CC^{0} = \add\{\tau_n^{-i}(\Lambda) \mid i\ge0\}$ is the $n$-cluster tilting subcategory of $\mod \Lambda$ by Proposition \ref{finite description}.
\end{remark}

In the rest of this section we assume that $\Lambda$ is $n$-representation infinite.
Following the representation theory of hereditary algebras, we introduce important subcategories of $\mod\Lambda$.

\begin{definition}
Let $\Lambda$ be an $n$-representation infinite algebra. By Proposition~\ref{prop.infinite_leftrightsymmetric} we can define two full subcategories of $\mod \Lambda$ by
\begin{align*}
\PP &:= \add\{\nu_n^{-i}(\Lambda) \mid i\ge0\} =\add\{\tau_n^{-i}(\Lambda) \mid i\ge0\},\\
\II &:= \add\{\nu_n^{i}(D\Lambda) \mid i\ge0\} =\add\{\tau_n^{i}(D\Lambda) \mid i\ge0\}. 
\end{align*}
We call modules in $\PP$ (respectively, $\II$)
\emph{$n$-preprojective} (respectively, \emph{$n$-preinjective}) modules.
\end{definition}

\begin{observation}\label{applied n-shifted serre}
Clearly we have $\nu_n(\PP)=\PP\vee(\inj\Lambda)[-n]$ and $\nu_n^{-1}(\II)=\II\vee(\proj\Lambda)[n]$.
Applying Observation~\ref{easy properties of nu_n}(a), we have equivalences
\begin{eqnarray*}
\Ext_{\Lambda}^i(\PP,X)=0  &\Longleftrightarrow& \Ext_{\Lambda}^{n-i}(X,\PP)=0,\\
\Ext_{\Lambda}^i (X,\II)=0  &\Longleftrightarrow& \Ext_{\Lambda}^{n-i}(\II,X)=0
\end{eqnarray*}
for any $X\in\mod\Lambda$ and $i$ with $1\le i\le n$, which we will often use in this paper.
\end{observation}

We collect easy observations.

\begin{proposition}\label{P and I}
Let $\Lambda$ be an $n$-representation infinite algebra. The following assertions hold.
\begin{itemize}
\item[(a)] We have a bijection from $\ind(\proj\Lambda)\times\Z_{\ge0}$ to $\ind\PP$ given by $(P,i)\mapsto\nu_n^{-i}(P)$.
\item[(b)] We have a bijection from $\ind(\inj\Lambda)\times\Z_{\ge0}$ to $\ind\II$ given by $(I,i)\mapsto\nu_n^i(I)$.
\item[(c)] $\UU=\II[-n]\vee\PP$.
\item[(d)] $\Hom_\Lambda(\II,\PP)=0$ and $\PP\cap\II=0$.
\item[(e)] $\PP\vee\II\subset\CC^{0}$.
\item[(f)] $\Ext_\Lambda^i(\PP\vee\II,\PP\vee\II)=0$ for any $i$ with $0<i<n$.
\item[(g)] $\CC^{0} = \{X\in\mod\Lambda \mid \forall i \in \{1, \ldots, n-1\}\colon
\Ext^i_\Lambda(\PP\vee\II,X)= 0\}$.
\end{itemize}
\end{proposition}

\begin{proof}
(a) and (b) follow from Proposition~\ref{ind U}, and
(c) and (g) are clear from definition.

(d) Since $\Hom_\Lambda(\nu_n^i(D\Lambda),\nu_n^{-j}(\Lambda))\iso\Hom_\Lambda(\Lambda[n],\nu_n^{-i-j-1}(\Lambda))=0$, we have the first assertion.
Now the second one is clear.

(e) and (f) follow immediately from Proposition~\ref{n rigid}.
\end{proof}

As a consequence, we have an answer to the second author's Question~\ref{ICRA XI}.

\begin{corollary} \label{Cor.Ext-orth-family}
Let $\Lambda$ be an $n$-representation infinite algebra.
Then $\ind(\PP\vee\II)$ is an infinite set of isomorphism classes of indecomposable $\Lambda$-modules such that
$\Ext_\Lambda^i(X,Y)=0$ for any $i$ with $0<i<n$ and any $X,Y\in\ind(\PP\vee\II)$.
\end{corollary}

Next we show that it is easy to calculate dimension vectors of modules in $\PP$ and $\II$.

\begin{observation}
Assume that $\Lambda$ is a basic $n$-representation infinite algebra over an algebraically closed field $K$
and $1=e_1+\cdots+e_m$ is a decomposition into primitive orthogonal idempotents. Let 
\[C:=\left[\dim(e_i\Lambda e_j)\right]_{1\le i,j\le m} \]
be the \emph{Cartan matrix} of $\Lambda$.
That is the columns of $C$ are the dimension vectors of the indecomposable projectives $\Lambda e_j$. Moreover let
\[ \Phi := (-1)^nC^t \cdot C^{-1}\]
be the \emph{Coxeter matrix} of $\Lambda$.
Then $\Phi$ gives the action of $\nu_n$ on the Grothendieck group $K_0(\DDD^{\bo}(\mod\Lambda))$ (with respect to the basis consisting of the simple $\Lambda$-modules).
In particular, we have
\begin{equation}\label{dimension vector}
\left[\underline{\dim} \; \nu_n^{-\ell}(\Lambda e_j)\right]_{1\le j\le m} = \Phi^{-\ell} C \quad \text{ and } \quad
\left[\underline{\dim} \; \nu_n^{\ell}(D(e_j\Lambda))\right]_{1\le j\le m} = \Phi^{\ell} C^t.
\end{equation}
\end{observation}

We give a simple example.

\begin{example}\label{beilinson dim2 part1}
Let $\Lambda$ be a Beilinson algebra of dimension $2$:
\[\xymatrix{
1  \ar@/^1pc/[r]|{x_0}\ar[r]|{x_1}\ar@/_1pc/[r]|{x_2} &
2  \ar@/^1pc/[r]|{x_0}\ar[r]|{x_1}\ar@/_1pc/[r]|{x_2} & 3}
\ \ \ x_ix_j=x_jx_i\ (i,j\in\{0,1,2\}).\]
Let
\[P_{i+3\ell}:=\nu_2^{-\ell}(\Lambda e_i)\ \mbox{ and }\ 
I_{4-i+3\ell}:=\nu_2^{\ell}(D(e_i\Lambda))\ \mbox{for $i\in\{1,2,3\}$ and $\ell\ge0$.}\]
Then we have $\PP=\add\{P_i\mid i\ge1\}$ and $\II=\add\{I_i\mid i\ge1\}$ by Proposition~\ref{P and I}(b)(c).
The Cartan matrix and the Coxeter matrix of $\Lambda$ are
\[C=\left[\begin{smallmatrix}
1&3&6\\
0&1&3\\
0&0&1\end{smallmatrix}\right] \quad \mbox{ and } \quad
\Phi=\left[\begin{smallmatrix}
1&-3&3\\
3&-8&6\\
6&-15&10\end{smallmatrix}\right].\]
Using \eqref{dimension vector}, one can show inductively
\[ \underline{\dim} \; P_i = \left[ \begin{smallmatrix} \frac{i(i+1)}{2} \\ \frac{(i-1)i}{2} \\ \frac{(i-2)(i-1)}{2} \end{smallmatrix} \right] \quad \text{ and } \quad \underline{\dim} \; I_i = \left[ \begin{smallmatrix} \frac{(i-2)(i-1)}{2} \\ \frac{(i-1)i}{2} \\ \frac{i(i+1)}{2} \end{smallmatrix} \right]. \]
\end{example}

In representation theory of representation infinite hereditary algebras, apart from preprojective and preinjective modules there is a third important class -- the regular modules. Inspired by this theory, we define the following.

\begin{definition}\label{definition of R}
The category of \emph{$n$-regular modules} is defined to be
\[\RR:=\{X \in \mod \Lambda \mid \forall i > 0 \colon \Ext_{\Lambda}^i(\PP, X) = 0 = \Ext_{\Lambda}^i(X, \II) \}.\]
\end{definition}

Immediately we have the following descriptions of $\RR$.

\begin{observation}\label{almost definition of R}
By Proposition~\ref{P and I}(g) and Observation~\ref{applied n-shifted serre},
\[ \RR =\{X \in \CC^0 \mid \Ext_{\Lambda}^n(\PP, X) = 0 = \Ext_{\Lambda}^n(X, \II)\} \]
and one can replace the condition $\Ext_{\Lambda}^n(\PP, X)=0$ (respectively, $\Ext_{\Lambda}^n(X, \II)=0$) by $\Hom_\Lambda(X,\PP)=0$ (respectively, $\Hom_\Lambda(\II,X)=0$).
\end{observation}

We have the following important description of $\RR$ in terms of the functor $\nu_n$.

\begin{proposition}\label{nakayama R}
Let $\Lambda$ be an $n$-representation infinite algebra. Then
\[\RR=\{X \in \mod \Lambda \mid \forall i \in \Z \colon \nu_n^i(X) \in \mod \Lambda \}.\]
\end{proposition}

We need the following observations.

\begin{lemma}\label{Ext and nu}
Let $\Lambda$ be an $n$-representation infinite algebra.
\begin{itemize}
\item[(a)] $\{X \in \mod \Lambda \mid \forall j\ge 0 \colon \nu_n^{j}(X)\in\mod\Lambda\}=\{X \in \mod \Lambda \mid \forall i>0 \colon \Ext_{\Lambda}^i(\PP,X) = 0 \}$.
\item[(b)] $\{X \in \mod \Lambda \mid \forall j\ge0 \colon \nu_n^{-j}(X)\in\mod\Lambda \}=\{X \in \mod \Lambda \mid \forall i>0 \colon \Ext_{\Lambda}^i(X,\II) = 0 \}$.
\end{itemize}
\end{lemma}

\begin{proof}
(a) We have the following isomorphisms:
\[
\Ext_{\Lambda}^i(\nu_n^{-j}(\Lambda), X) = \Hom_{\DDD^{\rm b}(\mod \Lambda)}(\nu_n^{-j}(\Lambda), X[i]) = \Hom_{\DDD^{\rm b}(\mod \Lambda)}( \Lambda, \nu_n^{j}(X)[i]) = \Ho^i(\nu_n^{j}(X)).
\]
Thus $\Ext_{\Lambda}^i(\PP, X)=0$ for any $i>0$ if and only if $\nu_n^{j}(X)\in\mod\Lambda$ for any $j\ge0$.

(b) This is shown dually.
\end{proof}

Now we are able to prove Proposition~\ref{nakayama R} as follows:

\begin{proof}[Proof of Proposition~\ref{nakayama R}]
\begin{align*}
\RR&=\{X \in \mod\Lambda \mid \forall i> 0 \colon \Ext_{\Lambda}^i(\PP, X) = 0 = \Ext_{\Lambda}^i(X, \II) \}\\
&=\{X \in \mod \Lambda \mid \forall i \in\Z \colon \nu_n^{-i}(X)\in\mod\Lambda \}
&&\mbox{by Lemma~\ref{Ext and nu}.} \qedhere
\end{align*}
\end{proof}

Now we have the following main result in this section.

\begin{theorem} \label{trichotomy}
Let $\Lambda$ be an $n$-representation infinite algebra.
\begin{itemize}
\item[(a)] We have
\[\CC^{0} = \PP \vee \RR \vee \II\ \mbox{ and }\ \CC=\bigvee_{\ell\in\Z}(\PP \vee \RR \vee \II)[\ell n].\]
\item[(b)] We have
\begin{eqnarray*}
\Hom_{\Lambda}(\RR, \PP) = 0,\ \Hom_{\Lambda}(\II, \PP) = 0,\ \Hom_{\Lambda}(\II, \RR) = 0,\\
\Ext_{\Lambda}^n(\PP, \RR) = 0,\ \Ext_{\Lambda}^n(\PP, \II) = 0,\ \mbox{ and }\ \Ext_{\Lambda}^n(\RR, \II) = 0.
\end{eqnarray*}
\end{itemize}
\end{theorem}

We need the following preliminary observations.

\begin{lemma}\label{P and I in C^0}
Let $\Lambda$ be an $n$-representation infinite algebra.
\begin{itemize}
\item[(a)] $\proj\Lambda=\add \{ X \in \ind \CC^{0} \mid \nu_n(X)\notin\CC^0\}$.
\item[(b)] $\PP = \add \{ X \in \ind \CC^{0} \mid \exists j > 0 \colon \nu_n^{j}(X)\notin\CC^0\}=\add \{ X \in \ind \CC^{0} \mid \Hom_\Lambda(X,\PP)\neq0\}$.
\item[(c)] $\inj\Lambda = \add \{ X \in \ind \CC^{0} \mid \nu_n^{-1}(X)\notin\CC^0\}$.
\item[(d)] $\II = \add \{ X \in \ind \CC^{0} \mid \exists j > 0 \colon \nu_n^{-j}(X)\notin\CC^0\}=\add \{ X \in \ind \CC^{0} \mid \Hom_\Lambda(\II,X)\neq0\}$.
\end{itemize}
\end{lemma}

\begin{proof}
(c) ``$\subseteq$'' is clear, and ``$\supseteq$'' is immediate from Lemma~\ref{about DnZ}.

(d) The first equality is immediate from (c) since $\II=\bigvee_{j\ge0} \; \nu_n^j(\inj\Lambda)$.
The second equality follows as follows:
\begin{align*}
& \add \{ X \in \ind \CC^{0} \mid \exists j > 0 \colon \nu_n^{-j}(X)\notin\CC^0\} \\
&=\add \{ X \in \ind \CC^{0} \mid \Ext^n_\Lambda(X, \II)\neq0\}
&&\mbox{by Lemma~\ref{Ext and nu},}\\
&=\add \{ X \in \ind \CC^{0} \mid \Hom_\Lambda(\II, X)\neq0\}
&&\mbox{by Observation~\ref{applied n-shifted serre}.}
\end{align*}

(a) and (b) are shown dually.
\end{proof}

Now we are ready to prove Theorem~\ref{trichotomy}.

\begin{proof}[Proof of Theorem~\ref{trichotomy}]
(b) We know $\Ext_{\Lambda}^n(\PP, \RR) = 0$, $\Ext_{\Lambda}^n(\RR, \II) = 0$, $\Hom_{\Lambda}(\RR, \PP) = 0$ and $\Hom_{\Lambda}(\II, \RR) = 0$ by Observation~\ref{almost definition of R}. We have $\Hom_\Lambda(\II,\PP)=0$ by Proposition~\ref{P and I}(d). By Observation~\ref{applied n-shifted serre}, we have $\Ext_{\Lambda}^n(\PP,\II) = 0$.

(a) We know $\PP\vee\RR\vee\II\subset\CC^{0}$ by Proposition~\ref{P and I}(e) and Observation~\ref{almost definition of R}.
Now we show $\PP \vee \RR \vee \II\supset\CC^{0}$.
If $X\in\ind\CC^{0}$ does not belong to $\PP$ or $\II$, then we have $\Hom_{\Lambda}(\II,X)=0$ and $\Hom_{\Lambda}(X, \PP)=0$ by Lemma~\ref{P and I in C^0}.
Then $X\in\RR$ by Observation~\ref{almost definition of R}.
Thus we have the first equality. The second one follows from Proposition~\ref{basic of C}(b).
\end{proof}

We note the following easy property of our categories.

\begin{proposition}\label{extension closed}
$\PP$, $\RR$, $\II$, and $\CC^0$ are extension-closed subcategories of $\mod\Lambda$.
\end{proposition}

\begin{proof}
Since $\RR$ and $\CC^0$ are defined by a vanishing of extension groups, they are extension-closed.

If $n\ge2$, then $\PP$ and $\II$ are extension-closed since $\PP\vee\II$ does not have first selfextensions by Proposition~\ref{P and I}(f). The claim is known in the case $n=1$.
\end{proof}

\subsection{Higher almost split sequences in $\PP$ and $\II$}

In this section, we study structure of the categories $\PP$, $\II$ and $\RR$ from Auslander-Reiten theoretic viewpoint.

Let $\CC^{0}_P$ (respectively, $\CC^{0}_I$) be the full subcategory of $\CC^{0}$ consisting of modules without non-zero projective (respectively, injective) direct summands.
Set $\PP_P = \PP \cap \CC^{0}_P$ and $\II_I = \II \cap \CC^{0}_I$.

\begin{proposition}\label{tau_n}
We have mutually quasi-inverse equivalences
\[\xymatrix{\CC^{0}_P\ar@<.5ex>[rr]^{\nu_n=\tau_n}&&\CC^{0}_I\ar@<.5ex>[ll]^{\nu_n^{-1}=\tau^-_n}}\]
which restricts to equivalences
\[\xymatrix{\PP_P\ar@<.5ex>[rr]^{\nu_n=\tau_n}&&\PP\ar@<.5ex>[ll]^{\nu_n^{-1}=\tau^-_n} } \text{, }
\xymatrix{\RR\ar@<.5ex>[rr]^{\nu_n=\tau_n}&&\RR\ar@<.5ex>[ll]^{\nu_n^{-1}=\tau^-_n} } \text{, and}
\xymatrix{\II\ar@<.5ex>[rr]^{\nu_n=\tau_n}&&\II_I.\ar@<.5ex>[ll]^{\nu_n^{-1}=\tau^-_n}} \]
\end{proposition}

\begin{proof}
We only have to prove the equivalence between $\CC^{0}_P$ and $\CC^{0}_I$.
By Observation~\ref{autoeq} we have that $\nu_n$ and $\nu_n^{-1}$ induce mutually quasi-inverse autoequivalences on $\CC$.
By Lemma~\ref{P and I in C^0}(a)(c), they induce equivalences between $\CC^0_P$ and $\CC^0_I$.
Moreover we have $\tau_n=\nu_n$ on $\CC^0_P$ and $\tau_n^-=\nu_n^{-1}$ on $\CC^0_I$.
\end{proof}

The following standard notion plays a key role to study the structure of Krull-Schmidt categories.

\begin{definition}
Let $\AA$ be a Krull-Schmidt category with Jacobson radical $J_{\AA}$.
\begin{itemize}
\item[(a)] For $X\in\ind\AA$, we say that a morphism $f\in J_{\AA}(C,X)$ in $\AA$ is \emph{right almost split} if the following sequence is exact on $\AA$.
\[\Hom_{\AA}(-,C)\xrightarrow{\cdot f}J_{\AA}(-,X)\to0.\]
A right minimal and right almost split morphism is called a \emph{sink morphism}.

\item[(b)] Dually we define a \emph{left almost split morphism} and a \emph{source morphism}.
\end{itemize}
\end{definition}

The following notion was introduced in \cite{I1} in higher dimensional Auslander-Reiten theory.

\begin{definition} \label{def.ASS}
Let $\AA$ be a Krull-Schmidt category with the Jacobson radical $J_{\AA}$.
We call a complex
\[0\to Y\xrightarrow{f_n}C_{n-1}\xrightarrow{f_{n-1}}C_{n-2}\xrightarrow{f_{n-2}}\cdots\xrightarrow{f_2}C_1\xrightarrow{f_1}C_0\xrightarrow{f_0}X\to 0\]
in $\AA$ \emph{an $n$-almost split sequence in $\AA$} if the following conditions are satisfied.
\begin{itemize}
\item[(a)] $X$ and $Y$ are indecomposable and $f_i\in J_{\AA}$ for any $i$.
\item[(b)] The following sequences are exact on $\AA$:
\begin{eqnarray}\label{sink sequence}
&0\to\Hom_{\AA}(-,Y)\xrightarrow{f_n}\Hom_{\AA}(-,C_{n-1})\xrightarrow{f_{n-1}}\cdots\xrightarrow{f_1}\Hom_{\AA}(-,C_0)\xrightarrow{f_0}J_{\AA}(-,X)\to0,&\\ \label{source sequence}
&0\to\Hom_{\AA}(X,-)\xrightarrow{f_0}\Hom_{\AA}(C_0,-)\xrightarrow{f_1}\cdots\xrightarrow{f_{n-1}}\Hom_{\AA}(C_{n-1},-)\xrightarrow{f_n}J_{\AA}(Y,-)\to0.&
\end{eqnarray}
\end{itemize}
\end{definition}

For $n$-representation finite algebras, we have the following existence theorem of $n$-almost split sequences.

\begin{example}
Let $\Lambda$ be an $n$-representation finite algebra.
Then the $n$-cluster tilting subcategory $\CC^0$ of $\mod\Lambda$ has $n$-almost split sequences by \cite[Theorem 3.3.1]{I1}.
\end{example}

The aim of this subsection is to prove the following existence theorem of $n$-almost split sequences in $\CC^{0}$ in case $\Lambda$ is $n$-representation infinite.

\begin{theorem}\label{n-almost split}
\begin{itemize}
\item[(a)] Any indecomposable $X\in\PP$ (respectively, $X\in\II$) has a sink morphism $C\to X$ and a source morphism $X\to C'$ in $\CC^{0}$ with $C,C'\in\PP$ (respectively, $C,C'\in\II$).
\item[(b)] For any indecomposable $X\in\PP_P\vee\II$ (respectively, indecomposable $Y\in\PP\vee\II_I$),
there exists an $n$-almost split sequence in $\CC^{0}$
\[0\to Y\xrightarrow{}C_{n-1}\xrightarrow{}C_{n-2}\xrightarrow{}\cdots\xrightarrow{}C_1\xrightarrow{}C_0\xrightarrow{}X\to 0\]
such that $Y\iso\tau_n(X)$ and $X\iso\tau_n^-(Y)$.
\item[(c)] If $X$ or $Y$ belongs to $\PP$ (respectively, $\II$) in (b), then all terms in the sequence belong to $\PP$ (respectively, $\II$).
\end{itemize}
\end{theorem}

Note that we do not know whether sink morphisms and source morphisms exist for $n$-regular modules. We will discuss this problem later in this subsection.

The first step of the proof is the following observation, which is an analogue of \cite[Proposition~3.3]{I1}.

\begin{lemma}\label{self duality}
Let
\[0\to Y\xrightarrow{f_n}C_{n-1}\xrightarrow{f_{n-1}}C_{n-2}\xrightarrow{f_{n-2}}\cdots\xrightarrow{f_2}C_1\xrightarrow{f_1}C_0\xrightarrow{f_0}X\to 0\]
be an exact sequence with terms in $\PP$ such that $f_i\in J_{\PP}$ for any $i$.
Then the following equivalent conditions are equivalent.
\begin{itemize}
\item[(a)] This is an $n$-almost split sequence in $\CC^{0}$.
\item[(b)] $X$ is indecomposable and $f_0$ is a sink map in $\CC^{0}$.
\item[(c)] $Y$ is indecomposable and $f_n$ is a source map in $\CC^{0}$.
\end{itemize}
In this case, we have $Y\iso\tau_n(X)$ and $X\iso\tau_n^-(Y)$.
\end{lemma}

\begin{proof}
By definition of $\CC$, we have
\begin{equation}\label{vanish C and P}
\Ext_\Lambda^i(\CC^{0},\PP)=0=\Ext_\Lambda^i(\PP,\CC^{0})\ \forall i\in\{1,\ldots,n-1\}.
\end{equation}
Using this, it is easily checked that the sequence \eqref{sink sequence} (respectively, \eqref{source sequence}) is exact if and only if
$f_0$ is a sink map (respectively, $f_n$ is a source map) in $\CC^{0}$ (see \cite[Lemma~3.2]{I1} for details).

Thus we only have to check that (b) is equivalent to (c).
We define functors $F\in\mod\CC^{0}$ and $G\in\mod(\CC^{0})^{\rm op}$ by exact sequences:
\begin{eqnarray*}
&\Hom_\Lambda(-,C_0)\xrightarrow{\cdot f_0}\Hom_\Lambda(-,X)\to F\to0,&\\
&\Hom_\Lambda(C_{n-1},-)\xrightarrow{f_n\cdot}\Hom_\Lambda(Y,-)\to G\to0.&
\end{eqnarray*}
Using \eqref{vanish C and P}, it is easily checked that
\begin{equation}\label{F and G}
F\iso D(G\circ\tau_n)\ \mbox{ and }\ G\iso D(F\circ\tau_n^-)
\end{equation}
hold (see \cite[Lemma~3.2]{I1} for details). Thus $F$ is a simple functor if and only $G$ is a simple functor. Now note that (b) (respectively, (c)) is equivalent to $F$ (respectively, $G$) being a simple functor.
Thus all conditions are equivalent.

Moreover, since $X$ (respectively, $Y$) is a unique indecomposable object in $\CC^0$ such that $F(X)\neq0$ (respectively, $G(Y)\neq0$),
the isomorphisms \eqref{F and G} imply $Y\iso\tau_n(X)$ and $X\iso\tau_n^-(Y)$.
\end{proof}

We denote by $\Sub\PP$ the full subcategory consisting of all submodules of modules in $\PP$.
Then we have the following easy observation.

\begin{lemma}\label{approximation0}
$\PP$ is contravariantly finite in $\Sub\PP$.
\end{lemma}

\begin{proof}
Let $M\in\PP$ and $N$ a submodule of $M$.
By Proposition~\ref{modules under nakayama}(b), all but a finite number of $X\in\ind\PP$ satisfy $\Hom_\Lambda(X,M)=0$ and so $\Hom_\Lambda(X,N)=0$.
This immediately implies that $N$ has a right $\PP$-approximation.
\end{proof}

The following result is an analogue of \cite[Theorem 2.2.3]{I1}.

\begin{proposition}\label{approximation}
For any $X\in\Sub\PP$, there exists an exact sequence
\[0\to C_{n-1}\xrightarrow{f_{n-1}}\cdots\xrightarrow{f_1}C_0\xrightarrow{f_0}X\to 0\]
with $C_i\in\PP$ for any $i$ and $f_i\in J_{\PP}$ for $i\neq0$ such that the following sequence is exact on $\CC^{0}$.
\[0\to\Hom_\Lambda(-,C_{n-1})\xrightarrow{\cdot f_{n-1}}\cdots\xrightarrow{\cdot f_1}\Hom_\Lambda(-,C_0)\xrightarrow{\cdot f_0}\Hom_\Lambda(-,X)\to 0.\]
\end{proposition}

\begin{proof}
Applying Lemma~\ref{approximation0} repeatedly, we have an exact sequence
\[0\to X_{n-1}\xrightarrow{f_{n-1}}C_{n-2}\xrightarrow{f_{n-2}}\cdots\xrightarrow{f_1}C_0\xrightarrow{f_0}X\to 0\]
with $C_i\in\PP$ for any $i$ and $f_i\in J_{\PP}$ for $i\neq0$ such that the following sequence is exact on $\PP$.
\[0\to\Hom_\Lambda(-,X_{n-1})\xrightarrow{\cdot f_{n-1}}\Hom_\Lambda(-,C_{n-2})\xrightarrow{\cdot f_{n-2}}\cdots\xrightarrow{\cdot f_1}\Hom_\Lambda(-,C_0)\xrightarrow{\cdot f_0}\Hom_\Lambda(-,X)\to 0.\]
It is exact on $\CC^{0}$ since $\Hom_\Lambda(\RR\vee\II,\Sub\PP)=0$.
Using $\Ext^i_\Lambda(\PP\vee\II,C_j)=0$ for any $i$ with $0<i<n$, one can easily check that $X_{n-1}$ satisfies $\Ext^i_\Lambda(\PP\vee\II,X_{n-1})=0$ for any $i$ with $0<i<n$ (see \cite[2.2.1(2)]{I1} for details).
This means $X_{n-1}\in\CC^{0}$ by Proposition~\ref{P and I}(e).
Since $X_{n-1}$ is a submodule of $C_{n-2}\in\PP$, we have $X_{n-1}\in\PP$ by Theorem~\ref{trichotomy}.
\end{proof}

Now we are ready to prove Theorem~\ref{n-almost split}.

\begin{proof}[Proof of Theorem~\ref{n-almost split}]
We only prove the assertions for $X$ or $Y$ in $\ind \PP$.
The corresponding assertions for $X$ or $Y$ in $\ind \II$ are shown dually.

Fix $X\in\ind\PP$.
By Theorem~\ref{trichotomy} and Proposition~\ref{modules under nakayama}(b), all but finitely many $C\in\ind\CC^{0}$ satisfy $\Hom_\Lambda(C,X)=0$.
Thus there exists a sink map $f_0 \colon C_0\to X$ in $\CC^{0}$.

Now we assume that $X$ is non-projective. Then $f_0$ is clearly surjective.
We apply Proposition~\ref{approximation} for $\Kernel f_0\in\Sub\PP$ to get an exact sequence
\[0\to C_n\xrightarrow{f_n}C_{n-1}\xrightarrow{f_{n-1}}\cdots\xrightarrow{f_2}C_1\xrightarrow{}\Kernel f_0\to 0.\]
Combining this with an short exact sequence $0\to\Kernel f_0\to C_0\to X\to0$, we have an exact sequence
\[0\to C_n\xrightarrow{f_n}C_{n-1}\xrightarrow{f_{n-1}}\cdots\xrightarrow{f_2}C_1\xrightarrow{f_1}C_0\xrightarrow{f_0} X\to0.\]
By Lemma~\ref{self duality}, this is an $n$-almost split sequence in $\CC^0$ and satisfies $C_n\iso\tau_n(X)$.

Since $\tau_n$ gives a bijection $\ind\PP_P\to\ind\PP$, the above argument shows the existence of $n$-almost split sequences in $\CC^0$ starting at any $Y\in\ind\PP$. In particular, any $Y\in\ind\PP$ has a source map in $\CC^0$.
\end{proof}

Now we give the following application of Theorem~\ref{n-almost split}.

\begin{proposition}\label{not contra}
\begin{itemize}
\item[(a)] No module in $\CC^{0}\backslash\PP$ has a right $\PP$-approximation.
In particular $\PP$ is not contravariantly finite in $\mod\Lambda$.
\item[(b)] No module in $\CC^{0}\backslash\II$ has a left $\II$-approximation.
In particular $\II$ is not covariantly finite in $\mod\Lambda$.
\end{itemize}
\end{proposition}

\begin{proof}
(a) Let $X\in\CC^{0}\backslash\PP$ be indecomposable and $f \colon Y\to X$ be a minimal right $\PP$-approximation.
By Theorem~\ref{n-almost split}, there exists a source map $g \colon Y\to Z$ in $\CC^{0}$ with $Z\in\PP$.
Then there exists $a \colon Z\to X$ such that $f=ga$.
Since $f$ is a right $\PP$-approximation, there exists $b \colon Z\to Y$ such that $a=bf$.
Now $f=gbf$ shows that $gb$ is an isomorphism, a contradiction to $g\in J_{\CC^0}$.

(b) is shown dually.
\end{proof}

Although Theorem~\ref{n-almost split} shows the existence of a sink (respectively, source) morphism of $X\in\PP\vee\II$ in $\CC^{0}$, it does not tell us anything about sink or source morphisms of $X\in\RR$ in $\CC^{0}$.

For the classical case $n=1$, we have $\CC^{0}=\mod\Lambda$, so we know that sink morphisms, source morphisms and almost split sequences exist by classical Auslander-Reiten theory.

For the case $n\ge 2$, we pose the following.

\begin{conjecture}
Let $n\ge 2$. Then no indecomposable $X\in\RR$ has a sink morphism or a source morphism in $\CC^{0}$.
\end{conjecture}

The following observation gives us some information related to this question.

\begin{observation}
\begin{itemize}
\item[(a)] Assume that some indecomposable $X\in\RR$ has a sink morphism $f \colon C\to X$ (respectively, source morphism $g \colon X\to C$) in $\CC^{0}$.
Then the sequence
\[0\to \Kernel f\to C\xrightarrow{f}X\to0\ (\mbox{respectively, }\ 0\to X\xrightarrow{g}C\to\Cokernel g\to0)\]
is an almost split sequence in $\RR$ and $\CC^0$.
\item[(b)] Let $n\ge 2$. Then no indecomposable $X \in \RR$ has an $n$-almost split sequence.
\end{itemize}
\end{observation}

\begin{proof}
(a) The case $n = 1$ is known, so we may assume $n \geq 2$. Applying $\Hom_\Lambda(\PP\vee\II,-)$, one can easily check that $\Ext_\Lambda^i(\PP\vee\II,\Kernel f)=0$ holds for any $i$ with $0<i<n$. Thus $\Kernel f\in\CC^0$.
By Theorem~\ref{trichotomy}, we have that $C$ and $\Kernel f$ belong to $\PP\vee\RR$. If $\Kernel f$ has a non-zero direct summand in $\PP$, then $\Ext^1_\Lambda(X,\PP)=0$ shows that $f$ is not right minimal, a contradiction.
Thus $\Kernel f\in\RR$. Since $\RR$ is an extension closed subcategory of $\RR$ by Proposition~\ref{extension closed}, we have $C\in\RR$.
By a standard argument (see \cite[Proposition~V.1.12]{ARS}), the sequence is an almost split sequence in $\RR$.
Since $\Ext^1_\Lambda(\PP\vee\II,\RR)=0=\Ext^1_\Lambda(\RR,\PP\vee\II)$, it is an almost split sequence in $\CC^0$.

(b) Assume there is an $n$-almost split sequence ending in $X$. Applying the exact sequence \eqref{sink sequence} of Definition~\ref{def.ASS} to $\Kernel f_0$, we see that $\Kernel f_0$ is a direct summand of $C_1$. Now the assumption $f_i \in J_{\CC^0}$ implies $C_1 = \Kernel f_0$, and $C_2 = C_3 = \ldots = Y = 0$. A contradiction.
\end{proof}

\begin{example}
Let $\Lambda$ be a Beilinson algebra of dimension $2$ (see Example~\ref{beilinson dim2 part1}).
Then the quiver of the categories $\PP$ and $\II$ are the following.
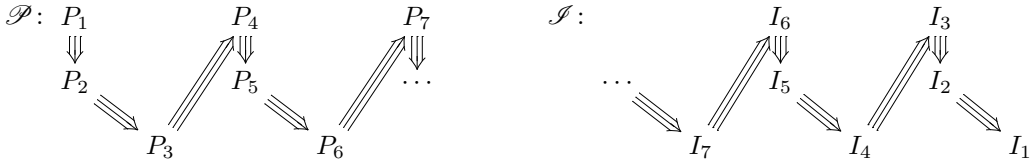
\[ \PP \colon \xymatrix@C=1.5em@R=1em{
P_1\ar@3{->}[d]&&P_4\ar@3{->}[d]&&P_7\ar@3{->}[d]\\
P_2\ar@3{->}[dr]&&P_5\ar@3{->}[dr]&&\cdots\\
&P_3\ar@3{->}[uur]&&P_6\ar@3{->}[uur]&
} \qquad \qquad
\II \colon \xymatrix@C=1.5em@R=1em{
&&I_6\ar@3{->}[d]&&I_3\ar@3{->}[d]&\\
\cdots\ar@3{->}[dr]&&I_5\ar@3{->}[dr]&&I_2\ar@3{->}[dr]&\\
&I_7\ar@3{->}[uur]&&I_4\ar@3{->}[uur]&&I_1}\]
Moreover we will see that the category $\RR$ is equivalent to $\coh_0\mathbb{P}^2$ in Example~\ref{example of description of R}(b).
\end{example}

\subsection{Preprojective algebras}

Let $\Lambda$ be an $n$-representation infinite algebra.
The \emph{$(n+1)$-preprojective algebra} (or simply \emph{preprojective algebra}) of $\Lambda$ is defined as the tensor algebra
\[\Pi=\Pi_{n+1}(\Lambda):=T_\Lambda\Ext^n_\Lambda(D\Lambda,\Lambda)\]
of the $\Lambda$-bimodule $\Ext^n_\Lambda(D\Lambda,\Lambda)$. Then $\Pi$ has a structure of a graded $K$-algebra by
\[\Pi=\bigoplus_{i\ge0}\Pi_i\ \mbox{ with }\ \Pi_i=
\underbrace{\Ext^n_\Lambda(D\Lambda,\Lambda)\otimes_\Lambda\cdots\otimes_\Lambda\Ext^n_\Lambda(D\Lambda,\Lambda)}_{i \text{ copies}}.\]
The name ``preprojective algebra'' is explained by the fact that $\Pi_i\iso\tau_n^{-i}(\Lambda)$ by \eqref{description of taun}, which means that $\Pi$ is the direct sum of all $n$-preprojective modules.

There is a strong connection between $\PP$ and the category $\gr\proj\Pi$ of finitely generated graded projedctive $\Pi$-modules. We define full subcategories of $\gr\proj\Pi$ by 
$\gr\proj_{\ge0}\Pi:=\add\{\Pi(i) \mid i\ge0\}$ and
$\gr\proj_{\le0}\Pi:=\add\{\Pi(i) \mid i\le0\}$.

\begin{proposition}\label{U and grproj}
\begin{itemize}
\item[(a)] We have an equivalence $\UU \equi \gr\proj\Pi$ which restricts to an equivalence $\PP \equi \gr\proj_{\ge0}\Pi$.
\item[(b)] We have an equivalence $\UU^{\op} \equi \gr\proj\Pi$ which restricts to an equivalence $\Hom_\Lambda(-,\Pi) \colon \PP^{\op} \equi \gr\proj_{\le0}\Pi^{\op}$.
\end{itemize}
\end{proposition}

\begin{proof}
(a) The correspondence $\nu_n^{-i}(\Lambda)\mapsto\Pi(i)$ gives an equivalence since 
$\Hom_\Lambda(\nu_n^{-i}(\Lambda),\nu_n^{-j}(\Lambda))\iso\tau_n^{-(j-i)}(\Lambda)\iso
\Pi_{j-i}\iso\Hom_{\gr\proj\Pi}(\Pi(i),\Pi(j))$.

(b) This is shown dually. The description as $\Hom$-functor follows since $\Hom_\Lambda(\nu_n^{-i}(\Lambda),\Pi)\iso\Pi(-i)$.
\end{proof}

Now we show that the homological behaviour of $\Pi$ is very nice if $\Lambda$ is $n$-representation infinite. The following notion plays a key role.

\begin{definition}\label{CY of GP}
Let $\Gamma=\bigoplus_{i\ge0}\Gamma_i$ be a positively graded $K$-algebra satisfying $\dim_K\Gamma_i<\infty$ for any $i\in\Z$.
We say that $\Gamma$ is a \emph{bimodule $\ell$-Calabi-Yau algebra of Gorenstein parameter $a$}
if $\Gamma\in \KKK^{\rm b}(\gr\proj \Gamma^{\rm e})$ and $\RHom_{\Gamma^{\rm e}}(\Gamma,\Gamma^{\rm e})\iso\Gamma[-\ell](a)$ in $\DDD(\Gr\Gamma^{\rm e})$.
\end{definition}

This is a graded version of Ginzburg's bimodule $(n+1)$-Calabi-Yau algebras \cite{G}. In particular, their derived categories satisfy the following property.

\begin{proposition}[{\cite[Lemma 4.1]{K1}}]\label{CY algebra gives CY triangulated}
Let $\Gamma$ be a bimodule $\ell$-Calabi-Yau algebra.
Then we have a functorial isomorphism
\[\Hom_{\DDD(\Mod\Gamma)}(X,Y)\simeq D\Hom_{\DDD(\Mod\Gamma)}(Y,X[\ell])\]
for any $X\in\DDD(\Mod\Gamma)$ and $Y\in\DDD^{\bo}_{\fd\Gamma}(\Mod\Gamma)$,
where $\DDD^{\bo}_{\fd\Gamma}(\Mod\Gamma)$ is the full subcategory of $\DDD^{\bo}(\Mod\Gamma)$ consisting of objects whose cohomologies are finite dimensional.
In particular, $\DDD^{\bo}_{\fd\Gamma}(\Mod\Gamma)$ is an $\ell$-Calabi-Yau triangulated category.
\end{proposition}

Recently the connection between $n$-representation infinite algebras and bimodule $(n+1)$-Calabi-Yau algebras of Gorenstein parameter $1$ 
was studied in representation theory \cite{K2,AIR} and non-commutative algebraic geometry \cite{MM}.
The main results are summarized as follows.

\begin{theorem}\label{preprojective correspondence}\cite{K2,AIR,MM}
Assume that $K$ is perfect.
There is a one-to-one correspondence between
isomorphism classes of $n$-representation infinite algebras $\Lambda$ and
isomorphism classes of graded bimodule $(n+1)$-Calabi-Yau algebras $\Gamma$ of Gorenstein parameter $1$. The correspondence is given by
\[\Lambda\mapsto\Gamma:=\Pi_{n+1}(\Lambda)\ \mbox{ and }\ \Gamma\mapsto\Lambda:=\Gamma_0.\]
\end{theorem}

\begin{proof}
(i) Let $\Gamma$ be a bimodule $(n+1)$-Calabi-Yau algebra of Gorenstein parameter $1$.

By \cite[Theorem 3.4]{AIR} (see also \cite[Theorem 4.2]{MM}), we have that $\Lambda:=\Gamma_0$ is an $n$-representation infinite algebra
such that $\Pi_{n+1}(\Lambda)$ is isomorphic to $\Gamma$ as a graded $K$-algebra.

(ii) For any $n$-representation infinite algebra, we will show that $\Pi_{n+1}(\Lambda)$ is a bimodule $(n+1)$-Calabi-Yau algebra of Gorenstein parameter $1$.

Since $K$ is perfect, we have $\Lambda\in \KKK^{\rm b}(\proj \Lambda^{\rm e})$.
In \cite{K2}, Keller introduced a DG algebra ${\bf\Pi}_{n+1}(\Lambda)$ called a derived preprojective algebra. Our $\Pi_{n+1}(\Lambda)$ is the 0-th cohomology of ${\bf\Pi}_{n+1}(\Lambda)$. By Keller's general result \cite[Theorem 4.8]{K2}, we have that ${\bf\Pi}_{n+1}(\Lambda)$ is a bimodule $(n+1)$-Calabi-Yau algebra. It is easy to check that ${\bf\Pi}_{n+1}(\Lambda)$ has Gorenstein parameter $1$. Since $\Lambda$ is $n$-representation infinite, we have that ${\bf\Pi}_{n+1}(\Lambda)$ is concentrated to degree zero, so ${\bf\Pi}_{n+1}(\Lambda)$ is quasi-isomorphic to $\Pi_{n+1}(\Lambda)$. Consequently we have the desired assumption.
\end{proof}

In Section~\ref{section: NCAG} the question if $\Pi$ is graded coherent will be of importance. Let us recall this notion.

\begin{definition}
We say that a graded ring $\Gamma$ is \emph{left graded coherent} (respectively, \emph{right graded coherent}) if the following equivalent conditions are satisfied.
\begin{itemize}
\item[(a)] The category of finitely presented graded $\Gamma$-modules (respectively, $\Gamma^{\rm op}$-modules) are closed under kernels.
\item[(b)] For any homogeneous homomorphism $f \colon P'\to P$ of finitely generated graded projective $\Gamma$-modules (respectively, $\Gamma^{\rm op}$-modules),
$\Kernel f$ is finitely generated.
\end{itemize}
\end{definition}

Now we can ask the question under which assumptions $\Pi$ is graded coherent. This is important for the non-commutative algebraic-geometric approach in Section~\ref{section: NCAG}, and has already been studied by Minamoto in \cite{M,MM}.

\begin{question}\label{graded coherernt}
When is $\Pi$ left graded coherent?
\end{question}

We give a representation theoretic interpretation of Question~\ref{graded coherernt} (cf.\ Proposition~\ref{not contra}).

\begin{proposition}\label{coherent}
Let $\Lambda$ be an $n$-representation infinite algebra and $\Pi$ be its preprojective algebra.
\begin{itemize}
\item[(a)] $\PP$ is covariantly finite in $\mod\Lambda$ if and only if $\Pi$ is right graded coherent.
\item[(b)] $\II$ is contravariantly finite in $\mod\Lambda$ if and only if $\Pi$ is left graded coherent.
\end{itemize}
\end{proposition}

\begin{proof}
We only prove (a).

By definition $\PP$ is covariantly finite if for any $X \in \mod \Lambda$ the covariant functor $\Hom_{\Lambda}(X, -) \colon \PP \to \mod k$ is finitely generated.
By the equivalence of Proposition~\ref{U and grproj}(b) this is equivalent to $\Hom_{\Lambda}(X, \Pi)$ being finitely generated as graded $\Pi^{\rm op}$-module.

By definition $\Pi$ is graded right coherent if any morphism in $\gr \proj \Pi^{\rm op}$ has a finitely generated kernel. Since degree shift is an autoequivalence of $\Gr \Pi^{\rm op}$ this is equivalent to asking that any morphism $P \to Q$ with $P, Q \in \gr \proj_{\le0} \Pi^{\rm op}$ has a finitely generated kernel. 

Now, by Proposition~\ref{U and grproj}(b), such $P$ and $Q$ can be written as $\Hom_{\Lambda}(P', \Pi)$ and $\Hom_{\Lambda}(Q', \Pi)$ for some $P'$ and $Q'$ in $\PP$, respectively. Moreover any map $P \to Q$ is of the form $\Hom_{\Lambda}(f, \Pi)$ for some $f \colon Q' \to P'$. It follows that the kernel of such a map is $\Hom_{\Lambda}(\Cokernel f, \Pi)$.

Since any $X \in \mod \Lambda$ can be realized as a cokernel of a morphism in $\PP$, we have that $\Pi$ is graded right coherent if and only if for any $X \in \mod \Lambda$ the graded $\Pi^{\rm op}$-module $\Hom_{\Lambda}(X, \Pi)$ is finitely generated.
\end{proof}

As an immediate consequence for the case $n=1$, we recover the following observation \cite{M}.

\begin{corollary}
The classical preprojective algebra is left and right graded coherent
\end{corollary}

\begin{proof}
Since $\PP$ is covariantly finite and $\II$ is contravariantly finite for the case $n=1$, we have the assertion immediately from Proposition~\ref{coherent}.
\end{proof}

\section{$n$-representation infinite algebras of type $\widetilde{A}$} \label{sect.Atilde}
In this section we assume that $K$ is algebraically closed of characteristic zero. Let $S = K[x_0, \ldots, x_{n}]$. Given a finite subgroup $H <\SL_{n+1}(K)$, the skew group algebra $S * H$ is bimodule $(n+1)$-Calabi-Yau \cite{BSW}. We shall now present these algebras by quivers with relations in a uniform way for the case when $H$ is abelian. Our presentation is based on the well-known description by McKay quivers and relations (see e.g.\ \cite{Re,RV,BSW}).
The difference is that we construct an algebra independent of $H$ that will have $S * H$ as an orbit algebra.

To construct the above mentioned algebra we will use the notation of root systems of type $A_n$, which we now introduce. For details see \cite{Hu}. Let $V = \{v \in \R^{n+1} \mid \sum_{i = 0}^{n} v_i = 0\}$. Consider the root system of type $A_n$ in $V$
\[
\Phi = \{e_i-e_j \mid 0 \le i \neq j \le n\}.
\]
As simple roots we take $\alpha_i = e_i - e_{i-1}$, where $1 \le i \le n$. We distinguish one additional root $\alpha_0 = e_{0} - e_n$. Thus we have the relation $\sum_{i=0}^{n}\alpha_i =0$. The root lattice $L$ is the lattice of vectors with integer coordinates in $V$. It is freely generated as an abelian group by the simple roots.

\begin{construction}
Define the quiver $Q$ as follows. The set of vertices in $Q$ is the root lattice $Q_0:=L$. The arrows are
\[
Q_1 := \{a^v_i \colon v \to (v+\alpha_i) \mid v \in Q_0, 0\le i \le n\}.
\]
For the sake of brevity we shall sometimes remove the superscript and write simply $a_i \colon v \to (v+\alpha_i)$ for the arrow $a^v_i $. This should cause no confusion since any $i_1,\dots, i_{m} \in \{0,\dots, n\}$ and $v \in Q_0$ determine a unique path of the form $a_{i_1} \cdots a_{i_m}$ in $Q$ from $v$ to $v + \sum_{j=1}^m \alpha_{i_j}$.

For each $v \in Q_0$ and $0 \le i < j \le n$ we define the a relation $r^v_{ij}$ from $v$ to $v + \alpha_i + \alpha_j$ by $r^v_{ij} := a_ia_j - a_ja_i$. We set
\[
\Gamma = KQ/\langle r^v_{ij} \mid v \in Q_0, 0\le i < j \le n \rangle.
\]
Be aware that $\Gamma$ does not have a unit element. 
\end{construction}

For $n = 1$, the quiver $Q$ is
\[
\cdots
\xymatrix{{\scriptstyle (-2,2)} \ar@/^/[r] &{\scriptstyle (-1,1)} \ar@/^/[r]\ar@/^/[l] &{\scriptstyle (0,0)}\ar@/^/[l] \ar@/^/[r] &{\scriptstyle (1,-1)} \ar@/^/[r]\ar@/^/[l] &{\scriptstyle (2,-2)}\ar@/^/[l]}
\cdots,
\]
and the algebra $\Gamma$ is the preprojective algebra of type $A_\infty^\infty$.

For $n=2$ the vertices $Q_0$ form a triangular lattice in the plane, $Q$ is
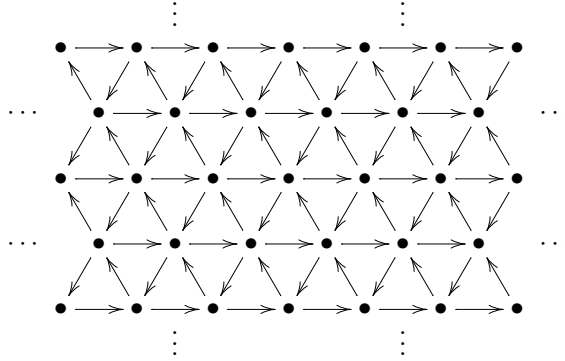
\[
\begin{xy} 0;<1cm,0cm>:<0.5cm,\halfrootthree cm>:: 
(0,0) *+{\bullet} ="00",
(1,0) *+{\bullet} ="10",
(2,0) *+{\bullet} ="20",
(3,0) *+{\bullet} ="30",
(4,0) *+{\bullet} ="40",
(5,0) *+{\bullet} ="50",
(6,0) *+{\bullet} ="60",
(0,1) *+{\bullet} ="01",
(1,1) *+{\bullet} ="11",
(2,1) *+{\bullet} ="21",
(3,1) *+{\bullet} ="31",
(4,1) *+{\bullet} ="41",
(5,1) *+{\bullet} ="51",
(-1,2) *+{\bullet} ="m12",
(0,2) *+{\bullet} ="02",
(1,2) *+{\bullet} ="12",
(2,2) *+{\bullet} ="22",
(3,2) *+{\bullet} ="32",
(4,2) *+{\bullet} ="42",
(5,2) *+{\bullet} ="52",
(-1,3) *+{\bullet} ="m13",
(0,3) *+{\bullet} ="03",
(1,3) *+{\bullet} ="13",
(2,3) *+{\bullet} ="23",
(3,3) *+{\bullet} ="33",
(4,3) *+{\bullet} ="43",
(-2,4) *+{\bullet} ="m24",
(-1,4) *+{\bullet} ="m14",
(0,4) *+{\bullet} ="04",
(1,4) *+{\bullet} ="14",
(2,4) *+{\bullet} ="24",
(3,4) *+{\bullet} ="34",
(4,4) *+{\bullet} ="44",
(-1,1) *+{\cdots} ="h01",
(-2,3) *+{\cdots} ="h02",
(6,1) *+{\cdots} ="h11",
(5,3) *+{\cdots} ="h12",
(1.7,-0.4) *+{\vdots} ="v10",
(4.7,-0.4) *+{\vdots} ="v20",
(-0.825,4.65) *+{\vdots} ="v11",
(2.175,4.65) *+{\vdots} ="v21",
"00", {\ar"10"},
"10", {\ar"20"},
"10", {\ar"01"},
"20", {\ar"30"},
"20", {\ar"11"},
"30", {\ar"40"},
"30", {\ar"21"},
"40", {\ar"50"},
"40", {\ar"31"},
"50", {\ar"60"},
"50", {\ar"41"},
"60", {\ar"51"},
"01", {\ar"00"},
"01", {\ar"11"},
"01", {\ar"m12"},
"11", {\ar"10"},
"11", {\ar"21"},
"11", {\ar"02"},
"21", {\ar"20"},
"21", {\ar"31"},
"21", {\ar"12"},
"31", {\ar"30"},
"31", {\ar"41"},
"31", {\ar"22"},
"41", {\ar"40"},
"41", {\ar"51"},
"41", {\ar"32"},
"51", {\ar"50"},
"51", {\ar"42"},
"m12", {\ar"02"},
"02", {\ar"01"},
"02", {\ar"12"},
"02", {\ar"m13"},
"12", {\ar"11"},
"12", {\ar"22"},
"12", {\ar"03"},
"22", {\ar"21"},
"22", {\ar"32"},
"22", {\ar"13"},
"32", {\ar"31"},
"32", {\ar"42"},
"32", {\ar"23"},
"42", {\ar"41"},
"42", {\ar"52"},
"42", {\ar"33"},
"52", {\ar"51"},
"52", {\ar"43"},
"m13", {\ar"m12"},
"m13", {\ar"03"},
"m13", {\ar"m24"},
"03", {\ar"02"},
"03", {\ar"13"},
"03", {\ar"m14"},
"13", {\ar"12"},
"13", {\ar"23"},
"13", {\ar"04"},
"23", {\ar"22"},
"23", {\ar"33"},
"23", {\ar"14"},
"33", {\ar"32"},
"33", {\ar"43"},
"33", {\ar"24"},
"43", {\ar"42"},
"43", {\ar"34"},
"m24", {\ar"m14"},
"m14", {\ar"m13"},
"m14", {\ar"04"},
"04", {\ar"03"},
"04", {\ar"14"},
"14", {\ar"13"},
"14", {\ar"24"},
"24", {\ar"23"},
"24", {\ar"34"},
"34", {\ar"33"},
"34", {\ar"44"},
"44", {\ar"43"},
\end{xy}
\]
and each arrow corresponds to a commutativity relation.

\medskip
From $\Gamma$ we shall construct several finite dimensional algebras in two steps. First we take certain orbit algebras, and then we factor out ideals generated by certain sets of arrows called cuts. For $n=1$, the orbit algebra will be preprojective of type $\widetilde{A}$ and the factor algebra will be a path algebra of a quiver of type $\widetilde{A}$.

First we define the orbit algebras. Since $L= Q_0$ is an abelian group it acts on itself translations. This action extends in a unique way to a $L$-action on $Q$.

\begin{construction}
Let $B$ be a cofinite subgroup of $L$, i.e.\ the factor group $L / B$ is finite. This is equivalent to requiring that ${\rm rank} B = {\rm rank} L (= n)$.
We denote by $\Gamma/B$ the orbit algebra of $\Gamma$ with respect to the $B$-action.
More explicitly, $\Gamma/B$ is the path algebra of $Q/B = (Q_0/B, Q_1/B)$ with relations $\overline{a}_i\overline{a}_j - \overline{a}_j\overline{a}_i$ from $\overline{v}$ to $\overline{v} + \overline{\alpha}_i + \overline{\alpha}_j$ for each $\overline{v} \in Q_0/B$ and $0 \le i < j \le n$ (here $\overline{a}$ and $\overline{v}$ denotes the $B$-orbits of $a \in Q_1$ and $v\in Q_0$ respectively).
\end{construction}

Observe that if $B = L$, then $\Gamma/B \iso S=K[x_0,\ldots,x_{n}]$.
We now show that for arbitrary $B$ the algebra $\Gamma/B$ is isomorphic to a skew group algebra over $S$.

\begin{lemma}\label{skew Atilde}
There exists a finite abelian subgroup $H$ of $\SL_{n+1}(K)$ such that $S * H \iso \Gamma/B$. 
\end{lemma}

\begin{proof}
Let $H$ be the group of group morphisms $\phi \colon L \to K^{\times}$ satisfying $\phi(B) = 1$. Since $L/B$ is finite and $K^{\times}$ contains primitive $q$-th roots of unity for every positive integer $q$ we have a nondegenerate bihomomorphism
\[
H \times L/B \to K^{\times}, \ (\phi,\overline{v}) \mapsto \phi(v)
\]
that allows us to identify the elements of $L/B$ with the isoclasses of irreducible representations of $H$ over $K$.
Define an embedding $H \to \SL_{n+1}(K)$ by sending $\phi \in H$ to the diagonal matrix with diagonal entries $\phi(\alpha_0), \ldots, \phi(\alpha_{n})$. This is well-defined because
\[
\prod_{i=0}^{n} \phi(\alpha_i) = \phi\left( \sum_{i=0}^{n} \alpha_i \right) = \phi(0) = 1.
\]
The above embedding defines an $H$-action on $S$, which satisfies $\phi x_i = \phi(\alpha_i)x_i$ for all $0\le i \le n$. Using our identification of $L/B$ with the isoclasses of irreducible representations of $H$ we get that $Q/B$ is the McKay quiver for $H$. Moreover, the relations in $\Gamma/B$ are just commutativity relations, so $S * H \iso \Gamma/B$.
\end{proof}

Our aim is to apply Theorem~\ref{preprojective correspondence} to obtain $n$-representation infinite algebras from $\Gamma/B$. To do this we need suitable gradings on $\Gamma/B$. We shall describe these gradings in a uniform way by returning our attention to the covering $\Gamma$.

\begin{definition}
For every permutation $\sigma$ of $\{0,\ldots, n\}$ and $v \in Q_0$ we get a cyclic path $a_{\sigma(0)} \cdots a_{\sigma(n)}$ in $Q$ starting and ending at $v$.
We call such paths \emph{small cycles}.
A subset $C \in Q_1$ is called a \emph{cut} if it contains precisely one arrow from each small cycle.
\end{definition}

Every cut $C$ defines a grading $g_C$ on $Q$ by
\[
g_C(a) = \begin{cases} 
1 & a \in C \\
0 & a \in Q_1\setminus C
\end{cases}
\]
The relation $r^v_{ij}$ is homogeneous with respect to $g_C$ since the element 
\[
r^v_{ij}a_0\cdots a_{i-1}a_{i+1}\cdots a_{j-1}a_{j+1}\cdots a_{n}
\]
is a difference of two small cycles.
So $g_C$ induces a grading on $\Gamma$.
Set $Q_C = (Q_0,Q_1\setminus C)$, so that the degree zero part $\Gamma_0$ of $\Gamma$ is a factor of $KQ_C$.

To get an induced grading on the orbit algebra $\Gamma/B$ we need to ensure that the action of $B$ is compatible with the grading $g_C$, i.e., that $B$ is a subgroup of the stabilizer
\[
L_C := \{\lambda \in L \mid \lambda C = C\} \le L
\]
of $C$. We now focus on cuts satisfying the following two conditions.

\begin{definition}
We say that a cut $C$ in $Q$ is
\begin{itemize}
\item[(a)] \emph{periodic} if $L_C$ is cofinite in $L$.
\item[(b)] \emph{bounding} if there is a natural number $N$ such that all paths in $Q_C$ have length at most $N$.
\end{itemize}
\end{definition}

Let $C$ be a periodic cut and $B \le L_C$ a cofinite subgroup. Then $B$ is cofinite in $L$ and we may consider $\Gamma/B$. 
Since $g_C$ is constant on the $B$-orbits of $Q_1$, it induces a grading $g_C$ on $\Gamma/B$.
We denote by $(\Gamma/B)_C$ the degree $0$ part of $(\Gamma/B,g_C)$.
Notice that if $C$ is also bounding, then all graded parts of $(\Gamma/B,g_C)$ are finite dimensional.

\begin{theorem}\label{CY Atilde}
Let $C$ be a bounding periodic cut in $Q$ and $B \le L_C$ be cofinite.
\begin{itemize}
\item[(a)] $(\Gamma/B,g_C)$ is bimodule $(n+1)$-Calabi-Yau of Gorenstein parameter $1$.
\item[(b)] $(\Gamma/B)_C$ is an $n$-representation infinite algebra whose preprojective algebra is $\Gamma/B$.
\end{itemize}
\end{theorem}

We call the algebras $(\Gamma/B)_C$ in Theorem~\ref{CY Atilde}, \emph{$n$-representation infinite of type $\widetilde{A}$}.

\begin{proof}
By Lemma~\ref{skew Atilde}, there is a subgroup $H$ of $\SL_{n+1}(K)$ such that $\Gamma/B \iso S * H $. This implies that $\Gamma/B$ has a canonical projective $(\Gamma/B)^{\rm e}$-resolution that we now recall following \cite{BSW} (compare also with \cite{AIR}).

Let $R$ be the semisimple subalgebra $K(Q_0/B)$ of $\Gamma/B$. We consider $\Gamma/B$ and $K(Q/B)$ as $R^{\rm e}$-modules. For every $0 \le m \le n+1$ we denote by $K(Q_m/B)$, the $R^{\rm e}$-submodule of $K(Q/B)$, spanned by all paths of length $m$. Let $\overline{v} \in Q_0/B$, $0 \le i_1 < \ldots < i_{m} \le n$ and set $I = \{i_1,\ldots,i_m\}$. For every $\sigma \in S_m$ let $w^{\overline{v}}_I(\sigma)$ be the path $\overline{a}_{i_{\sigma(1)}} \cdots \overline{a}_{i_{\sigma(m)}}$ from $\overline{v}$ to $\overline{v}^I:=\overline{v} + \sum_{j=1}^m \overline{\alpha}_{i_j}$ and set
\[
w^{\overline{v}}_I := \sum_{\sigma \in S_m} \sign(\sigma)w^{\overline{v}}_I({\sigma}).
\]
Let $W_m$ be the span of $\{w^{\overline{v}}_I \mid \overline{v} \in Q_0/B, |I| = m\}$ in $K(Q/B)$. Observe that $W_m$ is in fact a $R^{\rm e}$-submodule of $K(Q_m/B)$.

Define 
\[
d_m \colon  \Gamma/B \otimes_R K(Q_m/B) \otimes_R \Gamma/B \to \Gamma/B \otimes_R K(Q_{m-1}/B) \otimes_R \Gamma/B
\]
by
\[
d_m(x\otimes b_1\cdots b_m \otimes y) = x b_1\otimes b_2\cdots b_m \otimes y +(-1)^m x\otimes b_1\cdots b_{m-1} \otimes b_m y.
\]

Next define the $(\Gamma/B)^{\rm e}$-module complex $P$ as
\[
\Gamma/B \otimes_R W_{n+1} \otimes_R \Gamma/B \xrightarrow{d_{n+1}} \cdots \xrightarrow{d_{1}} \Gamma/B \otimes_R W_{0} \otimes_R \Gamma/B.
\]
By \cite{BSW} the complex $P$ is a projective resolution of $\Gamma/B \in \Mod (\Gamma/B)^{\rm e}$. Let $P^*$ be the $(\Gamma/B)^{\rm e}$-dual of $P$. Then there is an isomorphism $\varphi \colon P[-n-1] \to P^*$ where 
\[
\varphi_m \colon \Gamma/B \otimes_R W_{n+1-m} \otimes_R \Gamma/B \to \Hom_{(\Gamma/B)^{\rm e}}(\Gamma/B \otimes_R W_{m} \otimes_R \Gamma/B, (\Gamma/B)^{\rm e})
\]
satisfies 
\[
(\varphi_m(1 \otimes w^{\overline{v}}_I \otimes 1))(1 \otimes w^{\overline{v}'}_J \otimes 1) = 
\begin{cases} 
\pm e_{\overline{v}^{I}} \otimes e_{\overline{v}} & \text{ if $\overline{v}' = \overline{v}^I$ and $J = \{0,\ldots,n\}\setminus I$} \\ 
0 & \text{ otherwise.} 
\end{cases} 
\]

Now we take the grading into account. We consider $K(Q/B)$ to be graded by $g_C$. Each of the paths $w^{\overline{v}}_I({\sigma})$ can be completed into a small cycle and so has degree $1$ or $0$. Moreover, $w^{\overline{v}}_I({\sigma})$ does not depend on $\sigma$ as an element in $\Gamma/B$ and so neither does its degree. Hence $w^{\overline{v}}_I$ is homogeneous of degree $1$ or $0$. Thus $W_m$ is a graded subspace of $K(Q/B)$ and so each term in $P$ is a graded $(\Gamma/B)^{\rm e}$-module. The differential $d_m$ is homogeneous of degree $0$ and thus $P$ is in fact a complex in $\Gr (\Gamma/B)^{\rm e}$. Moreover, $\varphi$ is homogeneous of degree $-1$ so 
$\RHom_{(\Gamma/B)^{\rm e}}(\Gamma/B,(\Gamma/B)^{\rm e}) \iso \Gamma/B[-n-1](1)$. Hence, $\Gamma/B$ is  $(n+1)$-Calabi-Yau of Gorenstein parameter $1$. 

By Theorem~\ref{preprojective correspondence}, $(\Gamma/B)_C$ is $n$-representation infinite.
\end{proof}

\begin{example}
Let $n=1$. Then $\Gamma$ is the preprojective algebra of type $A_\infty^\infty$. For any cut $C$ the quiver $Q_C$ is a quiver of type $A_\infty^\infty$. In fact, the map $C \mapsto Q_C$ provides a bijection between cuts and orientations of $A_\infty^\infty$. Moreover, $C$ is periodic if and only if $Q_C$ has periodic orientation. Hence the algebras $(\Gamma/B)_{C}$, for $C$ periodic, are precisely the path algebras of quivers of type $\widetilde{A}$. Also $C$ is bounding if and only if $(\Gamma/B)_C$ is finite dimensional. Thus the $1$-representation infinite algebras of type $\widetilde{A}$ are precisely path algebras of quivers of type $\widetilde{A}$ with acyclic orientation.
\end{example}

For $n=1$, the two alternating orientations of $A_\infty^\infty$ are special in that they allow no paths of length $2$. We obtain these orientations by choosing the cut $C$ to consist of all arrows that start at the vertices of the form $(2t,-2t)$ or $(2t+1,-2t-1)$, respectively. We proceed to generalize these cuts to arbitrary $n$. 

\begin{example}
Define the group morphism $\omega \colon  L \to \Z/(n+1)\Z$ by $\omega(\alpha_i) = 1$ for any $i$. Let $k \in \Z/(n+1)\Z$ and set
\[
C_k = \{a_i \colon v \to (v+\alpha_i) \mid \omega(v) = k, \ 0\le i \le n\}
\]
Let $i_1,\ldots, i_{n+1} \in \{0 ,\ldots, n\}$, $v \in Q_0$ and consider the path $p = a_{i_1} \cdots a_{i_{n+1}}$ from $v$ to $v + \sum_{j=1}^{n+1} \alpha_{i_j}$. It passes the vertices $v^m := v + \sum_{j=1}^m \alpha_{i_j}$ where $0 \le m \le n+1$. Since $\omega(v^m) = \omega(v) +m$, there is precisely one $0 \le m \le n$ such that $\omega(v^m) = k$. Hence precisely one of the arrows in $p$ is in $C_k$. In particular, $C_k$ is a bounding cut. In fact, $C_k$ is extremely bounding in the sense that all paths in $Q_{C_k}$ have length at most $n$, which is the smallest possible bound for any cut. For each $v \in Q_0$ and $\lambda \in L$ we have  $\omega(\lambda v) = \omega(v) + \omega(\lambda)$ and so $\lambda C_k = C_{k + \omega(\lambda)}$. Thus $L_{C_k} = \ker \omega$. In particular, $L_{C_k}$ is cofinite and $C_k$ is periodic. 

For $n = 1$ the cuts $C_0$ and $C_1$ give rise to the two alternating orientations:
\[
\cdots
\xymatrix{\scriptstyle (-2, 2) \cut@/^/[r] & \scriptstyle (-1, 1) \ar@/^/[r]\ar@/^/[l] & \scriptstyle (0, 0) \cut@/^/[l] \cut@/^/[r] & \scriptstyle (1, -1) \ar@/^/[r]\ar@/^/[l] & \scriptstyle (2, -2) \cut@/^/[l]}
\cdots
\]
and
\[
\cdots
\xymatrix{\scriptstyle (-2, 2) \ar@/^/[r] & \scriptstyle (-1,1) \cut@/^/[r]\cut@/^/[l] & \scriptstyle (0,0) \ar@/^/[l] \ar@/^/[r] & \scriptstyle (1, -1) \cut@/^/[r]\cut@/^/[l] & \scriptstyle (2, -2) \ar@/^/[l]}
\cdots
\]
where the bold lines indicate arrows in $C_0$ and $C_1$ respectively. For $n=2$, $C_0$ is the following cut:
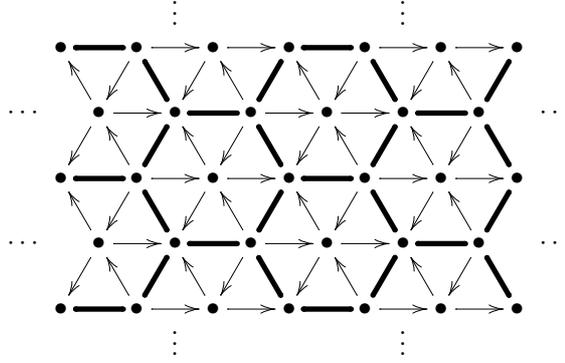
\[
\begin{xy} 0;<1cm,0cm>:<0.5cm,\halfrootthree cm>:: 
(0,0) *+{\bullet} ="00",
(1,0) *+{\bullet} ="10",
(2,0) *+{\bullet} ="20",
(3,0) *+{\bullet} ="30",
(4,0) *+{\bullet} ="40",
(5,0) *+{\bullet} ="50",
(6,0) *+{\bullet} ="60",
(0,1) *+{\bullet} ="01",
(1,1) *+{\bullet} ="11",
(2,1) *+{\bullet} ="21",
(3,1) *+{\bullet} ="31",
(4,1) *+{\bullet} ="41",
(5,1) *+{\bullet} ="51",
(-1,2) *+{\bullet} ="m12",
(0,2) *+{\bullet} ="02",
(1,2) *+{\bullet} ="12",
(2,2) *+{\bullet} ="22",
(3,2) *+{\bullet} ="32",
(4,2) *+{\bullet} ="42",
(5,2) *+{\bullet} ="52",
(-1,3) *+{\bullet} ="m13",
(0,3) *+{\bullet} ="03",
(1,3) *+{\bullet} ="13",
(2,3) *+{\bullet} ="23",
(3,3) *+{\bullet} ="33",
(4,3) *+{\bullet} ="43",
(-2,4) *+{\bullet} ="m24",
(-1,4) *+{\bullet} ="m14",
(0,4) *+{\bullet} ="04",
(1,4) *+{\bullet} ="14",
(2,4) *+{\bullet} ="24",
(3,4) *+{\bullet} ="34",
(4,4) *+{\bullet} ="44",
(-1,1) *+{\cdots} ="h01",
(-2,3) *+{\cdots} ="h02",
(6,1) *+{\cdots} ="h11",
(5,3) *+{\cdots} ="h12",
(1.7,-0.4) *+{\vdots} ="v10",
(4.7,-0.4) *+{\vdots} ="v20",
(-0.825,4.65) *+{\vdots} ="v11",
(2.175,4.65) *+{\vdots} ="v21",
"00", {\cut"10"},
"10", {\ar"20"},
"10", {\ar"01"},
"20", {\ar"30"},
"20", {\ar"11"},
"30", {\cut"40"},
"30", {\cut"21"},
"40", {\ar"50"},
"40", {\ar"31"},
"50", {\ar"60"},
"50", {\ar"41"},
"60", {\cut"51"},
"01", {\ar"00"},
"01", {\ar"11"},
"01", {\ar"m12"},
"11", {\cut"10"},
"11", {\cut"21"},
"11", {\cut"02"},
"21", {\ar"20"},
"21", {\ar"31"},
"21", {\ar"12"},
"31", {\ar"30"},
"31", {\ar"41"},
"31", {\ar"22"},
"41", {\cut"40"},
"41", {\cut"51"},
"41", {\cut"32"},
"51", {\ar"50"},
"51", {\ar"42"},
"m12", {\cut"02"},
"02", {\ar"01"},
"02", {\ar"12"},
"02", {\ar"m13"},
"12", {\ar"11"},
"12", {\ar"22"},
"12", {\ar"03"},
"22", {\cut"21"},
"22", {\cut"32"},
"22", {\cut"13"},
"32", {\ar"31"},
"32", {\ar"42"},
"32", {\ar"23"},
"42", {\ar"41"},
"42", {\ar"52"},
"42", {\ar"33"},
"52", {\cut"51"},
"52", {\cut"43"},
"m13", {\ar"m12"},
"m13", {\ar"03"},
"m13", {\ar"m24"},
"03", {\cut"02"},
"03", {\cut"13"},
"03", {\cut"m14"},
"13", {\ar"12"},
"13", {\ar"23"},
"13", {\ar"04"},
"23", {\ar"22"},
"23", {\ar"33"},
"23", {\ar"14"},
"33", {\cut"32"},
"33", {\cut"43"},
"33", {\cut"24"},
"43", {\ar"42"},
"43", {\ar"34"},
"m24", {\cut"m14"},
"m14", {\ar"m13"},
"m14", {\ar"04"},
"04", {\ar"03"},
"04", {\ar"14"},
"14", {\cut"13"},
"14", {\cut"24"},
"24", {\ar"23"},
"24", {\ar"34"},
"34", {\ar"33"},
"34", {\ar"44"},
"44", {\cut"43"},
\end{xy}
\]

In general, we have that $(\Gamma/B)_{C_k} \iso (\Gamma/B)_{C_0}$ for all cofinite $B \le \ker \omega$ and $k \in \Z/(n+1)\Z$, so we may fix $k = 0$. Let us consider the extreme case when $B = \ker \omega $. We identify $Q_0/B$ with $\Z/(n+1)\Z$ via $\omega$, and denote the arrows in $Q/B$ by $a^k_i\colon k \to k+1$, where $k \in \Z/(n+1)\Z$ and $i \in \{0,\ldots, n\}$. In $(\Gamma/B)_{C_0}$ we have the relations $a^k_ia^{k+1}_j = a^k_ja^{k+1}_i$. Furthermore $C_0/B$ is precisely the set of arrows $\{a^{0}_i \mid 0 \le i \le n\}$ and so $(\Gamma/B)_{C_0}$ is the Beilinson algebra:
\[\xymatrix{
1  \ar@/^1pc/[r]^{a^1_{0}}_{\scalebox{0.7}{\vdots}}\ar@/_ 1pc/[r]_{a^1_{n}} & 2 \ar@/^1pc/[r]^{a^2_{0}}_{\scalebox{0.7}{\vdots}}\ar@/_ 1pc/[r]_{a^2_{n}} & 3
} \cdots\xymatrix{
n \ar@/^1pc/[r]^{a^{n}_{0}}_{\scalebox{0.7}{\vdots}}\ar@/_ 1pc/[r]_{a^{n}_{n}}& n+1
},
\hspace{1cm}
a^k_ia^{k+1}_j = a^k_ja^{k+1}_i.
\]
Notice that if $n = 1$, then $(\Gamma/B)_{C_0}$ is the Kronecker algebra.
\end{example}

Observe that the quivers of $n$-representation infinite algebras of type $\widetilde{A}$ are acyclic. In fact, all examples of $n$-hereditary algebras that we are aware of share this property. Thus we end this section by the following experimental expectation.

\begin{question}
The quivers of $n$-hereditary algebras are acyclic.
\end{question}

\subsection{Relationship to $n$-representation finite algebras of type $A$}
In \cite{IO1} the class of $n$-representation finite algebras of type $A$ was introduced. We shall now compare it to the class of $n$-representation infinite algebras of type $\widetilde{A}$ defined above. Let $s$ be a non-negative integer and set
\[
L_{\circ} = \{(v_0,\ldots, v_{n}) \in Q_0 \mid v_i \ge 0 \mbox{ for all } 0 \le i \le n-1 \mbox{ and } v_{n} \ge -s\}.
\]
Notice that $L_{\circ}$ consists of the lattice points of an $n$-simplex in $V$ with corners $s(e_i-e_n)$. Let $I$ be the ideal in $\Gamma$ generated by the primitive idempotents corresponding to $L \setminus L_{\circ}$. Then the algebra $\Gamma^{\circ} := \Gamma/I$ is isomorphic to the algebra $\widehat{\Lambda}^{(n,s+1)}$ defined in \cite{IO1} (compare also \cite{HI1}). In particular, $\Gamma^{\circ}$ is a finite dimensional selfinjective algebra. For simplicity we will ignore the trivial case of $s=0$ and assume from now on that $s >0$.

Let $Q^{\circ}$ be the full subquiver of $Q$ with vertices $L_{\circ}$. Then $Q^{\circ}$ is the quiver of $\Gamma^{\circ}$. We call $C^{\circ} \subset Q^{\circ}_1$ a restricted cut if $C^{\circ}$ contains precisely one arrow from every small cycle contained in $Q^{\circ}$. As before $C^{\circ}$ defines a grading on $\Gamma^{\circ}$ and we denote the degree $0$ part by $\Gamma^{\circ} _{C^{\circ}}$. The algebras $\Gamma^{\circ} _{C^{\circ}}$ are (up to isomorphism) precisely the $n$-representation finite algebras of type $A$ defined in \cite{IO1}. Here is an example of $Q^{\circ}$ with a restricted cut $C^{\circ}$ indicated in bold for $n = 2$ and $s=2$.
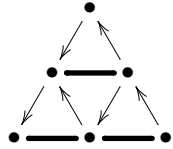
\[
\begin{xy} 0;<1cm,0cm>:<0.5cm,\halfrootthree cm>:: 
(1,2) *+{\bullet} ="12",
(2,2) *+{\bullet} ="22",
(3,2) *+{\bullet} ="32",
(1,3) *+{\bullet} ="13",
(2,3) *+{\bullet} ="23",
(1,4) *+{\bullet} ="14",
%
"12", {\cut"22"},
"22", {\cut"32"},
"22", {\ar"13"},
"32", {\ar"23"},
"13", {\ar"12"},
"13", {\cut"23"},
"23", {\ar"22"},
"23", {\ar"14"},
"14", {\ar"13"},
\end{xy}
\]
The sides of this triangle extend to lines in the plane. The affine reflections in these lines induce automorphisms of $Q$. Taking the union of all orbits of arrows in $C^{\circ}$ with respect to the group generated by these affine reflections we get the following cut $C$ in $Q$:
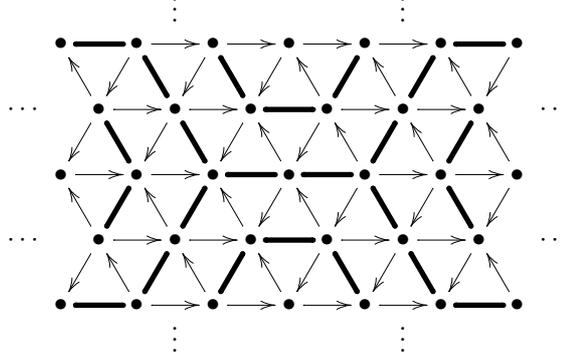
\[
\begin{xy} 0;<1cm,0cm>:<0.5cm,\halfrootthree cm>:: 
(0,0) *+{\bullet} ="00",
(1,0) *+{\bullet} ="10",
(2,0) *+{\bullet} ="20",
(3,0) *+{\bullet} ="30",
(4,0) *+{\bullet} ="40",
(5,0) *+{\bullet} ="50",
(6,0) *+{\bullet} ="60",
(0,1) *+{\bullet} ="01",
(1,1) *+{\bullet} ="11",
(2,1) *+{\bullet} ="21",
(3,1) *+{\bullet} ="31",
(4,1) *+{\bullet} ="41",
(5,1) *+{\bullet} ="51",
(-1,2) *+{\bullet} ="m12",
(0,2) *+{\bullet} ="02",
(1,2) *+{\bullet} ="12",
(2,2) *+{\bullet} ="22",
(3,2) *+{\bullet} ="32",
(4,2) *+{\bullet} ="42",
(5,2) *+{\bullet} ="52",
(-1,3) *+{\bullet} ="m13",
(0,3) *+{\bullet} ="03",
(1,3) *+{\bullet} ="13",
(2,3) *+{\bullet} ="23",
(3,3) *+{\bullet} ="33",
(4,3) *+{\bullet} ="43",
(-2,4) *+{\bullet} ="m24",
(-1,4) *+{\bullet} ="m14",
(0,4) *+{\bullet} ="04",
(1,4) *+{\bullet} ="14",
(2,4) *+{\bullet} ="24",
(3,4) *+{\bullet} ="34",
(4,4) *+{\bullet} ="44",
(-1,1) *+{\cdots} ="h01",
(-2,3) *+{\cdots} ="h02",
(6,1) *+{\cdots} ="h11",
(5,3) *+{\cdots} ="h12",
(1.7,-0.4) *+{\vdots} ="v10",
(4.7,-0.4) *+{\vdots} ="v20",
(-0.825,4.65) *+{\vdots} ="v11",
(2.175,4.65) *+{\vdots} ="v21",
"00", {\cut"10"},
"10", {\ar"20"},
"10", {\ar"01"},
"20", {\ar"30"},
"20", {\ar"11"},
"30", {\ar"40"},
"30", {\ar"21"},
"40", {\ar"50"},
"40", {\cut"31"},
"50", {\cut"60"},
"50", {\cut"41"},
"60", {\ar"51"},
"01", {\ar"00"},
"01", {\ar"11"},
"01", {\ar"m12"},
"11", {\cut"10"},
"11", {\ar"21"},
"11", {\ar"02"},
"21", {\cut"20"},
"21", {\cut"31"},
"21", {\ar"12"},
"31", {\ar"30"},
"31", {\ar"41"},
"31", {\ar"22"},
"41", {\ar"40"},
"41", {\ar"51"},
"41", {\cut"32"},
"51", {\ar"50"},
"51", {\cut"42"},
"m12", {\ar"02"},
"02", {\cut"01"},
"02", {\ar"12"},
"02", {\cut"m13"},
"12", {\cut"11"},
"12", {\cut"22"},
"12", {\cut"03"},
"22", {\ar"21"},
"22", {\cut"32"},
"22", {\ar"13"},
"32", {\ar"31"},
"32", {\ar"42"},
"32", {\ar"23"},
"42", {\ar"41"},
"42", {\ar"52"},
"42", {\ar"33"},
"52", {\ar"51"},
"52", {\ar"43"},
"m13", {\ar"m12"},
"m13", {\ar"03"},
"m13", {\ar"m24"},
"03", {\ar"02"},
"03", {\ar"13"},
"03", {\cut"m14"},
"13", {\ar"12"},
"13", {\cut"23"},
"13", {\cut"04"},
"23", {\ar"22"},
"23", {\ar"33"},
"23", {\ar"14"},
"33", {\cut"32"},
"33", {\ar"43"},
"33", {\ar"24"},
"43", {\cut"42"},
"43", {\ar"34"},
"m24", {\cut"m14"},
"m14", {\ar"m13"},
"m14", {\ar"04"},
"04", {\ar"03"},
"04", {\ar"14"},
"14", {\ar"13"},
"14", {\ar"24"},
"24", {\cut"23"},
"24", {\ar"34"},
"34", {\cut"33"},
"34", {\cut"44"},
"44", {\ar"43"},
\end{xy}
\]

We will show that any restricted cut $C^{\circ}$ can in a similar way be extended to a bounding periodic cut $C$ for arbitrary $n$ and $s$. Our aim is to apply this construction to prove that every $n$-representation finite algebra of type $A$ is a factor of an $n$-representation infinite algebra of type $\widetilde{A}$ by an ideal generated by some idempotent:

\begin{theorem}\label{A factor of Atilde}
Let $s$ be a non-negative integer and $C^{\circ} \subset Q^{\circ}_1$ a restricted cut. Then there is a bounding periodic cut $C$ in $Q$, a cofinite subgroup $B \le L_C$ and an idempotent $e \in K(Q_0/B)$ such that $(\Gamma/B)_C/(1-e)$ is isomorphic to $\Gamma^{\circ}_{C^{\circ}}$.
\end{theorem}

We start by generalizing the group of affine reflections. We consider the group of affine transformations on $V$ as a semidirect product $\Aff(V) = V \rtimes \GL(V)$, where the $\GL(V)$-action on $V$ is the natural one. For every $i,j \in\{0,\ldots,n\}$ define $\rho_{ij} \in \GL(V)$ by
\[
\rho_{ij}(\alpha_k) = \begin{cases} \alpha_j & k=i \\ \alpha_i & k=j \\ \alpha_k & i \neq k \neq j.\end{cases}
\]
It is well-defined since it is compatible with the relation $\sum_k \alpha_k = 0$. The corners of $L_{\circ}$ generates the sublattice $sL :=\{s\lambda \mid \lambda \in L\}$. For each $x \in sL$ define $\rho^{x}_{ij} = \Aff(V)$ by $\rho^{x}_{ij}(v) = \rho_{ij}(v-x) +x$. Now set
\[\begin{split}
R &= \langle \rho_{ij} \mid 0\le i,j\le n \rangle < \GL(V),\\
R_{\rm Aff} &= \langle \rho^{x}_{ij} \mid 0\le i,j\le n, \; x \in sL \rangle < \Aff(V), \\
T &= \langle s(\alpha_i - \alpha_j) \mid 0 \le i \neq j \le n \rangle < L.
\end{split}\]
By a change of coordinates we will show that the above groups are isomorphic to the Weyl group, affine Weyl group and root lattice associated to $\Phi$, respectively. For this purpose denote by $s_{\alpha,\ell}$, the reflection in the hyperplane $H_{\alpha,\ell} = \{v \in V \mid (v,\alpha) = \ell\}$
for $\alpha \in \Phi$ and $\ell \in \Z$. Set $s_\alpha = s_{\alpha,0}$ and define
\[\begin{split}
W &= \langle s_\alpha \mid \alpha \in \Phi \rangle < \GL(V), \\
W_{\rm Aff} &= \langle s_{\alpha,\ell} \mid \alpha \in \Phi, \ \ell \in \Z \rangle < \Aff(V)
\end{split}\]

Let $\Upsilon$ be the automorphism of $V$ defined by
\[
\Upsilon(\alpha_k) = \frac{1}{s}\left(-e_k + \frac{1}{n+1}\sum_{i=0}^{n} e_i \right)
\]
and define the automorphism $\widetilde{\Upsilon} \colon \Aff(V) \to \Aff(V)$, by $\widetilde{\Upsilon}(f) = \Upsilon^{-1}f \Upsilon$.

\begin{lemma}\label{AffineWeyl}
In the notation above $\widetilde{\Upsilon}(W) = R$, $\widetilde{\Upsilon}(L) = T$, and $\widetilde{\Upsilon}(W_{\rm Aff}) = R_{\rm Aff}$. Moreover,  $R_{\rm Aff} = T \rtimes R$.
\end{lemma}
\begin{proof}
Let $i,j,k \in \{0,\ldots, n\}$ and write $\rho_{ij}(\alpha_k) = \alpha_{k'}$. Set $\alpha = e_i -e_j \in \Phi$. Since $s_\alpha$ is induced by the automorphism of $\R^{n+1}$ that permutes the standard basis vectors by transpositioning $e_i$ and $e_j$ we obtain
\[
s_{\alpha}(\Upsilon(\alpha_k)) = \frac{1}{s}\left(-e_{k'} + \frac{1}{n+1}\sum_{i=0}^{n} e_i \right) = \Upsilon(\alpha_{k'}) = \Upsilon(\rho_{ij}(\alpha_k)),
\]
so $\widetilde{\Upsilon}(s_{\alpha}) = \rho_{ij}$ and $\widetilde{\Upsilon}(W) = R$.

Next notice that $\Upsilon(s(\alpha_i-\alpha_j)) = e_j - e_i$ and so $\Upsilon^{-1}(\Phi)$ generates $T$ as an abelian group. Hence $\widetilde{\Upsilon}$ sends the translations by elements in $L$ to the translations by elements in $T$.

Since $W_{\rm Aff} = L \rtimes W$ \cite[4.2]{Hu}, we have $\widetilde{\Upsilon}(W_{\rm Aff}) = T \rtimes R$. Observe that
\[
\rho^{x}_{ij}(v) = \rho_{ij}(v-x)+x = \rho_{ij}(v)+ x- \rho_{ij}(x).
\]
For $x = s\alpha_k$ we find
\[
x- \rho_{ij}(x) = \begin{cases} s(\alpha_i-\alpha_j) & k=i \\ s(\alpha_j-\alpha_i) & k=j \\ 0 & i \neq k \neq j,\end{cases}
\]
so $\{x - \rho_{ij}(x) \mid x \in sL\} = T$. Hence $R_{\rm Aff} = T \rtimes R = \widetilde{\Upsilon}(W_{\rm Aff})$.
\end{proof}

The $R_{\rm Aff}$-action on $V$ induces an $R_{\rm Aff}$-action on $Q_0$. This extends uniquely to an $R_{\rm Aff}$-action on $Q$. In particular $R_{\rm Aff}$ acts on the paths of $Q$ and this induces an action on $\Sigma$, the set of small cycles in $Q$. Let $\Sigma^{\circ}$ be the subset of small cycles contained in $Q^{\circ}$.

\begin{lemma}\label{fundamentalDomain}
The sets $Q^{\circ}_0$, $Q^{\circ}_1$ and $\Sigma^{\circ}$ are fundamental domains for the $R_{\rm Aff}$-actions on $Q_0$, $Q_1$ and $\Sigma$ respectively.
\end{lemma}
\begin{proof}
By \cite[4.5]{Hu} $W_{\rm Aff}$ acts simply transitively on the set of alcoves, which are the connected components of $V \setminus \bigcup_{\alpha,\ell} H_{\alpha,\ell}$ with fundamental alcove
\[
A_{\circ} = \{v \in V \mid (v,\alpha_0)>-1 \mbox{ and }(v,\alpha_i) >0 \mbox{ for all } 1\le i \le n \}.
\]
Moreover, its closure $\overline{A_{\circ}}$ is a fundamental domain for the $W_{\rm Aff}$-action on $V$ \cite[4.8]{Hu}. 

Let $v \in V$ and write $v = \sum_{k=1}^n m_k \Upsilon(\alpha_k)$. Then
\[
(v,\alpha_0) = \frac{m_n}{s}, \; (v,\alpha_1) = -\frac{m_1}{s} \mbox{ and } (v,\alpha_i) = \frac{m_{i-1}-m_{i}}{s}
\]
for all $2\le i \le n$. Thus $v \in A_{\circ}$ if and only if $ 0 > m_1 > \cdots > m_n > -s$ and
\[\begin{split}
\Upsilon^{-1}A_{\circ} 
&= \left\{\left. \sum_{k=1}^{n}m_k\alpha_k \ \right| 0 > m_1 > \cdots > m_n > -s  \right\} \\
&= \{v \in V \mid v_i > 0 \mbox{ for all } 0 \le i \le n-1 \mbox{ and } v_{n} > -s\}.
\end{split}
\]
Hence $L_{\circ} = \Upsilon^{-1}(\overline{A_{\circ}}) \cap L$ and so $Q^{\circ}_0$ is fundamental domain for the $R_{\rm Aff}$-action on $Q_0$, by Lemma~\ref{AffineWeyl}. Let $\varsigma$ be a small cycle and $X \subset V$ the set of vertices that it passes. Since $\Upsilon X$ does not lie in any hyperplane it contains a vertex $\Upsilon(v)$ that lies in some alcove. Since $W_{\rm Aff}$ acts simply transitively on the set of alcoves there is a unique $w \in W_{\rm Aff}$ such that $w\Upsilon(v) \in A_{\circ}$. By Lemma~\ref{AffineWeyl}, $r = \widetilde{\Upsilon}(w) \in R_{\rm Aff}$ and
\[
rv \in \left\{\left. \sum_{k=1}^{n}m_k\alpha_k \ \right| 0 > m_1 > \cdots > m_n > -s  \right\},
\]
which implies
\[\begin{split}
rX 
&\in \left\{\left. \sum_{k=1}^{n}m_k\alpha_k \ \right| 0 \ge m_1 \ge \cdots \ge m_n \ge -s  \right\} \cap L = L_{\circ}
\end{split}\]
Since every small cycle is determined by the vertices it passes we obtain that $\Sigma^{\circ}$ is a fundamental domain for the $R_{\rm Aff}$-actions on $\Sigma$.

Every arrow appears in some small cycle and is determined uniquely by its starting point and endpoint. Hence $Q^{\circ}_1$ is a fundamental domain for the $R_{\rm Aff}$-actions on $Q_1$.
\end{proof}

\begin{lemma}\label{reflectCut}
Let $C^{\circ}$ be a restricted cut in $Q^{\circ}$ and set $C = R_{\rm Aff}C^{\circ}$. Then $C$ is a bounding periodic cut with $s(n+1)L_{C} < T < L_C$.
\end{lemma}
\begin{proof}
By Lemma~\ref{fundamentalDomain}, each small cycle has a unique arrow in $C$. Hence $C$ is a cut. Moreover, $R_{\rm Aff} $ acts on $Q_1\setminus C$ with fundamental domain $Q^{\circ}_1\setminus C^{\circ}$.

Now let $p$ be a path of length $l\ge 1$ in $Q_C$. Its orbit under $R_{\rm Aff}$ is a path of length $l$ in $Q/R_{\rm Aff}$ and so corresponds to a path of length $l$ in $Q^{\circ}_{C^{\circ}}$. But in $Q^{\circ}_{C^{\circ}}$ every path has length at most $sn$ so the same holds for $Q_C$ and thus $C$ is bounding.

Since $C=R_{\rm Aff}C^{\circ}$ and $T < R_{\rm Aff}$ we have $T< L_C$. By
\[
\sum_{i=0}^{n} s(\alpha_j - \alpha_i) = s(n+1)\alpha_j - s\sum_{i=0}^{n} \alpha_i = s(n+1)\alpha_j
\]
we have $s(n+1)L_{C} < T$. It follows that $L_C$ is cofinite and so $C$ is periodic.
\end{proof}

Now we are ready to prove Theorem~\ref{A factor of Atilde}.

\begin{proof}[Proof of Theorem~\ref{A factor of Atilde}]
The claim is well-known for the case $n=1$ so we assume that $n \ge 2$. Take $C= R_{\rm Aff}C^{\circ}$ and $B= s(n+1)L$. Then by Lemma~\ref{reflectCut}, $C$ is a bounding periodic cut and $B \le L_C$. Let $e \in K(Q_0/B)$ be the sum of the idempotents corresponding to the vertices in $(BQ^{\circ}_0)/B$. 

We claim that for any $0 \neq \lambda \in B$ we have $Q^{\circ}_0 \cap \lambda Q^{\circ}_0 = \emptyset$ and there are no arrows between $Q^{\circ}_0$ and $\lambda Q^{\circ}_0$ in $Q$. Assume on the contrary that there are $v,v' \in Q^{\circ}_0$ such that either $v =\lambda v'$ or there is an arrow between $v$ and $\lambda v'$. Since $v,v' \in Q^{\circ}_0$ we may write
\[
v - v' = \sum_{k=1}^{n} m_k \alpha_k,
\]
for some unique $-s \le m_k \le s$. On the other hand $\lambda = v - v'$ or $\lambda = v-v'\pm \alpha_i$ for some $i$ which implies $\lambda = 0$ as $\lambda \in s(n+1)L \subset (s+2)L$. 

It follows that sending paths to their $B$-orbits defines a bijection between the paths in $Q^{\circ}$ and the paths in $Q/B$ that pass only through the vertices in $(BL_{\circ})/B$. Hence there is an isomorphism $\phi \colon \Gamma^{\circ} \to (\Gamma/B)/(1-e)$ sending vertices and arrows to their orbits. Since $\phi$ is homogeneous of degree $0$ it induces an isomorphism from $\Gamma^{\circ}_{C_{\circ}}$ to $(\Gamma/B)_{C}/(1-e)$.
\end{proof}

\section{Non-commutative algebraic-geometric approach}\label{section: NCAG}

`Non-commutative algebraic geometry' has been a source of methods and ideas in representation theory.
One important example is Geigle-Lenzing's theory of weighted projective curves, where representation theory of canonical algebras is controlled by the geometry of weighted projective curves.
On the other hand, Nakayama functors are nowadays regarded as an analog of canonical sheaves over algebraic varieties since they play a similar role in derived categories.
This point was recently developed further by Minamoto \cite{M}.
He introduced a class of finite dimensional algebras called ``Fano algebras'', which are analog of Fano varieties from the viewpoint of the behaviour of the Serre functors.
Our $n$-representation infinite algebras are closely related to Minamoto's ``Fano algebras of dimension $n$''.
Precisely speaking, $n$-representation infinite algebras are ``extremely quasi-Fano algebras of dimension $n$'' in \cite{MM}.
The aim of this section is to show that the category $\RR$ of $n$-regular modules is equivalent to the category $\qgr_0\Pi$,
which should be understood as the category of coherent sheaves of dimension zero over the non-commutative projective scheme $\qgr\Pi$ in the context of Artin-Zhang's theory \cite{AZ}.

For a graded $K$-algebra $\Gamma$, we denote by $\qGr\Gamma$ and $\qgr\Gamma$ the quotient categories of $\Gr\Gamma$ and $\gr\Gamma$ (see Section~\ref{notation}) by the subcategory of torsion modules. One of the main results in \cite{M} is that, for any $n$-representation-infinite algebra
(or more generally, ``quasi-Fano algebra of dimension $n$'') $\Lambda$, and its preprojective algebra $\Pi$, there exists a triangle equivalence
$\Pi\Lotimes_\Lambda- \colon \DDD(\Mod\Lambda)\equi\DDD(\qGr\Pi)$ which makes the following diagram commutative:
\[\xymatrix{
\DDD(\Mod\Lambda)\ar^{\Pi\Lotimes_\Lambda-}[rr]\ar^{\nu_n^{-1}}[d]&&\DDD(\qGr\Pi)\ar^{(1)}[d]\\
\DDD(\Mod\Lambda)\ar^{\Pi\Lotimes_\Lambda-}[rr]&&\DDD(\qGr\Pi).}\]
He also gave a version of this equivalence for bounded derived categories of finitely generated modules, which is more subtle and plays an important role in our study in this section.
Let us introduce some terminology.

We denote by $(\DDD^{\le0}(\mod\Lambda),\DDD^{\ge0}(\mod\Lambda))$ the standard
t-structure \cite{BBD} of $\DDD^{\rm b}(\mod\Lambda)$, and similarly by $(\DDD^{\le0}(\qgr\Pi),\DDD^{\ge0}(\qgr\Pi))$ the standard t-structure on $\DDD^{\rm b}(\qgr\Pi)$.
We define full subcategories of $\DDD^{\rm b}(\mod \Lambda)$ as follows:
\begin{eqnarray*}
\TT^{\le0}&:=&\{X\in\DDD^{\rm b}(\mod\Lambda) \mid \nu_n^{-i}(X)\in\DDD^{\le0}(\mod\Lambda)\ \mbox{ for }\ i\gg0\},\\
\TT^{\ge0}&:=&\{X\in\DDD^{\rm b}(\mod\Lambda) \mid \nu_n^{-i}(X)\in\DDD^{\ge0}(\mod\Lambda)\ \mbox{ for }\ i\gg0\}.
\end{eqnarray*}

Now we can state Minamoto's result:

\begin{theorem}[{\cite[Theorem 3.7]{M}}] \label{Minamoto}
Let $\Lambda$ be an $n$-representation-infinite algebra and $\Pi$ the preprojective algebra.
Then $\Pi$ is left graded coherent if and only if $(\TT^{\le0},\TT^{\ge0})$ is a t-structure of $\DDD^{\rm b}(\mod \Lambda)$. 
In this case, we have a triangle equivalence $\Pi\Lotimes_\Lambda- \colon \DDD^{\bo}(\mod\Lambda)\equi\DDD^{\bo}(\qgr\Pi)$ which makes the following diagram commutative
\[\xymatrix{
\DDD^{\bo}(\mod\Lambda)\ar^{\Pi\Lotimes_\Lambda-}[rr]\ar^{\nu_n^{-1}}[d]&&\DDD^{\bo}(\qgr\Pi)\ar^{(1)}[d]\\
\DDD^{\bo}(\mod\Lambda)\ar^{\Pi\Lotimes_\Lambda-}[rr]&&\DDD^{\bo}(\qgr\Pi).}\]
and induces equivalences $\TT^{\le0}\equi\DDD^{\le0}(\qgr\Pi)$ and $\TT^{\ge0}\equi\DDD^{\ge0}(\qgr\Pi)$.
\end{theorem}

\subsection{Description of $n$-regular modules}

Throughout this section, let $\Lambda$ be an $n$-representation infinite algebra and $\Pi$ the preprojective algebra of $\Lambda$.
In this section, we will see that methods from non-commutative algebraic geometry
are quite useful in the study of the category $\RR$ in terms of the preprojective algebra $\Pi$.

Throughout this section we assume that $\Pi$ is left graded coherent.
We identify $\DDD^{\bo}(\qgr\Pi)$ with $\DDD^{\bo}(\mod\Lambda)$ by the triangle equivalence given in Theorem~\ref{Minamoto}. Then we have the identifications
\[(1)=\nu_n^{-1},\ \DDD^{\le0}(\qgr\Pi)=\TT^{\le0},\ \DDD^{\ge0}(\qgr\Pi)=\TT^{\ge0} \text{, and}\]
\begin{equation}\label{what is qgr}
\qgr\Pi=\TT^{\le0}\cap\TT^{\ge0}=\{X\in\DDD^{\bo}(\mod\Lambda) \mid \nu_n^{-i}(X)\in\mod\Lambda\ \forall i\gg0\}.
\end{equation}

Following tradition in (non-commutative) algebraic geometry, we denote by $\OO$ the image of $\Pi$ in $\qgr\Pi$.
For a set $I$ of integers, let
\[\OO(I):=\add\{\OO(i) \mid i\in I\}.\]
Since $\nu_n^{-i}(\Lambda)=\OO(i)$ for any $i\in\Z$, we have
\[\UU=\OO(\Z),\ \PP=\OO(\Z_{\ge0})\ \mbox{ and }\ \II[-n]=\OO(\Z_{<0}).\]
Moreover we have the following observations.

\begin{proposition}\label{P, I and O}
\begin{itemize}
\item[(a)] $\CC\cap\qgr\Pi=\II[-n]\vee\PP\vee\RR$.
\item[(b)] $\DDD^{\le0}(\qgr\Pi)\supset\DDD^{\le0}(\mod\Lambda)$ and $\DDD^{\ge0}(\qgr\Pi) \subset\DDD^{\ge0}(\mod\Lambda)$.
\end{itemize}
\end{proposition}

\begin{proof}
(a) We have $\CC=\bigvee_{\ell\in\Z}(\PP\vee\RR\vee\II)[\ell n]$ by Theorem~\ref{trichotomy}. Moreover we have
\begin{align*}
\nu_n^{-i}(X)\subset\PP&\ \ \ (i\gg0)\ \ \ \forall\ X\in\II[-n]&&\mbox{by definition,}\\
\nu_n^{-i}(\PP)\subset\PP&\ \ \ (i\ge0)&&\mbox{by definition,}\\
\nu_n^{-i}(\RR)\subset\RR&\ \ \ (i\ge0)&&\mbox{by Proposition~\ref{nakayama R}.}
\end{align*}
Thus the assertion follows.

(b) Since we have $\TT^{\le0}\supset\DDD^{\le0}(\mod\Lambda)$ and $\TT^{\ge0}\subset\DDD^{\ge0}(\mod\Lambda)$ by Proposition~\ref{t-structure}, we have the assertions.
\end{proof}

\begin{lemma}\label{intersection}
The following subcategories of $\DDD^{\bo}(\mod\Lambda)=\DDD^{\bo}(\qgr\Pi)$ are the same.
\begin{itemize}
\item[(a)] $\mod\Lambda\cap\qgr\Pi$.
\item[(b)] $\{X\in\mod\Lambda \mid \forall i\ge0 \colon \nu_n^{-i}(X)\in\mod\Lambda\}$.
\item[(c)] $\{X\in\mod\Lambda \mid  \forall i<n \colon \Ext^i_\Lambda(\II,X)=0\}$.
\item[(d)] $\{X\in\qgr\Pi \mid  \forall i>0 \colon \Ext^i_{\qgr\Pi}(\OO,X)=0\}$.
\item[(e)] $\{X\in\qgr\Pi \mid \forall i>0 \colon \Ext^i_{\qgr\Pi}(\OO(\Z_{\le0}),X)=0\}$.
\end{itemize}
\end{lemma}

\begin{proof}
By Proposition~\ref{modules under nakayama}(a), if $X$ and $\nu_n^{-i}(X)$ belong to $\mod\Lambda$ for some $i\ge0$, then $\nu_n^{-j}(X)\in\mod\Lambda$ for any $j$ with $0\le j\le i$.
Thus (a) and (b) are the same by the description \eqref{what is qgr} of $\qgr\Pi$.

(b) and (c) are the same by Lemma~\ref{Ext and nu} and Observation~\ref{applied n-shifted serre}.

For any $X\in\mod\Lambda\cap\qgr\Pi$, we have $\Ext^i_{\qgr\Pi}(\OO,X)=
\Ext^i_\Lambda(\Lambda,X)=0$ for any $i>0$. Thus (a) is contained in (d). 
If $X\in\qgr\Pi$ satisfies $\Ext^i_{\qgr\Pi}(\OO,X)=0$ for any $i>0$, then
$X\in\DDD^{\le0}(\mod\Lambda)\cap\qgr\Pi\subset\DDD^{\le0}(\mod\Lambda)\cap\DDD^{\ge0}(\mod\Lambda)\cap\qgr\Pi=\mod\Lambda\cap\qgr\Pi$
by Proposition~\ref{P, I and O}(b). Thus (a) and (d) are the same.

Finally the category (e) is also the same since it is clearly contained in (d)
and conversely any $X\in\mod\Lambda\cap\qgr\Pi$ satisfies $\Ext^i_{\qgr\Pi}(\OO(\Z_{<0}),X)=\Ext^{n+i}_\Lambda(\II,X)=0$ for any $i>0$ by (c) and
hence belongs to the category (e).
\end{proof}

We introduce non-commutative analogue of local cohomology \cite{AZ,BH,BV}.

\begin{definition}\label{local cohomology}
Let $\Gamma=\bigoplus_{i\ge0}\Gamma_i$ be a graded $K$-algebra with $\dim\Gamma_i<\infty$ for any $i\ge0$.
For $X\in\gr\Gamma$, we define the \emph{$i$-th local cohomology} by
\[\Ho^i_{\mm}(X):=\varinjlim_j \bigoplus_{\ell\in\Z}\Ext^i_{\Gamma}(\Gamma_{<j},X(\ell)), \]
where the symbol $\mm$ refers to the ideal $\bigoplus_{i>0}\Gamma_i$ of $\Gamma$.

Define the \emph{dimension} of $X$ by
\[\dim X:=\sup\{i\ge0 \mid \Ho^i_{\mm}(X)\neq0\}.\]
For each $\ell\ge0$, define the full subcategory of $\gr\Gamma$ by
\[\gr_\ell\Gamma:=\{X\in\gr\Gamma \mid \dim X\le\ell\}\]
and let $\qgr_{\ell-1}\Gamma$ be the corresponding full subcategory of $\qgr\Gamma$.
\end{definition}

Now we have the following description of $n$-regular modules.

\begin{theorem}\label{description of R}
Let $\Lambda$ be an $n$-representation infinite algebra, $\RR$ the category of $n$-regular modules and $\Pi$ the preprojective algebra of $\Lambda$.
Assume that $\Pi$ is left graded coherent. Then
\[\RR=\qgr_0\Pi.\]
\end{theorem}

Let us start by recalling the following basic result, where we for $i\ge0$ write
\[\underline{\Ext}^{i}(\OO,X):=\bigoplus_{\ell\in\Z}\Ext_{\qgr\Pi}^i(\OO,X(\ell)).\]

\begin{proposition}[{\cite[Lemma 4.1.5]{BV}}]\label{Ext and Hm}
For any $X \in \gr\Pi$, we have an exact sequence
\[0\to \Ho^0_{\mm}(X)\to X\to\underline{\Hom}(\OO,X)\to \Ho^1_{\mm}(X)\to0\]
and an isomorphism $\underline{\Ext}^i(\OO,X)\to \Ho^{i+1}_{\mm}(X)$ for any $i\ge1$.
\end{proposition}

This immediately gives us the following description of $\qgr_{\ell}\Pi$.

\begin{lemma}\label{description of qgr_l}
For any $\ell$, we have
\[\qgr_{\ell}\Pi=\{X\in\qgr\Pi \mid \Ext_{\qgr \Pi}^i(\OO(\Z),X)=0\ \forall i>\ell\}.\]
\end{lemma}

\begin{proof}
We have
\begin{align*}
\dim X\le \ell +1 &\Longleftrightarrow\Ho^{i+1}_{\mm}(X)=0\ \forall\ i > \ell&\\
&\Longleftrightarrow\underline{\Ext}^i(\OO,X)=0\ \forall\ i > \ell&\mbox{by Proposition~\ref{Ext and Hm},}\\
&\Longleftrightarrow\Ext_{\qgr \Pi}^i(\OO(\Z),X)=0\ \forall\ i > \ell.&
\end{align*}
Thus the assertion holds.
\end{proof}

Now we are ready to prove Theorem~\ref{description of R}.

\begin{proof}[Proof of Theorem~\ref{description of R}]
We have equalities
\begin{align*}
\RR&=\{ X\in\mod\Lambda \mid \Ext_\Lambda^{n-i}(\II,X)=0=\Ext_\Lambda^i(\PP,X)=0\ \forall i>0\}&&\mbox{by Observation~\ref{almost definition of R},}\\
&=\{ X\in\mod\Lambda\cap\qgr\Pi \mid \Ext_\Lambda^i(\PP,X)=0\ \forall i>0\}&&\mbox{by Proposition~\ref{intersection}(a) and (c),}\\
&=\{ X\in\mod\Lambda\cap\qgr\Pi \mid \Ext_{\qgr\Pi}^i(\OO(\Z_{\ge0}),X)=0\ \forall i>0\}&&\mbox{since $\PP=\OO(\Z_{\ge0})$,}\\
&=\{ X\in\qgr\Pi \mid \Ext_{\qgr\Pi}^i(\OO(\Z),X)=0\ \forall i>0\}&&\mbox{by Proposition~\ref{intersection}(a) and (e),}\\
&=\qgr_0\Pi&&\mbox{by Lemma~\ref{description of qgr_l}.}
\end{align*}
Thus the assertion follows.
\end{proof}

As an application, we have the following result.

\begin{corollary}\label{Horrocks}
The following full subcategories of $\DDD^{\rm b}(\qgr\Pi)$ are the same.
\begin{itemize}
\item[(a)] $\II[-n]\vee\PP\vee\RR$.
\item[(b)] $\CC\cap\qgr\Pi$.
\item[(c)] $\OO(\Z)\vee\qgr_0\Pi$.
\item[(d)] $\{X\in\qgr\Pi \mid \forall i \in \{2, \ldots, n\} \colon \Ho^i_{\mm}(X)=0\}$.
\end{itemize}
\end{corollary}

\begin{proof}
(a) and (b) are the same by Proposition~\ref{P, I and O}(a), and 
(a) and (c) are the same by Theorem~\ref{description of R}.

We will show that (b) is equal to (d). For $X\in\qgr\Pi$, we have
\begin{align*}
X\in\CC & \Longleftrightarrow \Ext_{\qgr\Pi}^{i}(\OO(\Z),X)=0\ \forall\ 1\le i\le n-1 && \mbox{since $\II[-n]\vee\PP=\OO(\Z)$,}\\
& \Longleftrightarrow \Ho^{i+1}_{\mm}(X)=0\ \forall\ 1\le i\le n-1 && \mbox{by Proposition~\ref{Ext and Hm}.}
\end{align*}
Thus the assertion follows.
\end{proof}

We end this section with posing the following conjecture, which we believe to be true from our experience in the classical case $n=1$.

\begin{conjecture} \label{quest.R_exists}
$\RR$ is non-zero for any $n$-representation infinite algebra.
\end{conjecture}

When $\Pi$ is left graded coherent, then by Theorem~\ref{description of R} this is equivalent to asking for the existence of $X\in\gr_1\Pi$ which is infinite dimensional.
We will see in Corollary~\ref{consequence of tame} that Conjecture~\ref{quest.R_exists} has a positive answer for $n$-representation tame algebras, which are the main object in the next subsection.

\subsection{$n$-representation tame algebras}

In this subsection we introduce a class of $n$-representation infinite algebras
which we can control.
Recall that a ring $\Gamma$ is called a \emph{Noetherian $R$-algebra} (or simply \emph{Noetherian algebra})
if $R$ is a commutative Noetherian ring and $\Gamma$ is a finitely generated $R$-module.
If $\Gamma$ is a Noetherian algebra, then it is a Noetherian $Z$-algebra for the center $Z$ of $\Gamma$.

\begin{definition}
We say that an $n$-representation infinite algebra $\Lambda$ is \emph{$n$-representation tame}
if its preprojective algebra is a Noetherian algebra.
\end{definition}

The name is explained by part (a) of the following example, which tells us that representation tame path algebras are 1-representation tame.

\begin{example}\label{example of tame}
\begin{itemize}
\item[(a)] The path algebra $KQ$ of an extended Dynkin quiver $Q$ is $1$-representation tame.
In fact $\Pi$ is Morita equivalent to the skew group algebra $K[x,y]*G$ of a finite subgroup $G$ of $\SL_2(K)$, and this is a finitely generated module
over the invariant subring $K[x,y]^G$ which is Noetherian.
\item[(b)] Any $n$-representation infinite algebra of type $\widetilde{A}$, as introduced in Section~\ref{sect.Atilde}, is $n$-representation tame.
In fact $\Pi$ is isomorphic to the skew group algebra $K[x_0,\ldots,x_{n}]*H$ of a finite abelian subgroup $H$ of $\SL_{n+1}(K)$ by Theorem~\ref{CY Atilde} and Lemma~\ref{skew Atilde}, and this is a finitely generated module
over the invariant subring $K[x_0,\ldots,x_{n}]^G$ which is Noetherian.
\end{itemize}
\end{example}

We can control modules over a Noetherian algebra by using techniques from commutative algebra.
For example, the dimension introduced in Definition~\ref{local cohomology} turns out to be the Krull dimension:

\begin{proposition}\label{2 dimensions are same}
Let $\Gamma$ and $X$ be as in Definition~\ref{local cohomology}.
Assume that $\Gamma$ is a Noetherian $R$-algebra for a graded $K$-subalgebra $R$ of $\Gamma$. Then the dimension $\dim X$ given in Definition~\ref{local cohomology} coincides with the Krull dimension $\dim_RX$ of the $R$-module $X$.
\end{proposition}

\begin{proof}
Clearly we have $\Ho_{\mm}^0(X)=\varinjlim_j\Hom_{R}(R_{<j},X)$.
Taking right derived functors, we have $\Ho_{\mm}^i(X)=\varinjlim_j\Ext^i_{R}(R_{<j},X)$.
Thus the assertion follows from \cite[Theorem 3.5.7]{BH}.
\end{proof}

As an application, we give some consequences of results in the previous subsection.

\begin{corollary}\label{consequence of tame}
Let $\Lambda$ be an $n$-representation tame algebra over a perfect field $K$.
\begin{itemize}
\item[(a)] We have $\dim R_{\nn} \geq n+1$ for any maximal ideal $\nn$ of $R$ such that $\Pi_{\nn} \neq 0$.
\item[(b)] The category $\RR$ is non-zero.
\item[(c)] An object $X\in\qgr\Pi$ satisfies $\Ho^i_{\mm}(X)=0$ for any $1\le i\le n$ if and only if $X\in\OO(\Z)$.
\end{itemize}
\end{corollary}

Part (b) above gives a positive answer to Conjecture~\ref{quest.R_exists}.
Part (c) gives a non-commutative generalization of Horrocks criterion for projective spaces \cite[Theorem 2.3.1]{OSS}.

\begin{proof}
(a) By Proposition~\ref{CY algebra gives CY triangulated}, we have
\[ \Hom_{\DDD(\Mod\Pi)}(X,\Pi[i]) \iso D\Hom_{\DDD(\Mod\Pi)}(\Pi[i],X[n+1])=0 \]
for any $X\in\DDD^{\le0}_{\fd\Pi}(\Mod\Pi)$ and $0 \le i\leq n$.
Now let $X:=\Pi \Lotimes_R R/ \nn \in\DDD^{\le0}_{\fd\Pi}(\Mod\Pi)$. Then
\[\Ext^i_R( R/ \nn, \Pi)=\Hom_{\DDD(\Mod R)}( R/ \nn, \Pi[i])=\Hom_{\DDD(\Mod\Pi)}(X, \Pi[i])=0\]
for any $0 \le i\leq n$. Thus we have $\dim R_{\nn} \ge \depth\Pi_{\nn} > n$.

(b) There always exist a graded $\Pi$-module $X$ with $\dim_RX=1$.

(c) If $X\in\OO(\Z)$, then $X\in\gr\proj\Pi$ and we have
$\Ext^i_{\Pi}(\Pi_{<j},X)=D\Ext^{n+1-i}_{\Pi}(X,\Pi_{<j})=0$ for any $i<n+1$ by Proposition~\ref{CY algebra gives CY triangulated}. Thus we have $\Ho^i_{\mm}(X)=0$ for any $i<n+1$.

Conversely, assume $\Ho^i_{\mm}(X)=0$ for any $1\le i\le n$.
By Corollary~\ref{Horrocks}, we have $X\in\OO(\Z)\vee\qgr_0\Pi$.
If $X$ has a non-zero direct summand $Y$ in $\qgr_0\Pi$, then
$\dim_RX=1$ by Proposition~\ref{2 dimensions are same} and we have $\Ho^1_{\mm}(Y)\neq0$, a contradiction to $\Ho^1_{\mm}(X)=0$.
Thus we have $X\in\OO(\Z)$.
\end{proof}

In this section, we describe the category $\RR$ of $n$-regular modules over an $n$-representation infinite algebra in terms of
the categories of finite dimensional modules over certain infinite dimensional algebras associated with the preprojective algebra.

\medskip
In the rest of this section, let $\Gamma=\bigoplus_{i\ge0}\Gamma_i$ be a graded $K$-algebra such that $\dim_K\Gamma_i<\infty$ for any $i$.
Assume that $\Gamma$ is a ring-indecomposable Noetherian $R$-algebra for a graded $K$-subalgebra $R$ of $\Gamma$.
Then $R_0$ is a finite dimensional local $K$-algebra. Moreover we have
the following observation.

\begin{proposition}[{e.g.\ \cite[1.5.5]{BH}}]\label{finitely generated}
Under the above assumptions, $R$ is a finitely generated $K$-algebra.
\end{proposition}

\begin{proof}
For the convenience of the reader, we include a proof.
Clearly $\dim_{K}R_i<\infty$ for any $i\ge0$.
In particular we only have to show that $R$ is a finitely generated $R_0$-algebra.
Assume that $R$ is not finitely generated. Then we can choose an infinite sequence
$(a_1,a_2,a_3,\ldots)$ of homogeneous elements in $R$ satisfying the following
conditions:
\begin{itemize}
\item $a_\ell\in R_{d_\ell}$ and $1\le d_1 \le d_2 \le d_3 \le\cdots$.
\item $R=R_0[a_1,a_2,a_3,\ldots]$ and $a_{i+1}\notin R_0[a_1,a_2,\ldots , a_i]$.
\end{itemize}
Then we have
\begin{equation}\label{Z}
R_0[a_1,\ldots,a_\ell]\supset \bigoplus_{i=0}^{d_{\ell+1}-1}R_i.
\end{equation}
We show that $a_{\ell+1}$ does not belong to the ideal of $R$ generated by
$a_1,\ldots,a_\ell$. Otherwise there exist homogeneous elements $b_1,\ldots,b_\ell\in R$
such that
\begin{equation}\label{ab}
a_{\ell+1}=a_1b_1+\cdots+a_\ell b_\ell.
\end{equation}
Clearly the degree of each $b_i$ is smaller than $d_{\ell+1}$.
By \eqref{Z}, we have $b_i\in R_0[a_1,\ldots,a_\ell]$. This implies
$a_{\ell+1}\in R_0[a_1,\ldots,a_\ell]$ by \eqref{ab}, a contradiction.

Thus we have an infinite increasing chain
\[(a_1)\subsetneq(a_1,a_2)\subsetneq(a_1,a_2,a_3)\subsetneq\cdots\]
of ideals of $R$, a contradiction.
\end{proof}

As usual we denote by $\Spec R$ the set of all prime ideals of $R$ and by $\MaxSpec R$ the set of all maximal ideals of $R$.
We denote by $R_+:=\bigoplus_{i>0}R_i$, and by $\mm$ the unique maximal prime ideal of $R$ containing $R_+$.
We denote by $\Proj R$ the set of all homogeneous prime ideals of $R$ not containing $\bigoplus_{i>0} R_i$, and by $\MaxProj R$ the set of maximal elements of $\Proj R$ with respect to inclusion.

For $X\in\gr\Gamma$, we denote by $\Supp_R X$ the set of all $\pp\in\Spec R$ such that $X_\pp\neq0$.
For $\pp\in\MaxProj R$, we denote by $\gr_{\pp}\Gamma$ the full subcategory of $\gr\Gamma$ consisting of $X$ such that $\Supp_R X\subset V(\pp) :=\{\qq\in \Spec R \mid\pp\subset\qq\}$.

We need the following preparation.

\begin{lemma}\label{some basic}
\begin{itemize}
\item[(a)] For any $\pp\in\MaxProj R$, the ring $R/\pp$ has Krull dimension one.
\item[(b)] For $X \in \mod \Gamma$ the following are equivalent:
\begin{itemize}
\item[$\bullet$] $X \in \fd \Gamma$;
\item[$\bullet$] $X$ has finite length as a $\Gamma$-module;
\item[$\bullet$] $X$ has finite length as an $R$-module;
\item[$\bullet$] $\Supp_R X \subseteq \MaxSpec R$.
\end{itemize}
\end{itemize}
\end{lemma}

\begin{proof}
(a) This is well-known \cite[Theorem 1.58]{BH}.

(b) Since $R$ is a finitely generated $K$-algebra by Lemma~\ref{finitely generated}, any simple $R$-module is finite dimensional. Thus $X$ is finite dimensional if and only if it has finite length as an $R$-module.
It is clear that $X$ has finite length as an $R$-module if and only if $\Supp_R X\subset\MaxSpec R$.

If $X$ has finite length as an $R$-module, then it clearly has finite length as a $\Gamma$-module.
It remains to show that any finite length $\Gamma$-module has finite length as an $R$-module.
We only have to show that any simple $\Gamma$-module $X$ has finite length as an $R$-module.
We know that $R/\ann_RX$ is a subring of $\End_\Gamma(X)$, which is a division algebra.
Since $\End_\Gamma(X)$ is a finitely generated $R$-module, the ring $R/\ann_RX$ has to be a field.
Thus $X$ is a finitely generated $R$-module annihilated by a maximal ideal, so it has finite length as an $R$-module.
\end{proof}

We have the following decomposition of $\qgr_0\Gamma$.

\begin{proposition}\label{qgr decomposition}
We have $\qgr_0\Gamma=\coprod_{\pp\in\MaxProj R}\qgr_{\pp}\Gamma$.
\end{proposition}

\begin{proof}
Clearly we have $\Hom_{\qgr\Gamma}(X,Y)=0$ for any $X\in\qgr_{\pp}\Gamma$ and $Y\in\qgr_{\qq}\Gamma$ with $\pp\neq\qq$.

Thus we only have to show that, for any $X\in \gr_1\Gamma$ without non-zero finite dimensional submodules, there exists $\Gamma$-submodules $X_1,\ldots,X_\ell$ satisfying the following conditions:
\begin{itemize}
\item[(i)] For any $i=1,\ldots,\ell$, we have $X_i\in\qgr_{\pp_i}\Gamma$ for some $\pp_i\in\MaxProj R$.
\item[(ii)] $\sum_{i=1}^\ell X_i=\bigoplus_{i=1}^\ell X_i$ and $\dim_K(X/\sum_{i=1}^\ell X_i)<\infty$.
\end{itemize}
Then we have $X=\bigoplus_{i=1}^\ell X_i$ in $\qgr\Gamma$.

Let $\Supp_R  X= \bigcup_{i=1}^\ell V(\pp_i)$. For $i=1,\ldots,\ell$, let
\[X_i:=\sum Y\]
where $Y$ runs over all $\Gamma$-submodules of $X$ such that $\Supp_R  Y\subset V(\pp_i)$.
Then $\Supp_R  X_i\subset V(\pp_i)$, so (i) holds.
Now we show (ii). 
By Lemma~\ref{some basic}(b), we only have to show $\Supp_R (X/\sum_{i=1}^\ell X_i)\subset\MaxSpec R$.
Since
\[\Supp_R (X/\sum_{i=1}^\ell X_i)\subset\bigcap_{i=1}^\ell\Supp_R (X/X_i),\]
we only have to show $\pp_i\notin\Supp_R (X/X_i)$.
Assume $\pp_i\in\Supp_R (X/X_i)$. Then $\pp_i$ also belongs to the set of associated prime ideals $\Ass_R(X/X_i)$,
since minimal elements in $\Supp_R (X/X_i)$ belong to $\Ass_R(X/X_i)$.
Thus we have an element $x\in X$ such that the image $\overline{x}\in X/X_i$
satisfies $R\overline{x}\iso R/\pp_i$.
Let $Y:=X_i+\Gamma x$. Then $Y/X_i=\Gamma\overline{x}$ is annihilated by $\pp_i$,
so we have $\Supp_R (Y/X_i)\subset V(\pp_i)$.
This implies $\Supp_R  Y\subset V(\pp_i)$, a contradiction to the maximality of $X_i$.
\end{proof}

Next we show that the category $\qgr_{\pp}R$ has a very simple description for any $\pp\in\MaxProj R$.

For any $\pp=\bigoplus_{i\ge0}\pp_i\in\Proj R$,
we denote by $R_{[\pp]}$ the localization of $R$ with respect to the multiplicative set consisting of homogeneous elements in $R\backslash\pp$.
Clearly $R_{[\pp]}$ has a structure of a graded $K$-algebra $R_{[\pp]}=\bigoplus_{i\in\Z}R_{[\pp],i}$. We let $R_{(\pp)}:=R_{[\pp],0}$. In other words,
\[R_{[\pp]}\supset R_{(\pp)}:=\sum_{i\ge0}R_i(R_i\backslash\pp)^{-1}.\]
Similarly we define $\Gamma_{[\pp]}$ and $\Gamma_{(\pp)}$, and also $X_{[\pp]}$ and $X_{(\pp)}$ for $X\in\gr\Gamma$:
\begin{eqnarray*}
\Gamma_{[\pp]}:=\Gamma\otimes_RR_{[\pp]}\supset \Gamma_{(\pp)}&:=&\sum_{i\ge0}\Gamma_i(R_i\backslash\pp)^{-1},\\
X_{[\pp]}:=X\otimes_RR_{[\pp]}\supset X_{(\pp)}&:=&\sum_{i\ge0}X_i(R_i\backslash\pp)^{-1}.
\end{eqnarray*}
Clearly this gives a functor $(-)_{(\pp)} \colon \gr\Gamma\to\mod\Gamma_{(\pp)}$, which is the composition of $(-)_{[\pp]} \colon \gr\Gamma\to\mod\Gamma_{[\pp]}$
and $(-)_0 \colon \mod\Gamma_{[\pp]}\to\mod\Gamma_{(\pp)}$.
Since $X_{\pp}=0$ holds for any $X\in\gr\Gamma$ which is finite dimensional,
we have an induced functor
\[(-)_{(\pp)} \colon \qgr\Gamma\to\mod\Gamma_{(\pp)}.\]
It is easily checked that this functor is exact and dense.
When $\pp\in\MaxProj R$, we have the following much stronger result.

\begin{theorem}\label{qgr is fl}
Assume that $\Gamma$ is generated by $\Gamma_0$ and $\Gamma_1$ as a $K$-algebra.
Then for any $\pp\in\MaxProj R$, the functor
$(-)_{(\pp)}$ induces an equivalence $\qgr_{\pp}\Gamma\to\fd\Gamma_{(\pp)}$.
\end{theorem}

First we observe the following easy property of $\Gamma_{[\pp]}$.

\begin{lemma}\label{nice grading}
Assume that $\Gamma$ is generated by $\Gamma_0$ and $\Gamma_1$ as a $K$-algebra.
\begin{itemize}
\item[(a)] We have $\Gamma_{[\pp],-i}\Gamma_{[\pp],i}=\Gamma_{(\pp)}$ for any $i\in\Z$.
\item[(b)] We have $\Gamma_{[\pp]}(i)\in\add_{\gr\Gamma_{[\pp]}}(\Gamma_{[\pp]})$ for any $i\in\Z$.
\end{itemize}
\end{lemma}

\begin{proof}
(a) By our assumption, we have $\Gamma_i\Gamma_j=\Gamma_{i+j}$ for any $i,j\ge0$.
Take $r\in R_{\ell}\backslash\pp$ with $\ell>0$.
Since $r^i\in\Gamma_{i\ell}=\Gamma_{i\ell-i}\Gamma_{i}$, we have $1\in(\Gamma_{i\ell-i}r^{-i})\cdot\Gamma_i\subset\Gamma_{[\pp],-i}\Gamma_{[\pp],i}$.
Thus we have the assertion for $i\ge0$. The case $i<0$ follows similarly.

(b) We can identify $\Hom_{\gr\Gamma_{[\pp]}}(\Gamma_{[\pp]}(i),\Gamma_{[\pp]}(j))$ with $\Gamma_{[\pp],j-i}$.
Then the equality in (a) shows $\Hom_{\gr\Gamma_{[\pp]}}(\Gamma_{[\pp]}(i),\Gamma_{[\pp]})\cdot\Hom_{\gr\Gamma_{[\pp]}}(\Gamma_{[\pp]},\Gamma_{[\pp]}(i))=\End_{\gr\Gamma_{[\pp]}}(\Gamma_{[\pp]}(i))$.
This implies the assertion.
\end{proof}

Now we denote by $\gr\fd\Gamma_{[\pp]}$ the full subcategory of $\gr\Gamma_{[\pp]}$ consisting of degreewise finite dimensional modules.
We have the following equivalence.

\begin{proposition}[{cf.\ \cite[Theorem I.3.4]{NV}}]\label{new proof 1}
Assume that $\Gamma$ is generated by $\Gamma_0$ and $\Gamma_1$ as a $K$-algebra.
We have an equivalence $(-)_0 \colon \gr\Gamma_{[\pp]}\to\mod\Gamma_{(\pp)}$ which induces an equivalence $(-)_0 \colon \gr\fd\Gamma_{[\pp]}\to\fd\Gamma_{(\pp)}$.
\end{proposition}

\begin{proof}
We show that the functor $\Gamma_{[\pp]}\otimes_{\Gamma_{(\pp)}}- \colon \mod\Gamma_{(\pp)}\to\gr\Gamma_{[\pp]}$ gives a quasi-inverse.
Clearly the composition $(\Gamma_{[\pp]}\otimes_{\Gamma_{(\pp)}}-)_0$ is isomorphic to the identity functor on $\mod\Gamma_{(\pp)}$.

We have a morphism $\alpha \colon \Gamma_{[\pp]}\otimes_{\Gamma_{(\pp)}}(-)_0\to{\rm id}_{\gr\Gamma_{[\pp]}}$ of functors on $\gr\Gamma_{[\pp]}$ given by
$\alpha_X \colon \Gamma_{[\pp]}\otimes_{\Gamma_{(\pp)}}X_0\to X$ where $\alpha_X(\gamma\otimes x)=\gamma x$ for any $\gamma\in\Gamma_{[\pp]}$ and $x\in X_0$.
Clearly $\alpha_X$ is an isomorphism if $X\in\add_{\gr\Gamma_{[\pp]}}(\Gamma_{[\pp]})$.
For arbitrary $X\in\gr\Gamma_{[\pp]}$, there exists a projective presentation $P'\xrightarrow{g}P\xrightarrow{f} X\to0$ in $\gr\Gamma_{[\pp]}$
with $P,P'\in\add_{\gr\Gamma_{[\pp]}}(\Gamma_{[\pp]})$ by Lemma \ref{nice grading}(b).
Applying $\Gamma_{[\pp]}\otimes_{\Gamma_{(\pp)}}(-)_0$, we have a commutative diagram
\[\xymatrix{
\Gamma_{[\pp]}\otimes_{\Gamma_{(\pp)}}P'_0\ar[rr]^{\Gamma_{[\pp]}\otimes g_0}\ar[d]_{\wr}^{\alpha_{P'}}
&&\Gamma_{[\pp]}\otimes_{\Gamma_{(\pp)}}P_0\ar[rr]^{\Gamma_{[\pp]}\otimes f_0}\ar[d]_{\wr}^{\alpha_P}
&&\Gamma_{[\pp]}\otimes_{\Gamma_{(\pp)}}X_0\ar[r]\ar[d]^{\alpha_X}&0\\
P'\ar[rr]^g&&P\ar[rr]^f&&X\ar[r]&0,
}\]
where the left and middle vertical maps are isomorphisms. Thus $\alpha_X$ and hence $\alpha$ is an isomorphism.

Clearly $(-)_0$ gives a functor $\gr\fd\Gamma_{[\pp]}\to\fd\Gamma_{(\pp)}$, and $\Gamma_{[\pp]}\otimes_{\Gamma_{(\pp)}}-$ gives a functor $\fd\Gamma_{(\pp)}\to\gr\fd\Gamma_{[\pp]}$
since $\Gamma_{[\pp],i}$ is a finitely generated right $\Gamma_{(\pp)}$-module by Lemma \ref{nice grading}(b).
Thus we have an equivalence $\gr\fd\Gamma_{[\pp]}\to\fd\Gamma_{(\pp)}$.
\end{proof}

We also need the following elementary observations.

\begin{lemma}\label{properties of (pp)}
Let $\pp\in\MaxProj R$ and $X\in\gr_{\pp}\Gamma$.
Assume that $X$ does not have non-zero finite dimensional submodules.
\begin{itemize}
\item[(a)] Any element in $R\backslash\pp$ is $X$-regular.
\item[(b)] For any $r\in R_\ell\backslash\pp$, the map $r \colon X_i\to X_{i+\ell}$ is bijective for any $i\gg0$.
\item[(c)] The natural map $X\to X_{[\pp]}$ induces an isomorphism $X_i\to X_{[\pp],i}$ for any $i\gg0$.
\end{itemize}
\end{lemma}

\begin{proof}
(a) Fix $r\in R\backslash\pp$ and assume $rx=0$ for some $x\in X$.
Then $\Gamma x$ is annihilated by $(r)+\ann_RX$, so it is supported only by maximal ideals of $R$.
Thus $\Gamma x$ is finite dimensional by Lemma~\ref{some basic}(b), and we have $x=0$ by our assumption.

(b) Consider an exact sequence
\[0\to K\to X\xrightarrow{r}X\to C\to0.\]
Since $K$ and $C$ are annihilated by $(r)+\ann_R X$, they are supported only by maximal ideals of $R$.
Thus $K$ and $C$ are finite dimensional by Lemma~\ref{some basic}(b), and we have the assertion.

(c) Pick $r \in R_\ell\backslash\pp$ as in (b), and let $a \in \mathbb{Z}$ such that the map $r \colon X_i\to X_{i+\ell}$ is bijective for all $i \geq a$.
We will show that the natural map $X_i\to X_{[\pp],i}$ is an isomorphism for any $i\ge a$.
Any element of $X_{[\pp],i}$ can be written as $x s^{-1}$ for some $x \in X_{i+j}$ and $s \in R_j \setminus \pp$ with $j\ge0$. Replacing $(x,s)$ by $(xs^{\ell-1},s^{\ell})$, we can assume that $j$ is a multiple of $\ell$.
By (a), the map $s \colon X_i \to X_{i+j}$ is injective. Since $i\ge a$ and $j$ is a multiple of $\ell$, we have $\dim X_i = \dim X_{i+j}$. Thus the map $s \colon X_i \to X_{i+j}$ is bijective.
Hence there is $y \in X_i$ such that $ys = x$, so $x s^{-1}=y\in X_i$.
\end{proof}

As a consequence, we have the following equivalence.

\begin{proposition}\label{new proof 2}
The functor $(-)_{[\pp]} \colon \gr\Gamma\to\gr\Gamma_{[\pp]}$ induces an equivalence $\qgr_{\pp}\Gamma\to\gr\fd\Gamma_{[\pp]}$.
\end{proposition}

\begin{proof}
Since all finite dimensional $\Gamma$-modules are annihilated by some homogeneous element in $R\backslash\pp$, they are sent to zero by the functor $(-)_{[\pp]}$.
Thus we have an induced functor $\qgr\Gamma\to\gr\Gamma_{[\pp]}$.

In the rest, we fix a homogeneous element $r\in R_{\ell}\backslash\pp$ for some $\ell>0$. 
We show $X_{[\pp]}\in\gr\fd\Gamma_{[\pp]}$ for any $X\in\gr_{\pp}\Gamma$ without non-zero finite dimensional submodules.
By Lemma \ref{properties of (pp)}(c), we have $\dim X_{[\pp],i}<\infty$ for $i\gg0$.
Since we have an isomorphism $r \colon X_{[\pp],i}\to X_{[\pp],i+\ell}$ for any $i\in\Z$, we have the assertion.
Thus we have a functor $(-)_{[\pp]} \colon \qgr_{\pp}\Gamma\to\gr\fd\Gamma_{[\pp]}$.

We consider the functor $(-)_{\ge0} \colon \gr\fd\Gamma_{[\pp]}\to\Gr\Gamma$ sending $Y=\bigoplus_{i\in\Z}Y_i$ to $Y_{\ge0}:=\bigoplus_{i\ge0}Y_i$.
Since $Y_{i+\ell}=rY_i$ holds for any $i\in\Z$, we have that $Y_{\ge0}$ is generated by $Y_0$, $Y_1,\cdots,Y_{\ell-1}$. Thus $Y_{\ge0}\in\gr\Gamma$.
Since any homogeneous element in $R\backslash\pp$ is $(Y_{\ge0})$-regular, any associate prime ideal of $Y_{\ge0}$ is contained in $\pp$.
On the other hand, since the Hilbert function of $Y_{\ge0}$ is bounded, the Krull dimension of $Y_{\ge0}$ is at most one (\cite[Theorem 4.1.3]{BH}).
Thus $\Supp_R (Y_{\ge0})\subset V(\pp)$, and we have $Y_{\ge0}\in\gr_{\pp}\Gamma$.
Consequently, we have a functor $(-)_{\ge0} \colon \gr\fd\Gamma_{[\pp]}\to\gr_{\pp}\Gamma$.
We will show that the functors $(-)_{\ge0}$ and $(-)_{[\pp]}$ induce mutually quasi-inverse equivalences between $\gr\fd\Gamma_{[\pp]}$ and $\qgr_{\pp} \Gamma$.

For any $X\in\qgr_{\pp}\Gamma$, the natural map $X\to X_{[\pp]}$ gives a morphism $\beta_X \colon X\iso X_{\ge0}\to (X_{[\pp]})_{\ge 0}$ in $\qgr_{\pp}\Gamma$.
This is an isomorphism in $\qgr_{\pp}\Gamma$ by Lemma \ref{properties of (pp)}(c).
Thus we have an isomorphism $\beta \colon {\rm id}_{\qgr_{\pp}\Gamma}\to ((-)_{[\pp]})_{\ge 0}$ of functors on $\qgr_{\pp}\Gamma$.

For any $X\in\gr\fd\Gamma_{[\pp]}$, we have a natural inclusion $\gamma_X \colon (X_{\ge0})_{[\pp]}\to X$, which gives a morphism $\gamma \colon ((-)_{\ge0})_{[\pp]}\to{\rm id}_{\gr\fd\Gamma_{[\pp]}}$ of functors on $\gr\fd\Gamma_{[\pp]}$.
For any homogeneous element $x\in X_i$, we have $x=(xr^j)r^{-j}$ where $xr^j\in X_{i+\ell j}$ and $r^j\in R_{\ell j}\backslash\pp$. Taking $j$ large enough, we have $xr^j\in X_{\ge0}$ and we have $x\in(X_{\ge0})_{[\pp]}$.
Thus $\gamma_X$ is an isomorphism, and $\gamma$ is an isomorphism of functors.
\end{proof}

We have now completed the proof of Theorem~\ref{qgr is fl}:

\begin{proof}[Proof of Theorem \ref{qgr is fl}]
We only have to compose the equivalences given in Propositions~\ref{new proof 1} and \ref{new proof 2}.
\end{proof}

Now we apply the above results to give a description of the category $\RR$.

\medskip
Again let $\Lambda$ be an $n$-representation tame algebra,
so the preprojective algebra $\Pi$ is a Noetherian $R$-algebra for a graded $K$-subalgebra $R$ of $\Pi$.
Now we apply the above results to $(\Gamma,R):=(\Pi,R)$, which satisfies the assumption in Theorem~\ref{qgr is fl}.

\begin{theorem} \label{theorem.R_decomp_tame}
Let $\Lambda$ be an $n$-representation tame algebra.
Assume that the preprojective algebra $\Pi$ is a Noetherian $R$-algebra for a graded $K$-subalgebra $R$ of $\Pi$. Then we have
\[\RR=\coprod_{{\mathfrak p}\in\MaxProj R} \fd\Pi_{(\pp)}.\]
\end{theorem}

\begin{proof}
We have
\begin{align*}
\RR&=\qgr_0\Pi&&\mbox{by Theorem~\ref{description of R},}\\
&=\coprod_{{\mathfrak p}\in\MaxProj R}\qgr_{\pp}\Pi &&\mbox{by Proposition~\ref{qgr decomposition},}\\
&=\coprod_{{\mathfrak p}\in\MaxProj R}\fd\Pi_{(\pp)}&&\mbox{by Theorem~\ref{qgr is fl}.}
\end{align*}
Thus the assertion follows.
\end{proof}

This is very similar to the well-known result in representation theory of tame hereditary algebras.

We end this section by giving another property of $n$-representation tame algebras,
which is analogous to the classical case $n=1$.

\begin{corollary}
Let $\Lambda$ be $n$-representation tame. Then there exists a positive integer $\ell$ such that $\tau_n^\ell(X)\iso X$ holds for any $X\in\RR$.
\end{corollary}

\begin{proof}
Without loss of generality, we can assume that $X$ is indecomposable and that $Y\in\qgr_{\pp}\Pi$ corresponds to $X$.
Since $R$ is a finitely generated $K$-algebra, there exists $\ell>0$ such that $R$ is a finitely generated $R_0[R_{\ell}]$-module,
where $R_0[R_{\ell}]$ is an $R_0$-subalgebra of $R$ generated by $R_{\ell}$.
Then for any $\pp\in\MaxProj R$, there exists a homogeneous element $r\in R_{\ell}\backslash\pp$.
The morphism $r \colon Y\to Y(\ell)$ is an isomorphism in $\qgr_0\Pi$ by Lemma~\ref{properties of (pp)}(b). This means $X\simeq\tau_n^{-\ell}(X)$ in $\RR$.
\end{proof}

\begin{example}\label{example of description of R}
\begin{itemize}
\item[(a)] Let $\Lambda$ be a path algebra of an extended-Dynkin quiver, which is $1$-representation tame as we observed in Example~\ref{example of tame}(a),
and $\Pi$ be the corresponding classical preprojective algebra.
Then our description $\RR=\qgr_0\Pi$ in Theorem~\ref{description of R} was already given by Lenzing \cite{L}, and the decomposition of $\RR$ in Theorem \ref{theorem.R_decomp_tame} is well-known.
\item[(b)] Let $\Lambda$ be a Beilinson algebra of dimension $n$.
This is $n$-representation tame as we observed in Example~\ref{example of tame}(b), 
and $\Pi$ is a skew group algebra $S*H$ of a polynomial algebra $S=K[x_0,\ldots,x_{n}]$ and a subgroup $H$ of $\SL_{n+1}(K)$
generated by the scalar matrix ${\rm diag}(\zeta,\cdots,\zeta)$ where $\zeta$ is an $(n+1)$-th primitive root of unity. Then we have $\RR=\qgr_0\Pi$ by Theorem~\ref{description of R}, which is easily shown to be the same as $\qgr_0S=\coh_0\mathbb{P}^n$.
\end{itemize}
\end{example}

\section{Derived and representation dimension of $n$-hereditary algebras}

The notion of representation dimension was introduced by Auslander
to measure how far an algebra being from representation finite.
Let us start with recalling the definition.

\begin{definition}[Auslander \cite{A2}]
The \emph{representation dimension} of $\Lambda$ is defined by
\[\repdim\Lambda:=\inf\{\gl\End_\Lambda(\Lambda\oplus M\oplus D\Lambda) \mid M\in\mod\Lambda\}.\]
\end{definition}

Auslander showed in \cite{A2} that for a finite dimensional algebra $\Lambda$ we have
\[ \repdim \Lambda \begin{cases} \leq 2 & \text{ if $\Lambda$ is representation finite,} \\ \geq 3 & \text{ if $\Lambda$ is representation infinite.} \end{cases} \]
We suspect that a similar statement holds for $n$-hereditary algebras. More precisely we conjecture the following:

\begin{conjecture} \label{conj.repdim}
Let $\Lambda$ be an $n$-hereditary algebra. Then
\[ \repdim \Lambda \begin{cases} \leq n+1 & \text{ if $\Lambda$ is $n$-representation finite,} \\ \geq n+2 & \text{ if $\Lambda$ is $n$-representation infinite.} \end{cases} \]
\end{conjecture}

By Auslander's result the conjecture holds for $n=1$. We will see in this section that it also holds for $n=2$. For $n \geq 2$ we show that the conjecture holds if $\Lambda$ is $n$-representation finite or $n$-representation tame.

We first observe that the statement for $n$-representation finite algebras follows immediately from earlier work:

\begin{proposition}\label{repdim of finite}
Let $\Lambda$ be an $n$-representation finite algebra. Then $\repdim\Lambda\le n+1$.
\end{proposition}

\begin{proof}
Let $M$ be an $n$-cluster tilting $\Lambda$-module. Clearly we have $\Lambda\oplus D\Lambda\in\add M$.
Moreover we have $\gl\End_\Lambda(M)\le n+1$ by \cite[Theorem 0.2]{I2}.
\end{proof}

Next we show that Conjecture~\ref{conj.repdim} holds for $n=2$.

\begin{theorem}\label{at least 4}
Let $\Lambda$ be an $n$-representation infinite algebra for $n\ge2$. Then $\repdim\Lambda\ge4$.
\end{theorem}

\begin{proof}
By Corollary~\ref{Cor.Ext-orth-family}, the subcategory $\mathscr{P}$ of $n$-preprojective modules in $\mod \Lambda$ containes infinitely many indecomposable objects,
and moreover satisfies $\Ext_{\Lambda}^i(\PP,\PP) = 0$ for any $i \in \{1, \ldots, n-1\}$, in particular $\Ext^1_\Lambda(\PP,\PP)=0$ since we assumed $n \geq 2$.

On the other hand, it has been shown in \cite[Theorem 5.5.1(1)]{I2} that for algebras of representation dimension at most $3$ there does not exist an infinite set $\{ X_i \}_{i \in I}$ of indecomposable modules such that $\Ext^1(X_i, X_j) = 0 \; \forall i,j\in I$.
Combining these two results we see that an $n$-representation infinite algebra has representation dimension at least four.
\end{proof}

In the rest of this section, we study the representation dimension of $n$-representation tame algebras. More precisely we show the following.

\begin{theorem}\label{main}
Let $\Lambda$ be an $n$-representation tame algebra over a perfect field $K$. Then $\repdim \Lambda \ge n+2$.
\end{theorem}

The dimension of a triangulated category, introduced in \cite{Ro}, and certain generalizations introduced in \cite{O} have proved to be useful for studying representation dimension.
Let us recall the definition.

\begin{definition}\cite{Ro,O}
Let $\TT$ be a triangulated category. For $T \in \TT$, we put
\begin{align*}
\langle T\rangle_1 & := \add\{ T[i] \mid i\in\Z\},\\
\langle T\rangle_{\ell+1} & := \add\{X\in\TT \mid \exists Y \to X \to Z\to Y[1]\ \text{ with }  Y\in\langle T\rangle_1,\ Z\in\langle T\rangle_{\ell} \}.
\end{align*}
The \emph{dimension} of $\TT$ is defined by
\[\dim \TT := \inf\{\ell \mid \exists T \in\TT\ \mbox{ with }\langle T\rangle_{\ell+1}=\TT\}.\]
More generally, for a subcategory $\mathscr{S} \subseteq \mathscr{T}$, the \emph{dimension} of $\mathscr{S}$ is defined by
\[ \dim_{\mathscr{T}} \mathscr{S} := \inf\{\ell \mid \exists T\in\TT\ \mbox{ with }\langle T\rangle_{\ell+1} \supseteq \mathscr{S}\}.\]
\end{definition}

Using this notion, we obtain the following refinements of our Conjecture~\ref{conj.repdim}.

\begin{conjecture} \label{conj.derdim}
Let $\Lambda$ be an $n$-representation infinite algebra. Then
\begin{enumerate}
\item The dimension of the category $\mathscr{R}$ of $n$-regular modules is $\dim_{\DDD^{\rm b}(\mod\Lambda)} \mathscr{R} = n$.
\item The dimension of the module category is $\dim_{\DDD^{\rm b}(\mod\Lambda)} \mod \Lambda = n$.
\item The dimension of the derived category is $\dim \DDD^{\rm b}(\mod\Lambda) = n$.
\end{enumerate}
\end{conjecture}

\begin{remark} \label{rem.hierarchy_conjectures}
\begin{itemize}
\item We have
\[ \dim_{\DDD^{\rm b}(\mod\Lambda)} \mathscr{R} \leq \dim_{\DDD^{\rm b}(\mod\Lambda)} \mod \Lambda \leq \dim \DDD^{\rm b}(\mod\Lambda) \leq \gl \Lambda = n, \]
where the last inequality is \cite[Corollary 3.6]{KK}. Thus we get the implication $(1) \Longrightarrow (2) \Longrightarrow (3)$ in Conjecture~\ref{conj.derdim}.
\item Clearly Conjecture~\ref{conj.derdim}(1) implies the existence of non-zero $n$-regular modules, that is a positive answer to Conjecture~\ref{quest.R_exists}.
\item By \cite{O} we have $\repdim \Lambda \geq \dim_{\DDD^{\rm b}(\mod\Lambda)} \mod \Lambda + 2$. Thus Conjecture~\ref{conj.derdim}(2) implies Conjecture~\ref{conj.repdim}.
\end{itemize}
\end{remark}

We will prove now that Conjecture~\ref{conj.derdim}(1) holds for $n$-representation tame algebras. Thus, by the above remark, all parts of Conjecture~\ref{conj.derdim} as well as Conjecture~\ref{conj.repdim} hold for this class of algebras.

\begin{theorem} \label{theorem.dim_regular}
Let $\Lambda$ be an $n$-representation tame algebra. Then $\dim_{\DDD^{\rm b}(\mod \Lambda)} \mathscr{R} = n$. In particular $\repdim\Lambda\ge n+2$.
\end{theorem}

By Remark~\ref{rem.hierarchy_conjectures} we already know that $\dim_{\DDD^{\rm b}(\mod \Lambda)} \mathscr{R} \leq n$, thus it remains to show that $\dim_{\DDD^{\rm b}(\mod \Lambda)} \mathscr{R} \geq n$.
We use Minamoto's derived equivalence $\DDD^{\rm b}(\mod \Lambda) \equi \DDD^{\rm b}(\qgr \Pi)$ (see Section~\ref{section: NCAG}),
which induces an equivalence $\mathscr{R}\equi\qgr_0\Pi$ by Theorem~\ref{description of R}.
Thus Theorem~\ref{theorem.dim_regular} follows from the following.

\begin{proposition} \label{prop.dim_sky}
Let $\Lambda$ be an $n$-representation tame algebra and $\Pi$ the preprojective algebra of $\Lambda$. Then $\dim_{\DDD^{\rm b}(\qgr \Pi)} \qgr_0\Pi\geq n$.
\end{proposition}

Our strategy of the proof is mostly parallel to \cite{KK}, the difference is that $\Pi$ is non-commutative here, so we need to be careful about the interplay of $\Pi$ and its center $Z$.

\begin{proof}
First observe that, since the center $Z$ of $\Pi$ is a finitely generated $K$-algebra algebra by Proposition~\ref{finitely generated},
we can use Noether normalization to find a graded polynomial subring $R$ of $Z$ such that $\Pi$ is a Noetherian $R$-algebra.

Let $T$ be any object in $\DDD^{\rm b}(\qgr \Pi)$. Restriction to the (commutative) scheme $\Proj R$ gives an object $T|_{\Proj R} \in \DDD^{\rm b}(\coh \Proj R)$.
Since $R$ is regular, it is not difficult to see (\cite[4.1]{KK}) that there is an open subset $U \subseteq \Proj R$ such that
\begin{equation}\label{free on U}
(T|_{\Proj R})_{(\pp)} \in \left< R_{(\pp)} \right>_1\ \mbox{ for any }\ \pp \in U.
\end{equation}
Take any $\pp\in U\cap\MaxProj R$ and non-zero $X \in \qgr_{\pp}\Pi$.
Then $X_{(\pp)}$ is a non-zero finite dimensional $\Pi_{(\pp)}$-module.
By the Auslander-Buchsbaum formula,
\[ \pd_{R_{(\pp)}} X_{(\pp)} = \dim R_{(\pp)} \geq n, \]
where the last inequality follows from Corollary~\ref{consequence of tame}(a) and Lemma~\ref{some basic}(a).

An easy application of the Ghost Lemma, see for instance \cite[2.4]{KK}, shows that this implies that $X_{(\pp)}|_{R_{(\pp)}} \not\in \left< R_{(\pp)} \right>_n$.
It follows from \eqref{free on U} that also
\[ X \not\in \left< T \right>_n \]
Otherwise we obtain a contradiction by localizing at $(\pp)$ and restricting to $R_{(\pp)}$.

Thus we have shown that 
\[ \dim_{\DDD^{\rm b}(\qgr \Pi)} \qgr_0\Pi \geq n. \qedhere \]
\end{proof}

\end{document}